\documentclass[10pt]{amsart}
\usepackage{amsthm,amsfonts,amsmath,amssymb,latexsym}
\usepackage{epsfig,graphics,color}
\usepackage[matrix,arrow,curve]{xy}
\usepackage{hyperref}
\usepackage{url}
\usepackage{mathrsfs}
\newtheorem{PARA}{}[section]
\newtheorem{theorem}[PARA]{Theorem}

\newtheorem{lemma}[PARA]{Lemma}
\newtheorem{proposition}[PARA]{Proposition}
\newtheorem{definition}[PARA]{Definition}

\theoremstyle{definition}
\newtheorem{remark}[PARA]{Remark}
\newtheorem{example}[PARA]{Example}
\numberwithin{equation}{section}

\newcommand{\para}{\begin{PARA}\rm}
\newcommand{\arap}{\end{PARA}\rm}
\newcommand{\dfn}{\begin{definition}\rm}
\newcommand{\nfd}{\end{definition}\rm}
\newcommand{\rmk}{\begin{remark}\rm}
\newcommand{\kmr}{\end{remark}\rm}
\newcommand{\xmpl}{\begin{example}\rm}
\newcommand{\lpmx}{\end{example}\rm}
\newcommand{\cA}{\mathcal{A}}

\newcommand{\cC}{\mathcal{C}}

\newcommand{\cF}{\mathcal{F}}

\newcommand{\cH}{\mathcal{H}}

\newcommand{\cJ}{\mathcal{J}}

\newcommand{\cL}{\mathcal{L}}
\newcommand{\cM}{\mathcal{M}}

\newcommand{\cN}{\mathcal{N}}

\newcommand{\cP}{\mathcal{P}}

\newcommand{\cU}{\mathcal{U}}

\newcommand{\cW}{\mathcal{W}}

\newcommand{\sH}{\mathscr{H}}

\newcommand{\og}{{\overline{\gamma}}}
\newcommand{\ug}{{\underline{\gamma}}}
\newcommand{\oev}{\overline{\mathrm{ev}}}
\newcommand{\uev}{\underline{\mathrm{ev}}}

\newcommand{\uL}{{\underline{L}}}

\newcommand{\omu}{{\overline{\mu}}}

\newcommand{\op}{{\overline{p}}}
\newcommand{\up}{{\underline{p}}}

\newcommand{\utau}{\underline{\tau}}
\newcommand{\oz}{{\overline{z}}}
\newcommand{\uz}{{\underline{z}}}

\newcommand{\C}{{\mathbb{C}}}

\newcommand{\N}{{\mathbb{N}}}
\renewcommand{\P}{{\mathbb{P}}}
\newcommand{\Q}{{\mathbb{Q}}}
\newcommand{\R}{{\mathbb{R}}}

\newcommand{\T}{{\mathbb{T}}}
\renewcommand{\u}{{\mathbf{u}}}

\newcommand{\Z}{{\mathbb{Z}}}
\newcommand{\ind}{\mathrm{ind}}

\newcommand{\ev}{\mathrm{ev}}

\newcommand{\reg}{{\mathrm{reg}}}
\newcommand{\eps}{{\varepsilon}}
\newcommand{\om}{{\omega}}

\newcommand{\tf}{{\widetilde{f}}}

\newcommand{\tvarphi}{{\widetilde{\varphi}}}
\newcommand{\tPhi}{{\widetilde{\Phi}}}

\newcommand{\tH}{{\widetilde{H}}}

\newcommand{\tgamma}{{\widetilde{\gamma}}}

\newcommand{\tPi}{{\widetilde{\Pi}}}

\def\NABLA#1{{\mathop{\nabla\kern-.5ex\lower1ex\hbox{$#1$}}}}
\def\Nabla#1{\nabla\kern-.5ex{}_{#1}}
\def\Tabla#1{\Tilde\nabla\kern-.5ex{}_{#1}}
\renewcommand{\Tilde}{\widetilde}

\newcommand{\p}{{\partial}}

\begin{document}

\title[Symplectic- and linearized contact homology]{$S^1$-equivariant symplectic homology and linearized contact homology}
\author{Fr\'ed\'eric Bourgeois}
\address{Laboratoire de Math\'ematiques d'Orsay, UMR 8628, Universit\'e Paris-Sud \& CNRS, Orsay, France.}
\email{Frederic.Bourgeois@math.u-psud.fr}
\author{Alexandru Oancea}
\address{Sorbonne Universit\'es, UPMC Univ Paris 06, UMR 7586, Institut de Math\'ematiques de Jussieu-Paris Rive Gauche, Case 247, 4 place Jussieu, F-75005, Paris, France.}
\email{alexandru.oancea@imj-prg.fr}
\date{June 9, 2014}


\begin{abstract}
We present three equivalent definitions of $S^1$-equivariant symplectic homology. 
We show that, using rational coefficients, the positive part of $S^1$-equivariant symplectic homology is isomorphic to linearized contact homology, when the latter is defined. We present several computations and applications, and introduce a rigorously defined substitute for cylindrical/linearized contact homology based on an $S^1$-equivariant construction. 
\end{abstract}

\maketitle

\tableofcontents


\section{Introduction}

This paper is concerned with compact symplectic manifolds $(W,\omega)$ with contact type boundary $(M,\xi)$, meaning that $\omega$ is exact in a neighborhood of the boundary and admits a primitive that restricts on $M$ to a contact form for $\xi$ which induces the boundary orientation. We assume throughout the paper that $W$ is symplectically atoroidal, meaning that $\int f^*\omega =0$ for any smooth map $f:T^2\to W$. An important class of examples is that of Liouville domains~\cite{Eliashberg-Gromov,Seidel07,Cieliebak-Eliashberg-book}, for which the primitive of $\omega$ is globally defined. This contains the important subclass of Weinstein domains. To improve readability we  confine the main exposition to Liouville domains, and indicate along the way what changes are necessary in the general case. 

Denote $S^1=\R/\Z$. The goal of the present paper is to explain the relationship between the following symplectic invariants that are associated to $(W,\omega)$: symplectic homology $SH_*(W)$ and its versions $SH_*^\pm(W)$~\cite{Viterbo99}, linearized contact homology $CH_*^{lin}(M)$~\cite{EGH,BOcont,Bourgeois-Ekholm-Eliashberg-1}, $S^1$-equivariant symplectic homology $SH_*^{S^1}(W)$ and its versions $SH_*^{\pm,S^1}(W)$~\cite{Viterbo99,Seidel07}. The symplectic homology groups $SH_*$ and $SH_*^{S^1}$ are invariants of the symplectic completion $(\hat W,\hat\omega) = (W,\omega) \cup \left([1,\infty[\times M,d(r\alpha)\right)$ (see~\cite{Viterbo99,Seidel07} for $SH_*$, the $S^1$-equivariant case being similar), meaning that they do not depend on the choice of Liouville vector field and almost complex structure. In contrast, the linearized contact homology groups $CH_*^{lin}$ depend a priori on the choice of contact form on $M$ and of almost complex structure on $\hat W$. Invariance with respect to these choices is one of the expected applications of the polyfold theory of Hofer, Wysocki and Zehnder~\cite{Hofer-Polyfolds-survey,HWZ-polyfolds-I,HWZ-polyfolds-II,HWZ-polyfolds-III}. 
In this paper we do not use invariance and work for a fixed choice of contact form $\alpha$ and almost complex structure $J$ so that the groups $CH_*^{lin}(M,\alpha,J)$ are well-defined. For readability we use the shorthand notations $SH_*(W)$, $SH_*^{S^1}(W)$, $CH_*^{lin}(M)$. As a matter of fact, our Isomorphism Theorem~\ref{thm:isomorphism} implies invariance of linearized contact homology for the restricted class of Liouville domains that we consider. Moreover, the same theorem provides $SH_*^{+,S^1}(W)$ as an alternative to linearized contact homology which is a well-defined invariant of Liouville domains (see also the discussion at the beginning of~\S\ref{sec:CHlin}).  

We refer to~\cite{BOcont} for detailed definitions of symplectic homology and linearized contact homology, while the $S^1$-equivariant version of symplectic homology will be the main focus of this paper. The original definition is due to Viterbo~\cite[\S5]{Viterbo99} (see also~\cite{BOGysin}) and an alternative one was sketched by Seidel~\cite[\S8b]{Seidel07}. We will give details for each of them in~\S\ref{sec:S1}, prove that they are equivalent, and further refine Seidel's definition. The $S^1$-action that is referred to is the action by reparametrization on the space $\cL\hat W$ of smooth maps $\gamma:S^1\to \hat W$, given by $(\tau \cdot \gamma)(\cdot) := \gamma(\cdot -\tau)$, $\tau \in S^1$.

The above homology groups carry a $\Z$-grading if one uses coefficients in the \emph{Novikov ring\, }Ê$\Lambda_\om$ consisting of formal linear
combinations $\lambda:= \sum_{A\in H_2(W;\Z)} \lambda_A e^A$,
$\lambda_A \in \Z$ such that $\forall c>0$, $\#\{A\, : \, \lambda_A \neq
   0 \mbox{ and\: }Ê\omega(A) \le c\}<\infty$. The $\Z$-grading continues to exist if one uses $\Z$-coefficients provided that $c_1(W)$ is torsion. Again, to improve readability we restrict the main exposition to this case, and indicate along the way what changes are necessary -- if any -- in the general case.
     
Let $2n=\dim_\R W$. Our grading convention for symplectic homology is such that $SH_*^-(W)\simeq H_{*+n}(W,M)$. As an example, we have $SH_*(T^*N)\simeq H_*(\cL N)$ for a closed orientable spin manifold $N$~\cite{Viterbo-cotangent,SW,AS,AS-corrigendum}. Our grading convention for linearized contact homology uses the transverse Conley-Zehnder index, without the $n-3$ shift that is common in symplectic field theory~\cite{EGH}. As an example, we have $CH_*^{lin}(\P^+T^*N)=H_*(\cL N/S^1,N)$~\cite{Cieliebak-Latschev}, where $\P^+T^*N$ is the space of oriented contact elements of a closed oriented spin manifold $N$. In the case of the boundary $(S^{2n-1},\xi)$ of the unit ball in $\C^n$, we have $CH_*^{lin}(S^{2n-1})\simeq \Q$ if $*=n+1+2k$, $k\ge 0$, and $CH_*^{lin}(S^{2n-1})=0$ otherwise.

These homology groups count in various ways closed Reeb orbits on $M$, possibly combined with critical points of a Morse function on $W$. The differential relates loops that have the same free homotopy class in $W$, so that all the above homology groups split as direct sums indexed by $\pi_0\cL W\simeq \pi_1(W)/conj$. 
Most of our statements hold provided one specializes to some particular free homotopy class. 

The various versions of ($S^1$-equivariant) symplectic homology groups fit into \emph{tautological exact triangles} 
\begin{equation} \label{eq:tautologicaltriangle}
\xymatrix
@C=20pt
{
H_{*+n}^{[S^1]}(W,M) \ar[rr] & & 
SH_*^{[S^1]}(W) \ar[dl] \\ & SH_*^{+[,S^1]}(W) \ar[ul]^{[-1]}  
}
\end{equation}
Here the argument $[S^1]$ is optional, and the circle acts on the pair
$(W,M)$ trivially, so that
$H_*^{S^1}(W,M)=H_*(W,M)\otimes H_*(BS^1)$. Heuristically, the \emph{positive symplectic homology groups} $SH_*^{+[,S^1]}(W)$ are given by quotient
complexes generated by closed Reeb orbits on $M$ (these have positive action bounded away from zero), whereas $SH_*^{[S^1]}(W)$ captures the interaction between the Reeb dynamics on $M$ and the differential topology of $W$, represented by critical points of a Morse function (these have action close to zero).

\begin{remark}[comparison of notation with~\cite{BOcont}] \label{rmk:notation} 
We use the following notational convention. We do not specify in our notation the free homotopy class of loops. Through the whole of Section~\ref{sec:S1} the notation $SH_*^{[S^1]}$ and $SH_*^{+,[S^1]}$ stands for the invariant built on all free homotopy classes simultaneously. Thus, what we denote in section~\ref{sec:S1} by $SH_*^+$ coincides with $SH_*^+\oplus \bigoplus_{c\neq 0} SH_*^c$ in the notation of~\cite{BOcont}. 
Through the whole of Section~\ref{sec:CHlin} we need to restrict to a collection $\mathcal{C}$ of free homotopy classes for which the transversality assumptions of~\S\ref{sec:CH-Ham} are satisfied. Thus, what we denote in this paper by $CH_*^{lin}$ coincides with $\bigoplus_{c\in\mathcal{C}} HC_{*+n-3}^{i^{-1}(c)}$ in the notation of~\cite{BOcont}. Implicitly, in all statements relating symplectic homology and linearized contact homology, we understand that we restrict to such a collection $\mathcal{C}$ of free homotopy classes in the notation $SH_*^{[S^1]}$ and $SH_*^{+,[S^1]}$ as well. 
\end{remark}

From the perspective of the above remark, it is perhaps useful to remark that the exact triangle~\eqref{eq:tautologicaltriangle} only carries information related to contractible orbits: the maps having $H_{*+n}^{[S^1]}(W,M)$ as source or target land in, respectively are defined on, the summand of the symplectic homology groups determined by contractible orbits.

In order to better situate the arguments of this paper, let us consider the example of $S^1$-spaces, i.e. topological spaces endowed with a topological $S^1$-action. The $S^1$-equivariant \emph{singular} homology groups are a homological functor $H_*^{S^1}(-)$ defined on the category of $S^1$-spaces. This functor is characterized by the equality $H_*^{S^1}(X)=H_*(X/S^1)$ if the action is free, and the fact that an equivariant homotopy equivalence induces an isomorphism on $H_*^{S^1}$~\cite{Basu}. Indeed, it follows from these two properties that $H_*^{S^1}(X)=H_*(X \times_{S^1} ES^1)$, where $ES^1$ is a contractible space endowed with a free action. Such a space is unique up to equivariant homotopy equivalence, and an explicit model is $ES^1=S^\infty=\lim_{N}S^{2N+1}$. The space $X_{Borel}:=X\times_{S^1} ES^1$ is called \emph{the Borel construction}. In case $X=*$ is a point this is $BS^1:=ES^1/S^1=\C P^\infty$, \emph{the classifying space for principal $S^1$-bundles}. 

The space $X_{Borel}$ plays a role in the following three geometric pictures. 
\begin{itemize}
\item $X_{Borel}$ is the base of the principal $S^1$-bundle $X\times ES^1\stackrel{pr} \to X\times_{S^1} ES^1$.
\item $X_{Borel}$ is the total space of the locally trivial fibration $X_{Borel} \stackrel{pr_2}\to \C P^\infty$ with fiber $X$, induced by the second projection $X\times ES^1\stackrel{pr_2} \to ES^1$. 
\item the first projection $X\times ES^1\stackrel{pr_1} \to X$ induces a map $X_{Borel}\to X/S^1$. This is in general \emph{not} a locally trivial fibration: the fiber at an orbit $[x]\in X/S^1$ is the classifying space $B Stab(x)$ of the isotropy group at $x$. 
\end{itemize}
$$
\xymatrix
@C=15pt
{ S^1 \ar[r] & X\times ES^1 \ar[d]^{pr} & & X\ar[r] & X_{Borel} \ar[d]^{pr_2} && BStab(x)\ar[r] & X_{Borel} \ar[d]^{pr_1}  
\\
& X_{Borel}Ê& & & \C P^\infty & & & [x]\in X/S^1 \qquad
}
$$

Each of the above geometric pictures sheds its own light on $X_{Borel}$. The first point of view gives rise to the Gysin exact triangle in Theorem~\ref{thm:Gysin-spec}(i) and motivates Viterbo's definition of $S^1$-equivariant symplectic homology (see~\cite{Viterbo99,BOGysin} and~\S\ref{sec:S1-Borel}), in which Morse theory on $X_{Borel}$ is seen as Morse theory modulo $S^1$ on $X\times ES^1$, where the $S^1$-action is free. The second point of view gives rise to the spectral sequence in Theorem~\ref{thm:Gysin-spec}(ii) and motivates the more algebraic approach of Seidel and Smith (see~\cite{Seidel07,Seidel-Smith-localization} and~\S\ref{sec:S1-families}-\ref{sec:simplifying}), in which Morse theory on $X_{Borel}$ is seen as Morse theory \emph{\`a la} {Hutchings}~\cite{Hutchings03} for a family of Morse functions on $X$, parametrized by $\C P^\infty$. The third point of view leads to the isomorphism in Theorem~\ref{thm:isomorphism}, and its topological counterpart is the following. 

\begin{lemma} \label{lem:Leray}
Assume $Stab(x)$ is finite at every point $x\in X$. Then the canonical projection $X_{Borel}\to X/S^1$ induces an isomorphism in homology with $\Q$-coefficients. 
\end{lemma}

This Lemma has an algebraic counterpart in cyclic homology as the isomorphism over $\Q$ between the homology of the cyclic bicomplex and the homology of Connes' complex~\cite[Theorem~2.1.5]{Loday}.

The relationship between non-equivariant and $S^1$-equivariant symplectic homology is captured by the following theorem. Item~(ii) was stated in~\cite{Viterbo99} and item~(i) was proved by a different method in~\cite{BOGysin}.

\begin{theorem} \label{thm:Gysin-spec}
\renewcommand{\theenumi}{\roman{enumi}}
\begin{enumerate}
\item For $\dagger=\emptyset,-,+$ there are Gysin exact triangles 
\begin{equation} \label{eq:Gysintriangle}
\xymatrix
@C=20pt
{
SH_*^\dagger(W) \ar[rr] & & 
SH_*^{\dagger,S^1}(W) \ar[dl] \\ & SH_{*-2}^{\dagger,S^1}(W) \ar[ul]^{[+1]}  
}
\end{equation}
These exact triangles are compatible with the tautological ones in~\eqref{eq:tautologicaltriangle}.
\item For $\dagger=\emptyset,-,+$ there are spectral sequences $E^{\dagger,r}_{p,q}(W)$, $r\ge 2$ converging to $SH_*^{\dagger,S^1}$ with second page given by 
$$
E^{\dagger,2}_{p,q}\simeq H_p(BS^1)\otimes SH_q^\dagger(W).
$$
These are compatible with the tautological exact triangles in~\eqref{eq:tautologicaltriangle}. 
\end{enumerate}
\end{theorem}

Up to the change of notation described in Remark~\ref{rmk:notation}, the main result of~\cite{BOcont} is an exact triangle of the form 
\begin{equation} \label{eq:Gysin-contact}
\xymatrix
@C=20pt
{
SH_*^+(W) \ar[rr] & & 
CH_*^{lin}(M)  \ar[dl] \\ & CH_{*-2}^{lin}(M) \ar[ul]^{[+1]}  
}
\end{equation}

In order to make the connection with the Gysin exact triangle of~\eqref{eq:Gysintriangle} we shall upgrade $CH_*^{lin}(M)$ to a \emph{filled linearized contact homology group} $\overline{CH}_*^{lin}(M)$, so that the two are related by a \emph{tautological exact triangle}
\begin{equation} \label{eq:taut-triangle-contact}
\xymatrix
@C=20pt
{
H_{*+n}^{S^1}(W,M)\ar[rr] & & 
\overline{CH}_*^{lin}(M)  \ar[dl] \\ & CH_*^{lin}(M) \ar[ul]^{[-1]}  
}
\end{equation}

The paper~\cite{BOcont} also contains a definition of a non-equivariant version of linearized contact homology, that we denote here $NCH_*^{lin}(M)$. We shall upgrade this to a \emph{filled non-equivariant linearized contact homology group} $\overline{NCH}_*^{lin}(M)$, so that the two are related by a tautological exact triangle
\begin{equation} \label{eq:taut-triangle-ncontact}
\xymatrix
@C=20pt
{
H_{*+n}(W,M)\ar[rr] & & 
\overline{NCH}_*^{lin}(M)  \ar[dl] \\ & NCH_*^{lin}(M) \ar[ul]^{[-1]}  
}
\end{equation}
Similarly to~\eqref{eq:tautologicaltriangle}, the above two exact triangles only carry information related to closed Reeb orbits in $M$ which are contractible in $W$.

The definition of the groups $\overline{CH}_*^{lin}$, $NCH_*^{lin}$, $\overline{NCH}_*^{lin}$ is subject to similar caveats as in the case of linearized contact homology: they are defined in very restrictive cases, they depend a priori on the choice of auxiliary data (contact form, almost complex structure, Morse functions), and invariance is expected to be a consequence of polyfold theory. However, as for $CH_*^{lin}$, we do not use invariance in this paper and work with specific choices of auxiliary data. 

It is a direct consequence of the main construction in~\cite{BOcont} that these homology groups fit into Gysin exact triangles 
\begin{equation} \label{eq:Gysin-triangle-ncontact}
\xymatrix
@C=20pt
{
NCH_*^{lin}(M) \ar[rr] & & 
CH_*^{lin}(M)  \ar[dl] \\ & CH_{*-2}^{lin}(M) \ar[ul]^{[+1]}  
}
\end{equation}
and
\begin{equation} \label{eq:Gysin-triangle-ncontact-filled}
\xymatrix
@C=20pt
{
\overline{NCH}_*^{lin}(M) \ar[rr] & & 
\overline{CH}_*^{lin}(M)  \ar[dl] \\ & \overline{CH}_{*-2}^{lin}(M) \ar[ul]^{[+1]}  
}
\end{equation}

\begin{theorem} \label{thm:isomorphism}
Assume the coefficient ring contains $\Q$. Then there are natural isomorphisms between the Gysin triangle~\eqref{eq:Gysintriangle} for $SH_*^+$ and the exact triangle~\eqref{eq:Gysin-triangle-ncontact}, as well as between the Gysin triangle~\eqref{eq:Gysintriangle} for $SH_*$ and the exact triangle~\eqref{eq:Gysin-triangle-ncontact-filled}. 

These isomorphisms are compatible with the tautological triangles~\eqref{eq:tautologicaltriangle} and~\eqref{eq:taut-triangle-ncontact}.
\end{theorem}

The isomorphisms in Theorem~\ref{thm:isomorphism} are functorial with respect to exact inclusions of Liouville domains. This topic is taken up in~\cite{Cieliebak-Oancea}.

There are many other general questions about $S^1$-equivariant symplectic homology that we do not touch upon in this paper. One of them is \emph{functoriality} in the sense of Viterbo, which is treated for $SH_*^{S^1}$ in~\cite{Viterbo99}, and for $SH_*^{+,S^1}$ in~\cite{Gutt-thesis}, see also~\cite{Cieliebak-Oancea}. The transfer morphism described in~\S\ref{sec:rigorousCHcyl} plays a key role in the proof of invariance. Another question is that of the \emph{operations} that it carries. An overview is given by Seidel in~\cite{Seidel07}, and we explained in~\cite[Introduction]{BOGysin} a way to recover some of them from the Gysin sequence and from the product on symplectic homology, inspired by the paper of Chas and Sullivan~\cite{CS} on string topology operations. Yet another question is that of other versions of $S^1$-equivariant theories, analogous to \emph{negative} and \emph{periodic} cyclic homology theories. These are mentioned in Seidel~\cite{Seidel07} but their properties still remain to be investigated. Another discussion to be developed is that of the relationship with \emph{cyclic homology} of the Fukaya category for Lefschetz fibrations~\cite{Seidel-Hochschild}, also in relation with the isomorphism between symplectic homology and Hochschild homology proved by Ganatra and Maydanskiy~\cite{Bourgeois-Ekholm-Eliashberg-1}. Yet another remark is that the techniques of this paper can be adapted in order to define \emph{$S^1$-equivariant Rabinowitz Floer homology} (see~\cite{Cieliebak-Frauenfelder} for the original construction). Formal statements like the existence of the Gysin sequence and of the spectral sequence will hold in such a setup too. Such a construction would be relevant for the circle of ideas discussed in~\cite{Frauenfelder-Schlenk-vanKoert}, related to displaceability and the mean Euler characteristic.


The plan of the paper is the following. Section~\S\ref{sec:S1} deals with various versions of $S^1$-equivariant symplectic homology and the proof of Theorem~\ref{thm:Gysin-spec} appears at the end of~\S\ref{sec:algebraicS1}. Section~\S\ref{sec:CHlin} deals with the relationship with linearized contact homology using $\Q$-coefficients. The proof of Lemma~\ref{lem:Leray} is given at the beginning of the section and the proof of Theorem~\ref{thm:isomorphism} for positive $S^1$-equivariant symplectic homology is given in~\S\ref{sec:S1-CH}. Filled linearized contact homology is defined in~\S\ref{sec:CHlin-defi}. We do not spell out the proof of Theorem~\ref{thm:isomorphism} for filled linearized contact homology since it combines in a straightforward way the definition with the argument given for the positive case. In~\S\ref{sec:applications} we present some algebraic applications, and we also introduce a rigorously defined substitute for cylindrical contact homology using an $S^1$-equivariant point of view. We illustrate the versatility of this new invariant by presenting the subcritical surgery exact triangle from~\cite{Cieliebak-Oancea}. We also give an $S^1$-equivariant viewpoint on previous computations of Mei-Lin Yau and Cieliebak-Latschev for linearized contact homology.

\noindent {\bf Acknowledgements.} \S\ref{sec:S1} recasts and extends our previous joint work~\cite{BOGysin}. We are grateful to an anonymous referee for having put us on the right track by suggesting to investigate a more algebraic definition of $S^1$-equivariant Floer homology. This paper also owes a lot to discussions with Kai Cieliebak and Paul Seidel. We have benefited from conversations with many other mathematicians, and in particular Mohammed Abouzaid, Peter Albers, Yasha Eliashberg, Joel Fish, Nancy Hingston, Janko Latschev, Jean-Louis Loday, Jo Nelson, Dietmar Salamon, Claude Viterbo. The second author is particularly grateful to Ga\"el Collinet and Pierre Guillot for an illuminating discussion on Lemma~\ref{lem:Leray} at an early stage of this work.

F.B.: Partially supported by ERC Starting Grant StG-239781-ContactMath. A.O.: Partially supported by ERC Starting Grant StG-259118-Stein and NSF Grant DMS-0635607. Any opinions, findings and conclusions or recommendations expressed in this material are those of the author(s) and do not necessarily reflect the views of the National Science Foundation. A.O. acknowledges the hospitality of the Institute for Advanced Study, Princeton for the academic year 2011-2012 during which this paper was written. The present work is part of the authorsÕ activities within CAST, a Research Network Program of the European Science Foundation.

\section{$S^1$-equivariant symplectic homology} \label{sec:S1}

Hamiltonian Floer homology is sometimes heuristically referred to as the homology of the free loop space $\cL\hat W$ in semi-infinite dimensions. This was the object of extensive study ever since its discovery by Floer~\cite{Floer-SympFixedPoints}. Although it was clear for a long time that there should exist an $S^1$-equivariant version of Hamiltonian Floer homology~\cite{FHS} with reference to the canonical $S^1$-action, the first rigorous definition was given much later by Viterbo~\cite{Viterbo99}, and an alternative definition was sketched even later by Seidel~\cite{Seidel07}. Givental had earlier used a formal definition of $S^1$-equivariant Floer homology in his work on Mirror Symmetry~\cite{Givental}. 

In this section we give three equivalent definitions of $S^1$-equivariant symplectic homology. The first one is  Viterbo's~\cite[\S5]{Viterbo99} and we recall for the reader's convenience the discussion in~\cite{BOGysin,BOtransv,BOindex}. This mimicks the Borel construction at the level of the Hamiltonian action functional. The second one is inspired by Seidel's~\cite[\S8b]{Seidel07}, based on an idea of Hutchings~\cite{Hutchings03}, see also Seidel and Smith~\cite[\S2.1 and \S3.2]{Seidel-Smith-localization}. This recasts the Borel construction using the fibration structure $X\hookrightarrow X_{Borel}\to \C P^\infty$ and its main advantage is that the definition of the Floer differential uses standard continuation maps, for which no new analysis is needed (compare~\cite{BOtransv}). The third definition is a refinement of the second, and uses a canonical $S^1$-invariant extension of Hamiltonians on $\hat W$, based on the description of $S^{2N+1}$ as the join of $N+1$ copies of $S^1$.

Although we place ourselves within the setup of open manifolds, all the constructions work for closed symplectic manifolds, and the symplectic atoroidality assumption can be removed, provided of course one chooses a setup (monotone, semi-positive etc.) for which the usual, non-equivariant version of Floer homology is defined. 

Unless otherwise specified $(W,\omega)$ denotes in this section a Liouville domain with torsion first Chern class. 
The primitive of $\omega$ is denoted $\lambda$, its restriction to $M$ is denoted $\alpha$, and the primitive of $\hat\omega$ on $\hat W$ is denoted $\hat\lambda$.

\subsection{Borel construction}~\cite{Viterbo99,BOtransv}
\label{sec:S1-Borel}

Let $\cH$ denote the space of $1$-periodic Hamiltonians $H:S^1\times \hat W\to \R$ such that $H<0$ on $W$ and $H$ is of the form $H(\theta, r, x)=ar+b$ on $[R,\infty[\times M$, $R\gg 1$ with $a,b\in\R$ and $a>0$ not equal to the period of a closed Reeb orbit on $M$. We refer to $a$ as \emph{the slope of $H$}. 

Denote by $R_\alpha$ the Reeb vector field of the contact form $\alpha$ on $M$, defined by $\iota_{R_\alpha}\omega|_M=0$, $\alpha(R_\alpha)=1$. Let $\cJ$ be the space of $1$-periodic time-dependent almost complex structures on $\hat W$ which are compatible with $\hat \omega$ and which, on $[R,\infty[\times M$, $R\gg 1$, are independent of $\theta$, preserve $\xi$, and satisfy $J(r\frac\p {\p r})=R_\alpha$. 

Viterbo's idea to define $S^1$-equivariant symplectic homology~\cite[\S5]{Viterbo99} is to mimick Morse theory for the Borel construction at the level of the action functional: this is seen as Floer theory \emph{modulo $S^1$} for a functional defined on $\cL\hat W\times ES^1$ which is invariant with respect to the diagonal $S^1$-action. In practice, one approximates $ES^1$ by $S^{2N+1}$ and lets $N\to\infty$. 

The easiest way to provide such $S^1$-invariant action functionals is using $S^1$-invariant Hamiltonian families $H:S^1 \times \hat W \times S^{2N+1} \to \R$. Denote $H_z:=H(\cdot,\cdot,z)$. The $S^1$-invariance condition means that $H_{\tau z}(\theta+\tau,\cdot)=H_z(\theta,\cdot)$ for all $\theta,\tau\in S^1$, $z\in S^{2N+1}$. The corresponding family of
action functionals
$$
\cA : \cL \hat W\times S^{2N+1} \to \R,
$$
$$
\cA(\gamma,z) := \cA_z(\gamma) := -\int_{S^1} \gamma^*
\hat\lambda - \int_{S^1} H_z(\theta,\gamma(\theta)) \, d\theta
$$
is then $S^1$-invariant with respect to the diagonal $S^1$-action on $\cL\hat W\times S^{2N+1}$.

A significant particular case is that of families given by a single autonomous Hamiltonian $H(\theta,x,z)=H(x)$. 

We have studied such parametrized Hamiltonian action functionals in~\cite{BOtransv,BOindex}. The differential of $\cA$ is given by
\begin{equation*} 
d\cA(\gamma,z) \cdot (\zeta,\ell)=
\int_{S^1}\hat \om(\dot\gamma(\theta)-X_{H_z}(\gamma(\theta)),\zeta(\theta))
d\theta
-
\int_{S^1} \frac {\p H_z} {\p z} (\theta,\gamma(\theta)) d\theta
\cdot \ell
\end{equation*}
and therefore $(\gamma,z)$ is a critical point of $\cA$ if and
only if
\begin{equation} \label{eq:periodicpar}
 \gamma\in\cP(H_z) \quad \mbox{and} \quad
 \int_{S^1} \frac {\p H_z} {\p z} (\theta,\gamma(\theta))\,
d\theta =0.
\end{equation}
We denote by $\cP(H)$ the set of critical points of $\cA$. 
This set is $S^1$-invariant. Given
$p\in \cP(H)$ we denote $S_p$ the \emph{$S^1$-orbit of $p$}, so that $S_p=S_{\tau \cdot p}$, $\tau\in S^1$. 

An $S^1$-invariant Hamiltonian family $H:S^1\times \hat W\times S^{2N+1}\to \R$ is called \emph{admissible} if $H_z\in\cH$ for all $z$, with constant slope not depending on $z$. We denote the space of such families $H$ by $\cH^{S^1}_N$.

A family of almost complex structures
$J=(J_z^\theta)$, $\theta\in S^1$, $z\in S^{2N+1}$ is
called \emph{$S^1$-invariant} if it satisfies the condition
$J_{\tau z}^{\theta+\tau}=J_z^\theta$ for all $\tau\in S^1$ and all $\theta,z$,  
and is called \emph{admissible} if $J_z\in\cJ$ for all $z$. 
Such a $J$ induces an $S^{2N+1}$-family of $L^2$-metrics on
$\cL\hat W$ defined by
$$
\langle \zeta,\eta\rangle_z := \int_{S^1}
\om(\zeta(\theta),J_z^\theta\eta(\theta)) d\theta, \quad \zeta,\eta\in
T_\gamma \cL\hat W=\Gamma(\gamma^*T\hat W).
$$
The $S^1$-invariance condition ensures that, when coupled with an
$S^1$-invariant metric $g$ on $S^{2N+1}$, this family gives
rise to an $S^1$-invariant metric on
$\cL\hat W \times S^{2N+1}$. We denote by $\cJ_N^{S^1}$
the set of pairs $(J,g)$ consisting of an $S^1$-invariant admissible
family of almost complex structures $J$ and of an
$S^1$-invariant Riemannian metric $g$.

Given
$H\in\cH^{S^1}_N$, $(J,g)\in\cJ^{S^1}_N$, and
$\op:=(\og,\oz),\up:=(\ug,\uz)\in \cP(H)$, we denote by
$$
\hat \cM(S_\op,S_\up;H,J,g)
$$
the \emph{space of $S^1$-equivariant Floer trajectories}, consisting of
pairs $(u,z)$ with
$$
u:\R\times S^1 \to \hat W, \qquad z:\R\to S^{2N+1},
$$
satisfying
\begin{eqnarray}
\label{eq:Floer1}
 \p_s u + J_{z(s)}^\theta (\p_\theta u - X_{H_{z(s)}}^\theta (u)) & = & 0, \\
\label{eq:Floer2}
 \dot z (s) - \int_{S^1} \vec \nabla_z
H(\theta,u(s,\theta),z(s)) d\theta & = & 0,
\end{eqnarray}
and
\begin{equation} \label{eq:asymptotic}
 \lim_{s\to -\infty} (u(s,\cdot),z(s)) \in S_\op, \quad
 \lim_{s\to +\infty} (u(s,\cdot),z(s)) \in S_\up.
\end{equation}

These are formally the equations of negative $L^2$-gradient lines for $\cA$. If $\op$ and $\up$ lie on different $S^1$-orbits, the group $\R$ acts on $\hat \cM(S_\op,S_\up;H,J,g)$
by reparametrizations in the $s$-variable. We denote by
$$
\cM(S_\op,S_\up;H,J,g) := \hat \cM(S_\op,S_\up;H,J,g)/\R
$$
the \emph{moduli space of $S^1$-equivariant Floer trajectories}.

The study of~(\ref{eq:Floer1}--\ref{eq:asymptotic}) makes essential use of the operators $D_{(u,z)}$ that linearize the system, with asymptotic expressions encoded in the asymptotic operators $D_p$, $p\in\cP(H)$. We refer to~\cite[\S2]{BOtransv} for precise formulas for these operators. 

An $S^1$-orbit of critical points $S_p\subset \cP(H)$ is called \emph{nondegenerate} if the Hessian $d^2\cA(\gamma,z)$ has a
  $1$-dimensional kernel $V_p$
  for some (and hence any) $(\gamma,z)\in S_p$. 
A generator of
$V_p$ is then given by the infinitesimal generator of the $S^1$-action.
We define the set $\cH^{S^1}_{N,\reg}\subset \cH^{S^1}_N$ to consist
of elements $H$ all of whose $S^1$-orbits are nondegenerate. We proved in~\cite[Proposition~5.1]{BOtransv} that
 $\cH^{S^1}_{N,\reg}$ is of second Baire category in
 $\cH^{S^1}_N$. Moreover, if $H\in \cH^{S^1}_{N,\reg}$, each
 $S^1$-orbit $S_p\subset \cL\hat W\times S^{2N+1}$ is
 isolated and there are only finitely many of them. 

Let $d>0$ be small enough (for a fixed $H\in\cH^{S^1}_{N,\reg}$, one
can take $d>0$ to be smaller than the minimal spectral gap of the asymptotic
operators $D_p$, $p\in\cP(H)$), and fix $1<p<\infty$. Given
$\op,\up\in \cP(H)$ and $(u,z)\in \hat
\cM(S_\op,S_\up;H,J,g)$, we define
\begin{eqnarray*}
  \cW^{1,p,d} & := & W^{1,p}(u^*T\widehat
  W;e^{d|s|}dsd\theta) \oplus W^{1,p}(z^*
  TS^{2N+1};e^{d|s|}ds)\oplus V_{\op}\oplus V_{\up}, \\
 \cL^{p,d} & := & L^p(u^*T\widehat
  W;e^{d|s|}dsd\theta) \oplus L^p(z^*
  TS^{2N+1};e^{d|s|}ds).
\end{eqnarray*}
Here we identify $V_{\op}$, $V_{\up}$ with the $1$-dimensional spaces
generated by the sections $\beta(s)(\dot\og,X_{\oz})$, respectively $\beta(-s)(\dot\ug,X_{\uz})$
 of $u^*T\widehat W\oplus z^*TS^{2N+1}$. For this identification, we denote by $X_{\oz}$, $X_{\uz}$ the values of the infinitesimal generator of the $S^1$-action on $S^{2N+1}$
 at the points $\oz$, respectively $\uz$, and choose a cut-off function $\beta:\R\to [0,1]$ which is equal to $1$ near $-\infty$, and vanishes near $+\infty$.

\begin{proposition}[{\cite[Proposition~4.4]{BOGysin}}] \label{prop:indexMB}
Assume $S_\op,S_\up\subset \cP(H)$ are nondegenerate. For any
$(u,z)\in \hat \cM(S_\op,S_\up;H,J,g)$ the operator
$$
D_{(u,z)}: \cW^{1,p,d} \to \cL^{p,d}
$$
is Fredholm of index
$$
\ind\, D_{(u,z)} = -\mu(\op) +\mu(\up) + 1.
$$
Here $\mu(\op)$, $\mu(\up)$ denote the parametrized Robbin-Salamon indices, as defined in~\cite{BOindex}.
\end{proposition}

We recall from~\cite{BOindex} that, in order to define $\mu(p)$, one considers the extended Hamiltonian $\tH:S^1\times\hat W\times T^*S^{2N+1}\to \R$ defined by $\tH(\theta,x,(z,z^\vee)):=H(\theta,x,z)$, with $z^\vee$ the dual cotangent coordinate. The point $p=(\gamma,z)$ belongs to $\cP(H)$ if and only if $\tilde p:=(\gamma,(z,z^\vee(\cdot)))$ is a periodic orbit of $\tH$, with $z^\vee(\theta)=z^\vee(0)-\int_0^\theta\frac{\p H}{\p z}(t,\gamma(t),z)\,dt$.  The index $\mu(p)$ is by definition the Robbin-Salamon index of $\tilde p$ with respect to a symplectic trivialization along $\tilde p$ induced by a choice of trivialization along a reference loop in each free homotopy class in $\hat W$ (see for example~\cite[\S2]{BOauto}). Our assumption that $c_1(W)$ is torsion ensures that we do not have to use Novikov coefficients.

A pair $(J,g)\in \cJ^{S^1}_N$ is called
\emph{regular for $H\in\cH^{S^1}_{N,\reg}$} if the operator $D_{(u,z)}$ is surjective
for any $\op,\up\in\cP(H)$ and any $(u,z)\in \hat
\cM(S_\op,S_\up;H,J,g)$. We denote the set of such regular pairs by
$\cJ^{S^1}_{N,\reg}(H)$.
We defined in~\cite[\S7]{BOtransv} two distinguished classes $\cH_*\cJ'\subset \cH\cJ'$ in $\cH_N^{S^1}\times\cJ_N^{S^1}$, and we proved in~\cite[Theorem~7.4]{BOtransv} that
there exists an open subset $\cH\cJ'_{\mathrm{reg}}\subset \cH\cJ'$
which is dense in a neighborhood of $\cH_*\cJ'\subset \cH\cJ'$, consisting of triples $(H,J,g)$ such that
$H\in\cH^{S^1}_{N,\mathrm{reg}}$ and $(J,g)\in\cJ^{S^1}_{N,\mathrm{reg}}(H)$.

Given $(H,J,g)\in \cH\cJ'_{\reg}$, the moduli space $\cM(S_\op,S_\up;H,J,g)$ is a smooth manifold of dimension $-\mu(\op) +\mu(\up)$. 
Since $\cA$ and $(J,g)$ are $S^1$-invariant,
this moduli space
carries a free action of $S^1$ induced by the diagonal
action on $\cL\hat W \times S^{2N+1}$, and we denote the quotient by
$$
\cM_{S^1}(S_\op,S_\up;H,J,g) := \cM(S_\op,S_\up;H,J,g)/S^1.
$$
This is a smooth manifold of dimension
$$
\dim\, \cM_{S^1}(S_\op,S_\up;H,J,g) = -\mu(\op) +\mu(\up)  -1.
$$

An important feature of these moduli spaces is that they admit a system
of coherent orientations in the sense of~\cite{FH-coherent}. The difference with respect to
the setup of Floer homology is that the asymptotes for the moduli spaces are not fixed, but can vary along circles $S_p$, $p=(\gamma,z)\in\cP(H)$. However, if one chooses the trivializations of $\gamma^*T\widehat W \oplus T_z S^{2N+1}$ so that they are invariant under the $S^1$-action,
then the analytical expression of the asymptotic operators $D_p$, $p\in\cP(H)$ only depends on $S_p$.
It then follows from the arguments in~\cite{FH-coherent} that the relevant spaces of Fredholm operators, with fixed nondegenerate asymptotics, are contractible, and hence the corresponding determinant line bundles are orientable.
The system of coherent orientations on the moduli spaces $ \cM_{S^1}(S_\op,S_\up;H,J,g)$ is obtained by pulling back a system of coherent orientations on these spaces of Fredholm operators, as in~\cite{FH-coherent}.

We define
the \emph{$S^1$-equivariant chain complex} $SC^{S^1,N}_*(H,J,g)$ as the
chain complex whose underlying $\Z$-module is
\begin{equation} \label{eq:SCS1}
SC^{S^1,N}_*(H):=SC^{S^1,N}_*(H,J,g):=
\bigoplus_{S_p\subset \cP(H)} \Z\langle
S_p\rangle.
\end{equation}
The grading is defined by $|S_p| := -\mu(p) +N$. The \emph{$S^1$-equivariant differential}
$\partial^{S^1}:SC^{S^1,N}_*(H)\to SC^{S^1,N}_{*-1}(H)$ is defined
by
$$
\partial^{S^1}(S_\op):=\sum_{\substack{
  S_\up\subset \cP(H) \\
|S_\op| - |S_\up|=1}}
\ \sum_{\scriptstyle [u]\in \cM_{S^1}(S_\op,S_\up;H,J,g)}
\epsilon([u])S_\up.
$$
The sign $\epsilon([u])$ is obtained by comparing the coherent
orientation of the moduli space
$\cM_{S^1}(S_\op,S_\up;H,J,g)$ with the orientation induced by the
infinitesimal generator of the $S^1$-action.

\begin{proposition} \label{prop:partialS1}
The map $\partial^{S^1}$ satisfies
$$
\partial^{S^1}\circ \partial^{S^1}=0.
$$
\end{proposition}

\begin{proof}[Sketch of proof.]
One way of proving this, once transversality is guaranteed using~\cite{BOtransv}, is by the usual compactness-gluing argument in Floer homology, which carries over \emph{mutatis mutandis} to the present $S^1$-invariant situation. 

A second proof has been given in~\cite[Proposition~4.5]{BOGysin} by identifying $\p^{S^1}$ with the differential on the first page of a spectral sequence that gives rise to the Gysin triangle~\eqref{eq:Gysintriangle}. This method of proof relies on the algebraic construction in~\cite{BOcont}.
\end{proof}

We define the \emph{$S^1$-equivariant Floer homology groups} by
$$
SH^{S^1,N}_*(H,J,g):=H_*(SC^{S^1,N}_*(H),\partial^{S^1}).
$$

A standard continuation argument proves that, given $H\in\cH^{S^1}_{N,\reg}$ and two elements 
$(J_1,g_1),(J_2,g_2)\in\cJ^{S^1}_{N,\reg}(H)$, there exists a
canonical isomorphism
$$
SH^{S^1,N}_*(H,J_1,g_1) \simeq SH^{S^1,N}_*(H,J_2,g_2).
$$
Given $H\in\cH^{S^1}_{N,\reg}$ let
$SH^{S^1,N}_*(H):=SH^{S^1,N}_*(H,J,g)$ for $(J,g)\in
\cJ^{S^1}_{N,\reg}(H)$. By a similar continuation argument the homology groups $SH_*^{S^1,N}(H)$ only depend on the value of the slope of $H$ at infinity. In analogy with the construction of symplectic
homology, we define
$$
SH^{S^1,N}_*(W):=\lim_{\stackrel \longrightarrow {H\in
  \cH^{S^1}_{N,\reg}}} SH^{S^1,N}_*(H).
$$
The \emph{$S^1$-equivariant symplectic homology groups of $W$} are
defined by
$$
SH^{S^1}_*(W):=\lim_{\stackrel \longrightarrow N} SH^{S^1,N}_*(W).
$$
The direct limit is taken with respect to the embeddings
$S^{2N+1}\hookrightarrow S^{2N+3}$, which induce maps
$SH^{S^1,N}_*(W)\to SH^{S^1,N+1}_*(W)$. (At this point it is essential that the grading differs from the parametrized Robbin-Salamon index by $N$.)

To define $SH_*^{-,S^1}(W)$ we fix $\eps>0$
small enough and set
$$
SC^{-,S^1,N}_*(H,J,g):=\bigoplus _{\substack{S_p\subset\cP(H)
    \\ \cA(p)\le \eps}} \Z\langle S_p\rangle \subset
SC^{S^1,N}_*(H,J,g)
$$
and
$$
SC^{+,S^1,N}_*(H,J,g):=SC^{S^1,N}_*(H,J,g)/SC^{-,S^1,N}_*(H,J,g).
$$
The differential on $SC^{\pm,S^1,N}_*(H,J,g)$ is induced by
$\partial^{S^1}$. The corresponding homology groups
$SH^{\pm,S^1,N}_*(H,J,g)$ do not depend on $(J,g)$ and $\eps$, and we
define
$$
SH^{\pm,S^1,N}_*(W):=\lim_{\stackrel \longrightarrow
  {H\in\cH^{S^1}_{N,\reg}}}SH^{\pm,S^1,N}_*(H,J,g).
$$
Passing to the direct limit over $N\to\infty$, we define
$$
SH^{\pm,S^1}_*(W):=\lim_{\stackrel \longrightarrow
  {N}}SH^{\pm,S^1,N}_*(W).
$$
We call $SH^{+,S^1}_*(W)$ the \emph{positive $S^1$-equivariant
  symplectic homology group} of $W$. It follows from the
definitions that $SH_*^{S^1}$ and $SH_*^{\pm,S^1}$ fit into the tautological exact triangle~\eqref{eq:tautologicaltriangle}.

\begin{lemma}[{\cite[Lemma~4.7]{BOGysin}}] \label{lem:minus} There is a natural isomorphism
$$
SH^{-,S^1}_*(W) \simeq H^{S^1}_{*+n}(W,\partial W).
$$
Here $H^{S^1}_{*+n}(W,\partial W)\simeq
H_{*+n}(W,\partial W)\otimes H_*(\C P^\infty)$
is the $S^1$-equivariant homology with integral coefficients for the pair $(W,\partial W)$
with respect to the trivial $S^1$-action.
\end{lemma}

\begin{remark} \label{rmk:non-exact} {\it
Let us indicate what modifications are required in the above constructions in case one works with symplectically atoroidal manifolds instead of exact ones, and one does not assume anymore that $c_1(W)$ is torsion. One starts by fixing a reference loop in each free homotopy class, as well as cylinders that connect the loops in that free homotopy class to the reference loop. The action functional is defined by integrating the symplectic form over these cylinders. One can associate a homology class $A\in H_2(\hat W;\Z)$ to each connecting Floer trajectory, and the resulting moduli spaces $\cM^A_{S^1}(S_\op,S_\up;H,J,g)$ have dimension $-\mu(\op) +\mu(\up) +2\langle
c_1(T\hat W),A\rangle$, where it is understood that the trivialization required to define $\mu(\up)$ is obtained from that used to define $\mu(\up)$ by continuation along the Floer trajectory. The Floer complex is then $\Z$-graded provided one uses coefficients in the Novikov ring $\Lambda_\omega$. 

The definition of $SH_*^{-,S^1}$ has to be modified as well. One needs to use a Hamiltonian which is a $C^2$-small function on $W$ and is $C^2$-close to an increasing function of $r$ on $[1,\infty[\times M$. Since the symplectic form is not exact one cannot guarantee \emph{a priori} that small values of the action single out the constant orbits contained in $W$. The way out is to use arguments from~\cite[p.~654]{BOcont} to prove that, for these Hamiltonians, the constant orbits form a subcomplex for \emph{geometric} reasons (arguing by contradiction, one would obtain a Floer trajectory from a constant orbit at $-\infty$ to a  nonconstant orbit at $+\infty$, and this would contradict the maximum principle). 
}
\end{remark}

\subsection{Floer homology of families}~\cite{Hutchings03,Seidel07,Seidel-Smith-localization}
\label{sec:S1-families} 
 
Let $f:\C P^N\to\R$ be a Morse function and $\tf:S^{2N+1}\to\R$ be its $S^1$-invariant lift. Fix a transverse local slice for each critical $S^1$-orbit of $\tf$. We denote by $\cH_f\subset \cH^{S^1}_N$ the subclass consisting of families $H$ which satisfy the following conditions: 
\renewcommand{\theenumi}{\roman{enumi}}
\begin{enumerate}
\item Critical points of $\cA_{H+\tf}$ lie over critical points of $\tf$;
\item $H_z$ has nondegenerate $1$-periodic orbits for every $z\in\mathrm{Crit}(\tf)$;
\item Near every critical orbit of $\tf$ we have $H(\theta,x,z)=H'(\theta-\tau_z,x)$, where $\tau_z$ is uniquely determined by the condition that $\tau_z^{-1}z$ belongs to the local slice. 
\end{enumerate}

Such Hamiltonian families are easy to construct, as shown in Example~\ref{ex:semi-explicit} below. 

\begin{remark} {\it 
Condition~(iii) is equivalent to $H$ being constant along the local slice at each critical $S^1$-orbit of $\tf$. Note that we allow the corresponding Hamiltonians on $\hat W$ to differ from one local slice to another. The degree of applicability of our construction would not have been restricted by requiring these Hamiltonians to coincide, and we will actually specialize to this situation in~\S\ref{sec:simplifying}. 

Conditions~(i-iii) imply that $H+\tf$ has nondegenerate $S^1$-orbits. As a matter of fact, we could have replaced~(iii) by this nondegeneracy assumption. This will be the relevant setup for~\S\ref{sec:S1-join}. 
}
\end{remark}

\begin{example} \label{ex:semi-explicit}
{\it
Let $\cU$ be a small $S^1$-invariant neighborhood of $\mathrm{Crit}(\tf)$. Given $H_0\in\cH$ with nondegenerate $1$-periodic orbits, we extend it to an $S^1$-invariant function $H':S^1\times\hat W\times \cU\to \R$ by the formula $H'(\theta,x,z):=H_0(\theta-\tau_z,x)$, with $\tau_z$ the unique element in $S^1$ such that $\tau_z^{-1}\cdot z$ belongs to the local slice at $S^1\cdot z$.   
We then extend $H'$ to a globally defined family $H$ with constant slope. An explicit formula for such an extension is $H(\theta,x,z):=\beta(z)H_0(\theta-\tau_z,x)+ (1-\beta(z))\rho(x)H_0(\theta,x)$, with $\beta$ an $S^1$-invariant cut-off function on $S^{2N+1}$, equal to $1$ near $\mathrm{Crit}(\tf)$ and equal to $0$ outside $\cU$, and $\rho$ a cut-off function on $\hat W$, equal to $1$ outside some large compact set and equal to $0$ in the region where $H_0$ depends on time. 

By construction, the family $H$ is constant (equal to $H_0$) along the local slices.

Up to multiplying $\tf$ by a large enough constant, the critical set of $H+\tf$ is the union of nondegenerate $S^1$-orbits $\sqcup_{S_z} S^1\times(\cP(H_0)\times\{z\})$, where $S_z$ are the critical orbits of $\tf$ and $z\in S_z$ is the intersection point with the local slice for $S_z$. Thus $H\in\cH_{Cf}$ for $C\gg 1$.}
\end{example} 

Let $J=(J^\theta_z)$ be an $S^1$-invariant family of almost complex structures in $\cJ$ such that $J^\theta_z$ is independent of $z$ along each local slice. Let $g$ be an $S^1$-invariant metric on $S^{2N+1}$ such that $\nabla \tf$ is tangent to the local slices at the critical orbits of $\tf$. Given $\op=(\og,\oz),\up=(\ug,\uz)\in\mathrm{Crit}(H+\tf)$ we denote $\hat\cM(S_\op,S_\up;H,f,J,g)$ the space of solutions $u:\R\times S^1\to\hat W$, $z:\R\to S^{2N+1}$ to the system of equations (compare~(\ref{eq:Floer1}--\ref{eq:asymptotic}))
\begin{equation} \label{eq:Floer-cont}
\left\{\begin{array}{ccl}
\p_su+ J^\theta_{z(s)}(u)(\p_\theta u-X_{H_{z(s)}}(u)) & = & 0, \\
\dot z - \vec \nabla \tf(z) & = & 0,
\end{array}\right.
\end{equation}
and 
\begin{equation} \label{eq:Floer-cont-asymptotic}
 \lim_{s\to -\infty} (u(s,\cdot),z(s)) \in S_\op, \quad
 \lim_{s\to +\infty} (u(s,\cdot),z(s)) \in S_\up.
\end{equation}
If $S_\op\neq S_\up$ we denote by $\cM(S_\op,S_\up;H,f,J,g)$ the moduli space defined as the quotient by the reparametrization $\R$-action. This moduli space carries a free $S^1$-action and the quotient is denoted $\cM_{S^1}(S_\op,S_\up;H,f,J,g)$. 

Analytically, the system~\eqref{eq:Floer-cont} is much simpler than~(\ref{eq:Floer1}--\ref{eq:Floer2}). Indeed, our assumption on the metric $g$ guarantees that, in the neighborhood of each critical orbit, every gradient trajectory of $\tf$ moves along some $S^1$-translate of the local slice. Since the family $J$ is in turn assumed to be independent of $z$ on the translates of the local slices, the system~\eqref{eq:Floer-cont} reduces to a continuation equation 
\begin{equation} \label{eq:simple-conti}
\p_s u+ J^{\theta,s}_w(u)(\p_\theta u-X_{H^{\theta,s}_w}(u))=0.
\end{equation}
Here $J^{\theta,s}_w$ is a family of almost complex structures that depends on $(\theta,s)$ and on a parameter $w\in\cM(S_\oz,S_\uz;\tf,g)$, the moduli space of gradient trajectories for $\tf$ running from $S_\oz$ to $S_\uz$. More precisely, if we identify $w$ to a gradient trajectory $z(s)$ with specified parametrization, we have $J^{\theta,s}_w:=J^\theta_{z(s)}$. Similarly, we have $H^{\theta,s}_w:=H^\theta_{z(s)}$. Once we have fixed a regular metric $g$, the known transversality arguments in Floer theory~\cite{FHS} show that a generic choice of the family of almost complex structures is regular for the continuation equations~\eqref{eq:simple-conti}. This shows that transversality can be achieved in~(\ref{eq:Floer-cont}--\ref{eq:Floer-cont-asymptotic}) for a fixed $H$ by a generic choice of the pair $(J,g)$. Our additional assumptions on the behavior of $J$ and $g$ near the critical orbits of $\tf$ do not pose any problem, since it is enough to choose the data generically outside fixed neighborhoods of these orbits. 

We now perform the same construction as in~\S\ref{sec:S1-Borel} with~(\ref{eq:Floer-cont}--\ref{eq:Floer-cont-asymptotic}) playing the role of~(\ref{eq:Floer1}--\ref{eq:asymptotic}). For the reader's convenience, we give details for the main steps. 

Given $H\in\cH_f$, we define a chain complex whose underlying $\Z$-module is   
$$
SC_*^{S^1,N}(H,f):=SC_*^{S^1,N}(H,f,J,g):=\bigoplus_{S_p\subset\cP(H+\tf)}\Z\langle S_p\rangle, 
$$ 
with grading given by $|p|:=-\mu(p)+N$ (cf.~\S\ref{sec:S1-Borel}). If $H$ is constant equal to $H_0$ along a local slice at a critical orbit of $\tf$, for $p=(\gamma,z)$ with $\gamma\in\cP(H_0)$ and $z\in\mathrm{Crit}(\tf)$ we have $|p|=-CZ(\gamma)+\ind(z;-\tf)=|\gamma|+\ind(z;-\tf)$. (Here $CZ(\gamma)$ denotes the Conley-Zehnder index of $\gamma$ as a $1$-periodic orbit of $H_0$.)

Let us now assume that $(J,g)$ are chosen generically as above. The moduli spaces $\cM_{S^1}(S_\op,S_\up;H,f,J,g)$ are smooth manifolds of dimension $-\mu(\op)+\mu(\up)-1$ and they carry coherent orientations. As such, every element $[u]$ in such a $0$-dimensional moduli space gets a sign $\eps([u])$, obtained by comparing the coherent orientation with the one determined by the infinitesimal generator of the $S^1$-action. This determines a differential $\p^{S^1}:SC_*^{S^1,N}(H,f)\to SC_{*-1}^{S^1,N}(H,f)$ by the formula 
$$
\partial^{S^1}(S_\op):=\sum_{\substack{
  S_\up\subset \cP(H+\tf) \\
|S_\op| - |S_\up|=1}}
\ \sum_{\scriptstyle [u]\in \cM_{S^1}(S_\op,S_\up;H,f,J,g)}
\epsilon([u])S_\up,
$$
and the usual gluing compactness and gluing argument transposed to this $S^1$-equivariant setting shows that $\p^{S^1}\circ \p^{S^1}=0$. We denote the resulting homology groups $SH_*^{S^1,N}(H,f,J,g)$. The same limiting procedure as in the previous section gives rise to a corresponding version of $S^1$-equivariant symplectic homology groups. In view of the next proposition, these are isomorphic to the ones in~\S\ref{sec:S1-Borel}. 

To formulate the next result, we recall the notation $SH_*^{S^1,N}(H,J,g)$ from~\S\ref{sec:S1-Borel}.

\begin{proposition} \label{prop:the-same}
For any nondegenerate $K\in\cH^{S^1}_N$ that coincides with $H+\tf$ outside a compact set and for any pair $(J',g')\in\cJ^{S^1}_N$ which is regular for $K$ in the sense of~\S\ref{sec:S1-Borel} and  coincides with $(J,g)$ outside a compact set, there is a canonical isomorphism 
$$
SH_*^{S^1,N}(K,J',g')\simeq SH_*^{S^1,N}(H,f,J,g).
$$ 
These isomorphisms commute with continuation maps as $H\to\infty$ and $N\to\infty$. 
\end{proposition}

\begin{proof} 
Let $\eta:\R\to [0,1]$ be a function such that $\eta(s)= 1$ for $s\le 0$ and $\eta(s)= 0$ near $+\infty$. Let $K_s$, $s\in \R$ be an $S^1$-invariant homotopy such that $K_s=K$ for $s\ll 0$ and $K_s=H+\tf$ for $s\ge 0$. Let $(J_s,g_s)$ be an $S^1$-invariant homotopy equal to $(J',g')$ for $s\ll 0$ and equal to $(J,g)$ for $s\gg 0$. Assume also that the homotopies are constant outside a compact set, and that they are generic with respect to this condition. The count modulo $S^1$ of solutions to the equation 
\begin{equation} \label{eq:deformation}
\left\{\begin{array}{ccl}
\p_su+ J^\theta_{s,z(s)}(u)(\p_\theta u-X_{K_{s,z(s)}}(u)) & = & 0, \\
\dot z - \eta(s) \int_{S^1}\vec \nabla K_s(\theta,u(s,\theta),z(s))\,d\theta - (1-\eta(s))\vec \nabla \tf(z(s)) & = & 0
\end{array}\right.
\end{equation}
defines a chain map 
$$
SC_*^{S^1,N}(K,J',g')\to  SC_*^{S^1,N}(H,f,J,g).
$$
By reversing the direction of the homotopies and replacing $\eta$ with $1-\eta$ we obtain a chain map in the opposite direction. A standard continuation argument shows that the two maps are homotopy inverses of each other, hence they induce isomorphisms in homology. The fact that we work on a noncompact manifold 
plays no role since the homotopies are fixed outside a compact set, and so is the first equation in~\eqref{eq:deformation}. 

That the resulting isomorphism commutes with continuation maps is a consequence of the standard ``homotopy of homotopies'' argument in Floer homology. 
\end{proof} 

One drawback of the construction is that, unlike in~\S\ref{sec:S1-Borel}, the parametrized action $\cA_{H+\tf}$ does not decrease along solutions of~\eqref{eq:Floer-cont}. This can be annoying when it comes to define the $\pm$-versions of $S^1$-equivariant symplectic homology, for which we previously used the action filtration. The way out is to consider the kind of Hamiltonians already mentioned in Remark~\ref{rmk:non-exact}. More precisely, one requires $H$ to be $C^2$-small and independent of $\theta$ and $z$ inside $W$. This condition determines a subcomplex of a geometric nature, generated by critical orbits of the form $S^1\cdot(\gamma,z)$, with $\gamma$ a constant orbit at a critical point of $H$ in $W$. The arguments of Proposition~\ref{prop:the-same} extend to prove that the resulting versions of $SH_*^{\pm,S^1}(W)$ coincide with the ones of~\S\ref{sec:S1-Borel}.

\subsection{Simplifying the complex}~\cite{Seidel07} \label{sec:simplifying}

We consider in this section a particular case of the construction in~\S\ref{sec:S1-families} for which the complex $SC_*^{S^1,N}(H,f,J,g)$ takes a very simple form.  We consider a sequence $(f_N,\bar g_N)$, $N\ge 1$ such that $f_N:\C P^N\to \R$ is a perfect Morse function and $\bar g_N$ is a Riemannian metric for which the gradient flow of $f_N$ has the Morse-Smale property. Denote $i_0:\C P^N\hookrightarrow \C P^{N+1}$, $[z_0:\dots:z_{N-1}]\mapsto [z_0:\dots:z_{N-1}:0]$ and $i_1:\C P^N\hookrightarrow \C P^{N+1}$, $[z_0:\dots:z_{N-1}]\mapsto [0:z_0:\dots:z_{N-1}]$. We assume that the following ``periodicity'' conditions hold for all $N\ge 1$: 

\renewcommand{\theenumi}{\roman{enumi}}
\begin{enumerate}
\item \label{item:1} $\mathrm{im}(i_0)$ and $\mathrm{im}(i_1)$ are invariant under the gradient flow of $f_{N+1}$; 
\item \label{item:2} $f_N=f_{N+1}\circ i_0= f_{N+1}\circ i_1 + ct.$, and $i_1^*\bar g_{N+1}=i_0^*\bar g_{N+1}=\bar g_N$. 
\end{enumerate}

We infer that the gradient flow of $f_{N+1}$ on $\mathrm{im}(i_0)$ coincides with the flow on $\mathrm{im}(i_1)$. 
As an example take $f_N([z_0:\dots:z_N]):=C\sum_{j=0}^N(j+1)|z_j|^2/\sum_{j=0}^N|z_j|^2$ for some constant $C>0$, and $\bar g_N$ the metric induced by the round metric on $S^{2N+1}$. 

We denote by $\tf_N$ the $S^1$-invariant lift of $f_N$ to $S^{2N+1}$, and by $\tilde i_0,\tilde i_1:S^{2N+1}\to S^{2N+3}$ the maps $z\mapsto(z,0)$, resp. $z\mapsto(0,z)$, which are lifts of $i_0$ and $i_1$. 

Let $H\in\cH$ be fixed with nondegenerate $1$-periodic orbits. We consider a sequence $H_N\in\cH^{S^1}_N$, $N\ge 1$ as in~\S\ref{sec:S1-families} such that $H_N(\theta,x,z)=H(\theta-\tau_z,x)$ for every $z\in\mathrm{Crit}(\tf_N)$, a sequence $J_N\in\cJ^{S^1}_N$ such that $J_N$ is regular for $H_N$ and~\eqref{eq:Floer-cont}, and assume the following ``periodicity'' conditions to hold for all $N\ge 1$ and $z\in S^{2N+1}$: 
\renewcommand{\theenumi}{\roman{enumi}}
\begin{enumerate}
\setcounter{enumi}{2}
\item \label{item:3} $H_{N+1}(\cdot,\cdot,\tilde i_1(z))=H_{N+1}(\cdot,\cdot,\tilde i_0(z))=H_N(\cdot,\cdot,z)$;
\item \label{item:4} $J_{N+1,\tilde i_1(z)}=J_{N+1,\tilde i_0(z)}=J_{N,z}$.
\end{enumerate}

In this situation we have a natural identification 
$$
SC_*^{S^1,N}(H_N,f_N,J_N,g_N)\simeq \Z[u]/(u^{N+1}) \otimes_\Z SC_*(H,J), 
$$
where $SC_*(H,J)$ is the Floer homology complex of $H$ and $u$ is a formal variable of degree $2$. Indeed, critical points of $H_N+\tf_N$ are of the form $\sqcup_{j=0}^N S^1\cdot(\cP(H)\times \{Z_j\})$, with $Z_j\in\mathrm{Crit}(\tf_N)$ a critical point of index $2j$. Such a point is unique modulo the $S^1$-action since $f_N$ is assumed to be perfect, and the above identification is  given by $S^1\cdot(\gamma,Z_j)\mapsto u^j \otimes\gamma=:u^j\gamma$. For readability we write in the sequel $SC_*(H)$ instead of $SC_*(H,J)$.

Using this identification, it follows from our ``periodicity'' assumptions~(\ref{item:1}-\ref{item:4}) that the differential $\p^{S^1}$ on $SC_*^{S^1,N}(H_N,f_N,J_N,g_N)$ is determined by the maps  
$$
\varphi_j:SC_*(H)\to SC_{*+2j-1}(H), \qquad j=0,\dots,N
$$ 
defined by counting solutions of~(\ref{eq:Floer-cont}--\ref{eq:Floer-cont-asymptotic}) that run from $S^1\cdot(\og,Z_j)$ to $S^1\cdot(\ug,Z_0)$. More precisely, we have 
$$
\varphi_j(\og)=\sum_{\substack{
  \ug\in\cP(H) \\
|\ug|=|\og|+2j-1}}
\ \sum_{\scriptstyle [u]\in \cM_{S^1}(S_{(\og,Z_j)},S_{(\ug,Z_0)};H_N,f_N,J_N,g_N)}
\epsilon([u])\ug,
$$
and
\begin{equation}\label{eq:S1-diff-u}
\p^{S^1}(u^\ell\gamma)=\sum_{j=0}^\ell u^{\ell-j}\varphi_j(\gamma). 
\end{equation}

It also follows from our assumptions~(\ref{item:1}-\ref{item:4}) that, for a fixed $j$, the maps $\varphi_j$ obtained for varying values of $N\ge j$ coincide. As such, we can encode the limit $N\to\infty$ into a complex denoted
$$
SC_*^{S^1}(H):=\Z[u] \otimes_\Z SC_*(H),\qquad |u|=2, 
$$ 
with differential given by~\eqref{eq:S1-diff-u}. Formally we can write (compare~\cite[(8.6)]{Seidel07})
$$
\p^{S^1}=d + u^{-1}\varphi_1 + u^{-2}\varphi_2 + \dots 
$$
and $d$ is the usual Floer differential on $SC_*(H)$. As for the classical case of symplectic homology, there are well-defined continuation maps induced by increasing homotopies of Hamiltonians in $\cH$, and we have 
$$
SH^{S^1}_*(W):=\lim_{\stackrel \longrightarrow
  {H\in\cH}}SH^{S^1}_*(H).
$$

The $\pm$ versions are defined as in the previous sections, based on the $\pm$ versions of $SC_*^{S^1}(H)$ for a Hamiltonian $H\in\cH$ that is a $C^2$-small Morse function on $W$.

\subsection{Algebraic $S^1$-actions} \label{sec:algebraicS1}

At this point it is useful to bring to the fore the algebraic structure that appeared in the previous section.

\begin{definition}
Let $C=(C_\cdot,\p)$ be a $\Z$-graded homological chain complex. An \emph{$S^1$-structure on} $(C_\cdot,\p)$ is the datum of a sequence of maps 
$$
\varphi=(\varphi_0=\p,\varphi_1,\varphi_2,\dots)
$$ 
such that $\varphi_i$ has degree $2i-1$ and we have the relations 
\begin{equation} \label{eq:fi}
\forall k\ge 0,\ \sum_{i+j=k} \varphi_i\varphi_j=0.
\end{equation}
We refer to the pair $(C_\cdot,\varphi)$ as \emph{a $S^1$-complex}, or as \emph{a multicomplex}.
\end{definition}

\begin{remark}{\it 
The notion of a multicomplex already appears in somewhat greater generality in the work of Meyer~\cite{Meyer}. Lapin~\cite{Lapin} and Dotsenko, Shadrin, Vallette~\cite{DSV} use the definition above and study these algebraic objects from the perspective of the Homological Perturbation Lemma. We propose the alternative name of $S^1$-complex in order to stress the geometric origin of this algebraic structure in our context. We view an $S^1$-structure as being extra data attached to a chain complex, much in the same spirit as an $S^1$-action is extra data on a topological space.  
}
\end{remark}

\begin{remark} \label{rmk:mixed-complex} {\it 
A remarkable instance of $S^1$-complex is that in which $\varphi_j=0$ for $j\ge 2$. The datum of a triple $(C_\cdot,b,B)$ such that $b$ has degree $-1$, $B$ has degree $1$, $b^2=0$, $B^2=0$ and $bB+Bb=0$ is now called a \emph{mixed complex} in the literature on cyclic homology (see~\cite[\S2.5.13]{Loday} and the references therein, most notably Kassel~\cite{Kassel-1987} and Burghelea~\cite{Burghelea-1986}, the latter having introduced them under the name of \emph{algebraic $S^1$-chain complexes}). From this point of view, an $S^1$-structure or multicomplex should be understood as, and could be alternatively called, a \emph{mixed complex up to homotopy}, or \emph{$\infty$-mixed complex}. This point of view can be made precise via the Homological Perturbation Lemma for multicomplexes, proved in~\cite{Lapin,Loday-Vallette-book,DSV}. 
}
\end{remark}

The map $\varphi_1$ is the most important in the sequence $(\varphi_1,\varphi_2,\dots)$. Equation~\eqref{eq:fi} shows that $\varphi_1$ is a chain map, whereas the maps $\varphi_i$, $i\ge 2$ can be interpreted as secondary operations, as is customary in the study of algebraic structures up to homotopy. This motivates the following definition. 

\begin{definition} \label{defi:BV} Given an $S^1$-complex $(C_\cdot,\varphi)$ with $\varphi=(\p,\varphi_1,\varphi_2,\dots)$ we call $\varphi_1$ \emph{the Batalin-Vilkovisky operator} or, for short, \emph{the BV-operator}. Traditionally we denote $\Delta:=\varphi_1$. 
\end{definition} 

While $\Delta$ is a chain map, the case $k=2$ in equation~\eqref{eq:fi} shows that $\varphi_2$ is a chain homotopy between $\Delta^2$ and $0$. In particular we have $\Delta^2=0$ in homology. This last identity is also a consequence of the fact that, in the notation of Proposition~\ref{prop:Gysin+spectral} below, $\Delta$ induces in homology the map $B\circ I$. 

Given an $S^1$-complex $(C_\cdot,\varphi)$ we can form the complex $\tilde C=(\tilde C_\cdot,\tilde\p_\varphi)$
 with 
$$
\tilde C_. := \Z[u]\otimes_\Z C_\cdot, \qquad |u|=2
$$
and differential
$$
\tilde\p_\varphi ( u^\ell\otimes x) := \sum_{j=0}^\ell  u^{\ell-j} \otimes \varphi_j(x). 
$$
Note that $\tilde C_\cdot$ is \emph{not} a differential $\Z[u]$-module with respect to the natural multiplication by $u$. On the other hand, $\tilde C_\cdot$ is a $\Z[u^{-1}]$-module with respect to the operation 
$$
S:\tilde C_\cdot\to \tilde C_{\cdot -2},
$$  
$$
S(u^\ell\otimes x) := \left\{\begin{array}{cl}Ê0, & \ell=0,\\Êu^{\ell-1}\otimes x, &  \ell\ge 1.
\end{array}Ê\right.
$$

\begin{definition}
The $S^1$-equivariant homology of the $S^1$-complex $(C_\cdot,\varphi)$ is 
$$
H_\cdot^{S^1}(C_\cdot,\varphi):=H_\cdot(\tilde C).
$$
\end{definition}

\begin{proposition} \label{prop:Gysin+spectral}
Let $(C_\cdot,\varphi)$ be an $S^1$-complex. 
\renewcommand{\theenumi}{\roman{enumi}}
\begin{enumerate}
\item There is a Gysin exact triangle 
\begin{equation*} 
\xymatrix
@C=20pt
{
H_\cdot(C_\cdot) \ar[rr]^I & & 
H_\cdot^{S^1}(C_\cdot,\varphi) \ar[dl]^{[-2]}_S \\ & H_\cdot^{S^1}(C_\cdot,\varphi) \ar[ul]^{[+1]}_B
}
\end{equation*}
The map $I$ is induced by $x\mapsto 1\otimes x$, the map $S$ is induced by ``multiplication with $u^{-1}$" as above, whereas the map $B$ is induced by $u^\ell\otimes x\mapsto \varphi_{\ell+1}(x)$.
\item There is a spectral sequence $E^r_{p,q}$, $r\ge 2$ converging to $H_\cdot^{S^1}(C_\cdot,\varphi)$ with second page given by 
$$
E^2_{p,q}\simeq H_p(BS^1)\otimes H_q(C).
$$
\end{enumerate}
\end{proposition} 

\begin{proof}
To prove (i) we observe that there is a short exact sequence of complexes
\begin{equation} \label{eq:Gysin-short-ex-seq}
0\to C_\cdot \to \tilde C_\cdot \to \tilde C[-2]_\cdot \to 0
\end{equation} 
Here we denote $\tilde C[-2]_\cdot:=\tilde C_{\cdot -2}$. The first map in the above exact sequence acts by $x\mapsto 1\otimes x$, whereas the second map is projection onto $u\Z[u]\otimes_\Z C_\cdot$, which is naturally identified with $\tilde C[-2]$. The Gysin triangle is the homology exact triangle associated to~\eqref{eq:Gysin-short-ex-seq}. The expressions for the maps $I$ and $S$ at chain level follow directly from the definitions, whereas the expression for the map $B$ at chain level follows from a straightforward diagram chasing: given a cycle $\sum_{i=0}^{\ell-1}u^{\ell-i-1}\otimes x_i\in\tilde C_{k-2}$ with $|x_i|=k-2\ell+2i$, we lift it to $\sum_{i=0}^{\ell-1}u^{\ell-i}\otimes x_i\in \tilde C_k$, and compute 
\begin{eqnarray*}
\tilde \p_\varphi(\sum_{i=0}^{\ell-1}u^{\ell-i}\otimes x_i) & = & \sum_{i=0}^{\ell-1} \sum_{j=0}^{\ell-i} u^{\ell-i-j}\otimes \varphi_j(x_i) \\
& = & u \sum_{i=0}^{\ell-1} \sum_{j=0}^{\ell-i-1} u^{\ell-i-1-j} \otimes \varphi_j(x_i) \ + \ \sum_{i=0}^{\ell-1} 1\otimes \varphi_{\ell-i}(x_i) \\
& = & u\tilde \p_\varphi( \sum_{i=0}^{\ell-1}u^{\ell-i-1}\otimes x_i) \ + \ \sum_{i=0}^{\ell-1} 1\otimes \varphi_{\ell-i}(x_i)\\
& = & 1\otimes (\sum_{i=0}^{\ell-1} \varphi_{\ell-i}(x_i)).
\end{eqnarray*}
The term $\sum_{i=0}^{\ell-1} \varphi_{\ell-i}(x_i)$ is a cycle in $C_{k-1}$. 
By definition, the boundary map $B$ in the Gysin exact triangle is induced by 
$$
\sum_{i=0}^{\ell-1}u^{\ell-i-1}\otimes x_i \mapsto \sum_{i=0}^{\ell-1} \varphi_{\ell-i}(x_i),
$$
which is in turn the same as the chain map $u^\ell\otimes x\mapsto \varphi_{\ell+1}(x)$ (note that this being a chain map is equivalent to the fact that the $\varphi_i$'s satisfy the relations~\eqref{eq:fi}).

To prove~(ii), we consider the increasing filtration on $\tilde C$ defined by $F_p\tilde C:=\Z\langle 1,u,\dots,u^p\rangle\otimes_\Z C_\cdot$. The associated spectral sequence is the desired one, taking into account that $H_\cdot(BS^1)\simeq \Z[u]$.  
\end{proof}

\begin{remark} {\it
The spectral sequence in~(ii) converges if either $C_i=0$ for $i\ll 0$ or for $i\gg 0$, or $C_i$ is finite-dimensional for each $i$. If neither of these conditions is satisfied, convergence cannot be guaranteed. 

In our intended application $C_\cdot=SC_\cdot(H)$ is finite dimensional since $H$ is assumed to be linear at infinity. Note that in both non-equivariant and equivariant symplectic homology the result of the direct limit over $H\in\cH$ is equal to the homology group obtained using a single quadratic Hamiltonian. 
}
\end{remark}

\begin{definition} \label{defi:S1chainmap}
Let $(C_\cdot,\varphi)$ and $(D_\cdot,\psi)$ be two $S^1$-complexes. An \emph{$S^1$-equivariant chain map} or, for short, \emph{chain map}, from $(C_\cdot,\varphi)$ to $(D_\cdot,\psi)$ is the datum of a sequence of maps
$$
\Phi=(\Phi_0,\Phi_1,\Phi_2,\dots)
$$
such that $\Phi_i$ has degree $2i$ and we have the relations
$$
\forall k\ge 0, \ \sum_{i+j=k} \Phi_i \varphi_j - \psi_j\Phi_i=0.
$$
Given two chain maps $\Phi$ and $\Psi$, an \emph{$S^1$-equivariant homotopy} or, for short, \emph{homotopy}, from $\Phi$ to $\Psi$ is a sequence of maps 
$$
h=(h_0,h_1,h_2,\dots)
$$
such that $h_i$ has degree $2i+1$ and we have the relations 
$$
\forall k\ge 0,\ \Phi_k-\Psi_k=\sum_{i+j=k} h_i\varphi_j + \psi_jh_i. 
$$
\end{definition}

An $S^1$-equivariant chain map $\Phi$ defines a chain map $\tilde \Phi:\tilde C\to \tilde D$ via $u^\ell\otimes x\mapsto \sum_{i=0}^\ell u^{\ell-i}\otimes \Phi_i(x)$. Similarly, an $S^1$-equivariant homotopy $h$ from $\Phi$ to $\Psi$ defines a homotopy from $\tilde \Phi$ to $\tilde \Psi$ via $u^\ell\otimes x\mapsto \sum_{i=0}^\ell u^{\ell-i}\otimes h_i(x)$.

\begin{definition}
Let $(C_\cdot,\varphi)$ be an $S^1$-complex. An \emph{$S^1$-subcomplex} is a submodule $C_\cdot^-\subset C_\cdot$ which is invariant under the maps $\varphi_i$, $i\ge 0$. 
\end{definition}

Given an $S^1$-subcomplex $C_\cdot^-$, the quotient complex $C_\cdot^+:=C_\cdot/C_\cdot^-$ inherits an $S^1$-structure $\bar\varphi=(\bar\varphi_0,\bar\varphi_1,\dots)$ defined by $\bar\varphi_i([x]):=[\varphi_i(x)]$. Here we denoted by $[x]$ the class of an element $x\in C_\cdot$. The maps $\bar\varphi_i$ are well-defined because $C_\cdot^-$ is assumed to be invariant under the maps $\varphi_i$. 

The following result is an algebraic version of Theorem~\ref{thm:Gysin-spec}.

\begin{proposition} \label{prop:Gysin-spec} 
Let $(C_\cdot,\varphi)$ be an $S^1$-complex and $C_\cdot^-\subset C_\cdot$ an $S^1$-subcomplex. Denote $C_\cdot^+$ the quotient $S^1$-complex. The Gysin exact triangles and spectral sequences in Proposition~\ref{prop:Gysin+spectral} for $C_\cdot$ and $C_\cdot^\pm$ are compatible with the homology exact triangle associated to the exact sequence $0\to C_\cdot^-\to C_\cdot\to C_\cdot^+\to 0$.
\end{proposition}

\begin{proof}
The starting point is the induced short exact sequence of complexes $0\to \tilde C_\cdot^-\to \tilde C_\cdot\to \tilde C_\cdot^+\to 0$. 
The assertion concerning the spectral sequences is straightforward, since the respective filtrations fit into exact sequences $0\to F_p\tilde C_\cdot^-\to F_p\tilde C_\cdot\to F_p\tilde C_\cdot^+\to 0$ for all $p\ge 0$. 

The assertion concerning the Gysin exact triangles follows from Lemma~\ref{lem:grid} below, applied with 
$(X,X',X'')=(C^-,C,C^+)$, $(Y,Y',Y'')=(\tilde C^-,\tilde C,\tilde C^+)$, and $(Z,Z',Z'')=(\tilde C^-[-2],\tilde C[-2],\tilde C^+[-2])$. 
\end{proof}

\begin{lemma} \label{lem:grid}
Consider the following commutative diagram whose lines and co\--lumns are short exact sequences of complexes
\begin{equation} \label{eq:exact}
\xymatrix{
& 0 \ar[d] & 0 \ar[d] & 0 \ar[d] &  \\ 
0\ar[r] & X \ar[r]^u \ar[d]^f & Y \ar[r]^v \ar[d]^g & Z \ar[r]
\ar[d]^h & 0 \\
0\ar[r] & X' \ar[r]^{u'}  \ar[d]^{f'} & Y' \ar[r]^{v'} \ar[d]^{g'} & Z' \ar[r] \ar[d]^{h'}
 & 0 \\
0\ar[r] & X'' \ar[r]^{u''} \ar[d] & Y'' \ar[r]^{v''} \ar[d] & Z'' \ar[r] \ar[d]
 & 0  \\
& 0 & 0 & 0 &
}
\end{equation}
and consider the diagram formed by the homological long exact sequences 
\begin{equation} \label{eq:big-grid}
\xymatrix
@C=20pt
@R=20pt
{
 & \vdots \ar[d] & \vdots \ar[d] & \vdots \ar[d] & \vdots \ar[d] & \\
\cdots \ar[r] & H_\cdot(X) \ar[r] \ar[d] & H_\cdot(Y) \ar[r] \ar[d]
& H_\cdot(Z) \ar[r]^-\p \ar[d] & H_{\cdot-1}(X) \ar[r] \ar[d] & \cdots \\
\cdots \ar[r] & H_\cdot(X') \ar[r] \ar[d] &
H_\cdot(Y') \ar[r] \ar[d] & H_\cdot(Z')
\ar[r]^-{\p'} \ar[d] & H_{\cdot-1}(X') \ar[r] \ar[d] & \cdots \\
\cdots \ar[r] & H_\cdot(X'') \ar[r] \ar[d]_-{\p_X} & H_\cdot(Y'') \ar[r]
\ar[d]_-{\p_Y} & H_\cdot(Z'') \ar[r]^-{\p''} \ar[d]_-{\p_Z} & H_{\cdot-1}(X'') \ar[r]
\ar[d]^-{\p_X} & \cdots \\
\cdots \ar[r] & H_{\cdot-1}(X) \ar[r] \ar[d] & H_{\cdot-1}(Y) \ar[r] \ar[d]
& H_{\cdot-1}(Z) \ar[r]^-\p \ar[d] & H_{\cdot-2}(X) \ar[r] \ar[d] & \cdots \\
& \vdots & \vdots & \vdots & \vdots &
}
\end{equation}
All the squares in this diagram are commutative with the exception of the bottom right square
\begin{equation} \label{eq:square-diag}
\xymatrix
@C=20pt
@R=20pt
{
H_\cdot(Z'')\ar[r]^-{\p''} \ar[d]_{\p_Z} & H_{\cdot-1}(X'') \ar[d]^{\p_X} \\
H_{\cdot-1}(Z) \ar[r]^-{\p\ }Ê& H_{\cdot-2}(X)
}
\end{equation}
which is anti-commutative, meaning that it satisfies the relation
\begin{equation} \label{eq:square}
\p_X\p''+\p\p_Z=0.
\end{equation}
\end{lemma}

\begin{proof}
The homological long exact sequence is functorial with respect to morphisms of short exact sequences. This implies commutativity of all the squares in~\eqref{eq:big-grid} with the exception of the bottom right square~\eqref{eq:square-diag} which has to be treated separately and for which we prove that the relation~\eqref{eq:square} holds.  

The sign change can be interpreted as follows if we work with field coefficients: the middle line in our diagram of complexes is isomorphic to a short exact sequence of cones (see Remark~\ref{rmk:cone}) of chain maps $X''\to X[-1]$, $Y''\to Y[-1]$, respectively $Z''\to Z[-1]$. The resulting diagram in homology would commute, which amounts to anti-commutativity of the above square in view of $\p_{[-1]}=-\p$ (see also~\cite[Lemma~5.6]{BOGysin}). If we work with arbitrary coefficients one can argue in a similar fashion using distinguished triangles in the derived category.

We give here an elementary proof of~\eqref{eq:square}. In order to follow the proof we suggest that the reader draws three copies of~\eqref{eq:exact} in degrees $i$, $i-1$ and $i-2$. 

Denote the differentials of our complexes by $d_X$, $d_Y$ etc. Elements of the complexes will be denoted by $x_i\in X_i$, $y'_i\in Y'_i$ etc. We recall that the connecting map $\p'':H_\cdot(Z'')\to H_{\cdot-1}(X'')$ is defined as follows. Given a cycle $z''_i\in Z''_i$, we lift it to $y''_i$, apply $d_{Y''}$, find a (unique) preimage $x''_{i-1}$ of $d_{Y''}y''_i$ and define $\p''[z''_i]:=[x''_{i-1}]$. 

Let us prove that $\p_X\p''+\p\p_Z=0$. 

Let $z''_i\in Z''_i$ be a cycle. Let us choose lifts $z'_i$ and $y''_i$ admitting a common lift $y'_i$. Then $d_{Y'}y'_i$ is a common lift for $d_{Z'}z'_i$ and $d_{Y''}y''_i$. Choose unique preimages $z_{i-1}$ and $x''_{i-1}$  for the latter (these are cycles and define  $\p''[z''_i]$, resp. $\p_Z[z''_i]$). Choose a lift $y_{i-1}$ of $z_{i-1}$. Then $g_{i-1}(y_{i-1})$ and $d_{Y'}y'_i$ both project onto $d_{Z'}z'_i$, so that there exists $x'_{i-1}$ such that 
$$
d_{Y'}y'_i=g_{i-1}(y_{i-1}) - u'_{i-1}(x'_{i-1}).
$$
The element $x'_{i-1}$ is necessarily a lift of $\bar x''_{i-1}:=-x''_{i-1}$ since its image under the composite map $g'_{i-1}\circ u'_{i-1}$ is $-d_{Y''}y''_i$. Let $x_{i-2}$, $\tilde x_{i-2}$ be the unique preimages of $d_{X'}x'_{i-1}$ and $d_Yy_{i-1}$. By applying $d_{Y'}$ the previous equation yields $u'_{i-2} d_{X'}x'_{i-1}= g_{i-2}d_Yy_{i-1}$ and, by injectivity of all the maps that are involved, we obtain $x_{i-2}=\tilde x_{i-2}$. The class of this cycle defines both $\p_X[\bar x''_{i-1}]$ and $\p[z_{i-1}]$, and we obtain $-\p_X\p''[z''_i]=\p\p_Z[z''_i]$.  
\end{proof}

\begin{proof}[Proof of Theorem~\ref{thm:Gysin-spec}]
The main observation is that, for any Hamiltonian $H$ as in~\S\ref{sec:simplifying}, the subcomplex $SC_*^{S^1,-}(H)$ is an $S^1$-subcomplex in the sense of the previous definition. The Theorem is then a straightforward consequence of Proposition~\ref{prop:Gysin-spec}, using the exactness of the direct limit functor and the fact that a continuation map determined by an increasing homotopy of Hamiltonians determines a morphism between the  diagrams~\eqref{eq:exact} associated to inclusions $SC_*^{S^1,-}(H)\subset SC_*^{S^1}(H)$.
\end{proof}

\subsection{Join construction}
\label{sec:S1-join}

\subsubsection{The join construction~\cite{Milnor1956}} \label{sec:join}

Milnor's construction, which we now recall, provides explicit models for $ES^1$ and $BS^1:=ES^1/S^1$. Let 
$$
\Delta^N=\{(t_0,\dots,t_N)\in\R^{N+1}\, : \, t_i\ge 0, \ \sum_{i=0}^N t_i=1\}
$$ 
be the standard $N$-dimensional simplex. Given topological spaces $X_i$, $i=0,\dots,N$, denote by $(t_0x_0,\dots,t_Nx_N)$ an element of $X_0\times\dots\times X_N\times\Delta^N$, with $x_i\in X_i$ and $t=(t_0,\dots,t_N)\in\Delta^N$. The join $X_0*\dots*X_N$ is defined as the quotient of $X_0\times\dots\times X_N\times\Delta^N$ by the equivalence relation for which
$(t_0x_0,\dots,t_Nx_N)\sim (t'_0x'_0,\dots,t'_Nx'_N)$ if and only if $t_i=t'_i$ for all $i$ and $x_i=x'_i$ whenever $t_i=t'_i>0$. The join of a countable sequence $(X_i)_{i\ge 0}$ of topological spaces is defined as the direct limit of finite joins under the maps $X_0*\dots *X_N\hookrightarrow X_0*\dots*X_{N+1}$ induced by the standard inclusions $\Delta^N\hookrightarrow \Delta^{N+1}$. 
Given a topological group $G$, the infinite join $G*G*G\dots$ is a model for the \emph{universal principal $G$-bundle $EG$}. The group $G$ acts on $EG$ by $g\cdot(t_0g_0,\dots,t_Ng_N,\dots):=(t_0(gg_0),\dots,t_N(gg_N),\dots)$. The quotient $BG:=EG/G$ is called \emph{the classifying space of $G$}. 

In the case $G=S^1$ which is of interest for this paper, we have an identification between the $2N+1$-dimensional sphere $S^{2N+1}=\{(z_0,\dots,z_N)\in\C^{N+1}\, : \, \sum_{i=0}^N |z_i|^2=1\}$ and the join of $N+1$ circles $S^1*\dots *S^1$ via the map  
$$
(z_0,\dots,z_N)\mapsto (|z_0|^2\arg(z_0),\dots,|z_N|^2\arg(z_N)).
$$
We thus recover the models $ES^1=\lim_{N\to\infty} S^{2N+1}$ and $BS^1=\C P^\infty$. Our convention for $\arg(z)\in\R/\Z$ is to define it by the formula $z=|z|e^{2\pi \arg(z) i}$ for $z\in\C^*$.

\subsubsection{$S^1$-invariant extension of time-dependent Hamiltonians}\label{sec:HN}

Let $H:S^1\times \hat W\to \R$ be a time-dependent Hamiltonian. We extend $H$ to an $S^1$-invariant Hamiltonian 
$$
H_{N,0}:S^1\times \hat W\times S^{2N+1}\to \R
$$
by the formula 
$$
H_{N,0}(\theta,\cdot,z):=\sum_{j=0}^N |z_j|^2 H^{\theta-\arg(z_j)}. 
$$
Equivalently, viewing $S^{2N+1}$ as the join of $N+1$ copies of $S^1$ as in~\S\ref{sec:join}, we have 
$$
H_{N,0}(\theta,\cdot,(t_0\tau_0,\dots,t_N\tau_N)):=\sum_{j=0}^N t_j H^{\theta-\tau_j}.
$$ 

\begin{remark}[On the smoothness of $H_{N,0}$] {\it 
The Hamiltonian $H_{N,0}$ is of class $C^1$ with respect to the parameter $z$, as can be seen from the formula for $\vec\nabla_z H_{N,0}$ in the proof of Lemma~\ref{lem:HNf} below. However, $H_{N,0}$ is in general not of class $C^2$. In practice, one has to consider a smooth approximation of $H_{N,0}$, say of the form $\sum_{j=0}^N \rho(|z_j|)H^{\theta-\arg(z_j)}$, with $\rho:[0,1]\to [0,1]$ a smooth strictly increasing function that equals $1$ at $1$ and that vanishes at $0$ together with all its derivatives. In the sequel we work with $H_{N,0}$ in order not to burden the notation. 
}
\end{remark}

Given a function $f:\C P^N\to \R$ we define 
$$H_{N,f}:S^1\times \hat W\times S^{2N+1}\to \R$$ 
by the formula
$$
 H_{N,f}(\theta,x,z):=H_{N,0}(\theta,x,z)+\tf(z),
$$
where $\tf:S^{2N+1}\to\R$ denotes as usual the $S^1$-invariant lift of $f$.

For the statement of the next Lemma, we recall the definition of Hofer's norm  
$$\|H\|=\int_{S^1} \max_x H(\theta,x) - \min_x H(\theta,x)\, d\theta.$$ 

\begin{lemma} \label{lem:HNf}
Let $f([z_0:\dots:z_N]):=\sum_{j=0}^N a_j|z_j|^2/\|z\|^2$ with $a_j\in\R$. Denote by $Z_j\in S^{2N+1}$ the vector whose $j$-th component is equal to $1$. 

(i) We have $S^1\cdot (\cP(H)\times \{Z_j\, : \, j=0,\dots,N\}) \subseteq \mathrm{Crit}(\cA_{H_{N,f}})$. 

(ii) Assume $a_j-a_{j-1}>\|H\|$ for all $j=1,\dots,N$. Then we have equality in (i).
\end{lemma}

\begin{proof}
Denote $\dot H(\theta,x)=\p_\theta H(\theta,x)$ and remark that $\tf(z)=\sum_{j=0}^N a_j|z_j|^2$. A straightforward computation shows that 
$$
\vec\nabla_z H_{N,0}(\theta,x,z)= \big( 2(H^{\theta-\arg z_j}(x)-H_{N,0}(\theta,x,z))z_j - \frac 1 {2\pi} \dot H^{\theta-\arg z_j}(x) iz_j\big)_{j=0,\dots,N},
$$
whereas
$$
\vec \nabla \tf(z) = \big( 2(a_j-\tf(z))z_j\big)_{j=0,\dots,N}.
$$

The equations for a critical point $(\gamma,z)$ of $\cA_{H_{N,f}}$ are therefore 
\begin{equation} \label{eq:crit1}
\dot\gamma(\theta)=\sum_{j=0}^N |z_j|^2 X_H^{\theta-\arg z_j}(\gamma(\theta))
\end{equation}
and, for all $j$, either $z_j=0$ or 
\begin{equation}\label{eq:crit2}
\int_{S^1} H^{\theta-\arg z_j}(\gamma(\theta))\, d\theta - \int_{S^1} H_{N,0}(\theta,\gamma(\theta),z)\, d\theta + a_j - \tf(z) =0,
\end{equation}
\begin{equation}\label{eq:crit3}
\int_{S^1} \dot H^{\theta-\arg z_j}(\gamma(\theta))\, d\theta=0.
\end{equation}

Let us prove (i). We prove more, namely that if $z=Z_j$, then $(\gamma,z)\in\mathrm{Crit}(\cA_{H_{N,f}})$ if and only if $\gamma\in\cP(H)$. Indeed, we have $H_{N,f}(\theta,x,z)=H(\theta,x) + a_j$ and~\eqref{eq:crit1} is equivalent to $\gamma\in\cP(H)$. Assuming now that $\gamma\in\cP(H)$ we see that~\eqref{eq:crit2} is tautologically satisfied, whereas for~\eqref{eq:crit3} we compute with $\arg z_j=0$ 
\begin{eqnarray*}
\int_{S^1}\dot H^{\theta-\arg z_j}(\gamma(\theta))\, d\theta & = & \int_{S^1} \frac d {d\theta} H^\theta(\gamma(\theta))\, d\theta - \int_{S^1} \p_x H^\theta(\gamma)\dot\gamma\, d\theta \\Ê
& = & -\int_{S^1}\hat\omega(X_H^\theta(\gamma),X_H^\theta(\gamma))\, d\theta=0.
\end{eqnarray*}

We prove (ii). Let $(\gamma,z)\in\mathrm{Crit}(\cA_{H_{N,f}})$ and denote $h_j=\int H^{\theta-\arg z_j}(\gamma(\theta))\, d\theta$. Assuming by contradiction that there exist $k>j$ such that $z_k$ and $z_j$ are both nonzero, we subtract equations~\eqref{eq:crit2} for $k$ and $j$ to obtain $h_j-h_k=a_k-a_j$. Now $h_j-h_k$ is bounded from above by $\|H\|$, which contradicts our assumption on $f$. Thus the vector $z$ has exactly one nonzero component and, up to multiplying by an element of $S^1$, we can assume that the latter is equal to $1$, so that $z$ is equal to some $Z_j$. We then conclude using (i). 
\end{proof}

\begin{remark} {\it
The above condition $a_j-a_{j-1}>\|H\|$ in~(ii) is satisfied by functions of the form $f([z_0:\dots:z_N])=C\sum_{j=0}^N(j+1)|z_j|^2/\sum_{j=0}^N|z_j|^2$ with $C>\|H\|$, and these functions also satisfy assumption~(\ref{item:2}) in~\S\ref{sec:simplifying}.
}
\end{remark}

\subsubsection{Representations}Ê\label{sec:representations}

Let $\mu=(t_0,\dots,t_N):\C P^N\to\Delta^N$ be the moment map for the standard $\T^N$-action, defined by 
$$
t_j([z]):=|z_j|^2/\|z\|^2.
$$
Given indices $N\ge k\ge j\ge 0$ we denote by $\Delta_{k,j}$ the $k-j$-dimensional simplex $\{t=(t_0,\dots,t_N)\in\Delta^N\, : \, \mathrm{supp}(t)\subset \{j,\dots,k\}\}$. Given indices $k\ge k'\ge j'\ge j$ we denote $i_{\Delta_{k',j'}\hookrightarrow \Delta_{k,j}}$ the natural inclusion. 

Let $f:\C P^N\to\R$ be given by $f([z_0:\dots:z_N]):=\sum_{j=0}^N a_j|z_j|^2/\|z\|^2$ with $a_N>a_{N-1}>\dots>a_0>0$. 
This is a perfect Morse function whose unique critical point of index $2j$ is $[Z_j]=[0\ldots :1:0\dots]$, the point for which the non-vanishing coordinate appears on the $j$-th position.  We have $f([Z_j])=a_j$. 

Let $\tf$ be the $S^1$-invariant lift of $f$ to $S^{2N+1}$ equipped with the round metric. We have computed $\vec\nabla \tf (z)=2((a_j-\tf(z))z_j)$, so that for any $k>j$ the negative gradient flow of $\tf$ preserves the spheres $S_{k,j}=\{z=(z_0,\dots,z_N)\in S^{2N+1}\, : \, \mathrm{supp}(z)\subset\{j,\dots,k\}\}$. Each such sphere has dimension $2(k-j)+1$, and the function $\tf|_{S_{k,j}}$ achieves its minimum, resp. maximum, along the orbit $S^1\cdot Z_j$, resp. $S^1\cdot Z_k$. Similarly, the negative gradient flow of $f$ preserves $S_{k,j}/S^1\simeq \C P^{k-j}$.

For $k>j$ we denote $\hat\cM(k,j)$ the space of negative gradient trajectories of $f$ connecting $[Z_k]$ and $[Z_j]$, and $\cM(k,j):=\hat\cM(k,j)/\R$ the moduli space of negative gradient trajectories. We denote $\hat\cM(j,j)=\cM(j,j)=\{[Z_j]\}$. Note that $\cM(k,j)$ is identified with $S^{2(k-j)-1}\setminus S^{2(k-j)-3}$, and its Morse-Floer compactification $\overline \cM(k,j)$ is such that 
$$\p \overline \cM(k,j):=\overline \cM(k,j)\setminus \cM(k,j)=\cup_{k>\ell>j} \overline \cM(k,\ell)\times \overline \cM(\ell,j).
$$
As we will see below, the identities $\sum_{i+j=k}\varphi_i\varphi_j=0$, $k\ge 0$ from~\eqref{eq:fi} defining an $S^1$-complex reflect the above structure of the boundary of the moduli space $\mathcal{M}(N,0)$. 

In the next definition we denote by $\mathscr{X}$ a collection consisting of topological spaces $X_{k,j}$, $N\ge k\ge j\ge 0$ together with maps $i_{(k',j')\rightarrow (k,j)}:X_{k',j'}\to X_{k,j}$ defined for $k\ge k'\ge j'\ge j$ and satisfying $i_{(k',j')\rightarrow (k,j)}\circ i_{(k'',j'')\rightarrow (k',j')}=i_{(k'',j'')\rightarrow (k,j)}$ and $i_{(k,j)\rightarrow (k,j)}=\mathrm{Id}$. Examples are the collection $\mathscr{J}$ consisting of $X_{k,j}=\cJ$, the space of time-dependent almost complex structures on $\hat W$ with all maps equal to the identity, and the collection $\Delta$ consisting of $X_{k,j}=\Delta_{k,j}$ with maps $i_{(k',j')\rightarrow (k,j)}:=i_{\Delta_{k',j'}\hookrightarrow \Delta_{k,j}}$. We denote $\cF(\R,\mathscr{X}):=\sqcup_{i,j} \cF(\R,X_{i,j})$, where $\cF(\R,X_{i,j})$ is the space of continuous functions $\R\to X_{i,j}$.

\begin{definition} \label{defi:representation} 
A \emph{$\cF(\R,\mathscr{X})$-valued representation of $(\C P^N,f)$} is a collection of continuous maps 
$$
\rho_{k,j} : \cM(k,j)\to\cF(\R,X_{k,j}), \quad N\ge k\ge j \ge 0
$$
satisfying the following \emph{Gluing Axiom}:
\begin{center}
Assume $\lambda^\nu\in\cM(k,j)$, $\nu\in\N$ is a sequence that converges in the Floer sense~\cite{Floer-Lagrangian}, i.e. in $C^\infty_{loc}$ up to shifts, to a broken gradient trajectory $\lambda_1\#\dots\#\lambda_\ell$, $\ell\ge 1$ with $\lambda_i\in\cM(k_{\ell-i+1},k_{\ell-i})$, $k_\ell=k$, $k_0=j$. There exist sequences of real numbers (``shifts") $\sigma_1^\nu,\dots,\sigma_\ell^\nu$ s.t.
$$
\forall\,i, \ \rho_{k,j}(\lambda^\nu)(\cdot + \sigma_i^\nu)\stackrel {C^0_{\mathrm{loc}}}\longrightarrow i_{(k_{\ell-i+1},k_{\ell-i})\rightarrow (k,j)} \circ \rho_{k_{\ell-i+1},k_{\ell-i}}(\lambda_i).
$$
\end{center}
\end{definition}

\begin{definition} \label{defi:translinvrep}
A $\cF(\R,\mathscr{X})$-valued representation of $(\C P^N,f)$ is said to be \emph{translation invariant} if we are given identifications $X_{k,j}\simeq X_{k-\ell,j-\ell}$ for all $k>j$ and $0\le \ell\le j$, such that $\rho_{k,j}=\rho_{k-\ell,j-\ell}$ with respect to the natural identifications $\cM(k,j)\simeq \cM(k-\ell,j-\ell)$.
\end{definition}

For given $f$ and $\mathscr{X}$, the space of (translation invariant) representations is convex if each $X_{k,j}$ is convex and the inclusions are linear, and it is contractible if each $X_{k,j}$ is contractible. 

In the case $\mathscr{X}=\Delta$ or $\mathscr{X}=\mathscr{J}$, which are the ones of interest for our construction, we assume in addition to Definition~\ref{defi:representation} convergence in $C^\infty_{\mathrm{loc}}$. One could of course formulate a variant of Definition~\ref{defi:representation} along these lines assuming that $\mathscr{X}$ is an arbitrary closed subset of a Fr\'echet manifold, but such a generality would not be justified by the specific purposes that we have in mind. 

We now give two examples. 

\begin{example} \label{ex:1} {\it 
Let us fix for each $N\ge k>j\ge 0$ a section of the principal $\R$-bundle $\hat\cM(k,j)\to\cM(k,j)$ and identify $\cM(k,j)$ with the image of this section in $\hat\cM(k,j)$. It then makes sense to talk about $\lambda(s)$, $s\in\R$ for $\lambda\in\cM(k,j)$, and for every $\hat\lambda\in\hat\cM(k,j)$ there exist unique $\lambda\in\cM(k,j)$ and $\sigma\in\R$ such that $\hat\lambda=\lambda(\cdot+\sigma)$. 

Recall the moment map $\mu=(t_0,\dots,t_N):\C P^N\to\Delta^N$ and define $\rho_{k,j}:\cM(k,j)\to\cF(\R,\Delta_{k,j})$ by 
$$
\rho_{k,j}(\lambda)(s):=\mu(\lambda(s)).
$$
The collection $\{\rho_{k,j}\}$ is a $\cF(\R,\Delta)$-valued representation of $(\C P^N,f)$. Moreover, we have by definition $(\rho_{k,j})_i(\lambda)=0$ whenever $t_i(\lambda)=0$, since $\mu=(t_0,\dots,t_N)$.
This representation is translation invariant if the coefficients of the function $f$ satisfy $a_j-a_{j-1}=ct.$ for all $j=1,\dots,N$ and if the sections $\cM(k,j)\to\hat\cM(k,j)$ coincide under the natural identifications $\cM(k,j)\simeq \cM(k-\ell,j-\ell)$, $0\le \ell\le j$.
}
\end{example}

\begin{example}Ê\label{ex:2} {\it 
We identify $\cM(j,j-1)$, $N\ge j\ge 1$ with $S^1$ and denote the parameter on this circle by $\tau_{j-1}$. The gluing construction for gradient trajectories of $f$ provides a diffeomorphism between a neighborhood of $\cM(N,N-1)\times\cM(N-1,N-2)\times\dots\times\cM(1,0)$ inside $\overline\cM(N,0)\setminus \p\overline\cM(N,0)$ and the space $(S^1)^N\times ([R,\infty[)^{N-1}$, $R\gg 0$, with parameters $(\tau_{N-1},\dots,\tau_0)\in(S^1)^N$ and $(L_{N-1},\dots,L_1)\in([R,\infty[)^{N-1}$. The parameters $L_i$ are ``gluing parameters''. 

Note that $\cM(N,0)\simeq S^{2N-1}\setminus S^{2N-3}\simeq S^1\times\dots\times S^1\times\dot\Delta^{N-1}/\!\!\sim$. Here $\sim$ is the same equivalence relation as the one defining the join of $N$ circles, and $\dot\Delta^{N-1}$ is the simplex $\Delta^{N-1}$ with its first face removed. The parameter space $([R,\infty[)^{N-1}$ is diffeomorphic to a neighborhood of the face inside $\dot\Delta^{N-1}$, and we can extend this to a diffeomorphism $([0,\infty[)^{N-1}\simeq \dot\Delta^{N-1}$. We obtain in this way a diffeomorphism 
$\cM(N,0)\simeq (S^1)^N\times ([0,\infty[)^{N-1}/\sim$.

Similarly we define diffeomorphisms $\cM(k,j)\simeq (S^1)^{k-j}\times ([0,\infty[)^{k-j-1}/\sim$ for all $N\ge k>j\ge 0$. 

Let $\beta : \R \to [0, 1]$ be an increasing function  
such that $\beta(s) = 0$ for $s \le 0$ and $\beta(s) = 1$ for $s \ge 1$. We define 
$$
\rho_{N,0}=(t_0,\dots,t_N):\cM(N,0)\simeq (S^1)^N\times ([0,\infty[)^{N-1}/\sim\ \to\ \cF(\R,\Delta^N)
$$ 
by
$$
t_i(s):=\beta(s-L_{N-1}-\dots-L_{i+1})-\beta(s-L_{N-1}-\dots-L_i), \qquad i=N-1,N-2,\dots, 1
$$
and 
$$
t_0(s)=\beta(s-L_{N-1}-\dots-L_1), \qquad t_N(s)=1-\beta(s)=1-\sum_{i=0}^{N-1}t_i(s).
$$
The expanded formulas read 
$$t_N(s)=1-\beta(s),$$ 
$$t_{N-1}(s)=\beta(s)-\beta(s-L_{N-1}),$$ 
$$t_{N-2}(s)=\beta(s-L_{N-1})-\beta(s-L_{N-1}-L_{N-2}),$$ 
and so on up to 
$$t_1(s)=\beta(s-L_{N-1}-\dots-L_2)-\beta(s-L_{N-1}-\dots-L_1),$$ 
$$t_0(s)=\beta(s-L_{N-1}-\dots-L_1).$$

Similarly, for $k>j$ we define $\rho_{k,j}=(0,\dots,t_j,\dots,t_k,0,\dots):\cM(k,j)\simeq (S^1)^{k-j}\times ([0,\infty[)^{k-j-1}/\sim \ \to \ \cF(\R,\Delta^{k-j})$ by
$$t_k(s)=1-\beta(s),$$
$$t_{k-1}(s)=\beta(s)-\beta(s-L_{k-1}),$$
$$t_{k-2}(s)=\beta(s-L_{k-1})-\beta(s-L_{k-1}-L_{k-2}),$$
and so on up to 
$$t_{j+1}(s)=\beta(s-L_{k-1}-\dots-L_{j+2})-\beta(s-L_{k-1}-\dots-L_{j+1}),$$
$$t_j(s)=\beta(s-L_{k-1}-\dots-L_{j+1}).$$

The collection $\{\rho_{k,j}\}$ defines a $\cF(\R,\Delta)$-valued representation of $(\C P^N,f)$. 
And again, we have $(\rho_{k,j})_i(\lambda)=0$ whenever $t_i(\lambda)=0$ by definition. This representation is translation invariant by construction. 
}
\end{example}

\subsubsection{Reformulation of the $S^1$-equivariant complex using representations} \label{sec:reformulate}

We explain in this section how to transform the system~\eqref{eq:Floer-cont} into a single equation with parameters. 

Given $H\in\cH$ and a $\cF(\R,\Delta)$-valued representation $\rho=\{\rho_{k,j}\}$, we define for every $\lambda\in\cM(N,0)$ the Hamiltonian $H_\lambda:\R\times S^1\times \hat W\to \R$ via the formula 
$$
H_\lambda(s,\theta,x):=\sum_{j=0}^N (\rho_{N,0})_j(\lambda)(s)H^{\theta-\tau_j(\lambda)}(x),
$$
with the convention $\tau_N=0$.
Formally we have $H_\lambda=H_{N,0}\circ \rho(\lambda)$ with $H_{N,0}$ as in~\S\ref{sec:HN}. In the previous formula we use that the tuple $(\tau_N,\dots,\tau_0)$ is well-defined along a gradient trajectory of $\tilde f$, and hence if we fix $\tau_N=0$ the tuple $(\tau_{N-1},\dots,\tau_0)$ is well-defined along a gradient trajectory of $f$ on $\mathbb{C}P^N$. 
Note that we need to have $(\rho_{N,0})_j(\lambda)=0$ whenever $t_j(\lambda)=0$ in order for the formula to make sense. However, since $t_j(\lambda)\neq 0$ for all $j$ and all $\lambda\in\cM(N,0)$, this condition is tautologically satisfied.

Let $H\in\cH$. Choose $f$ of the form $[z]\mapsto C\sum_{j=0}^N (j+1)|z_j|^2/\sum_{j=0}^N |z_j|^2$ with $C>\|H\|$ (see~\S\ref{sec:simplifying} and Lemma~\ref{lem:HNf}). Fix  two \emph{translation invariant} representations of $(\C P^N,f)$ as follows: a representation $\rho=\{\rho_{k,j}\}$ which is $\cF(\R,\Delta)$-valued, and a representation $J=\{J_{k,j}\}$ which is $\cF(\R,\mathscr{J})$-valued and generic. 

Given $\og,\ug\in\cP(H)$ we define the moduli space 
$\cM^N(\og,\ug;H,J,\rho)$ as the set of tuples $(u,\lambda)$, with $\lambda\in\cM(N,0)$ and $u:\R\times S^1\to\hat W$ solving the equation 
$$
\p_s u + J_{N,0}^\lambda(s)(\theta)(u)\left(\p_\theta u - X _{H_\lambda}(s,\theta,u(s,\theta))\right) = 0
$$
subject to the asymptotic conditions 
$$
u(s,\cdot)\stackrel {s\to-\infty}\longrightarrow \og(\cdot), \qquad u(s,\cdot) \stackrel {s\to+\infty}\longrightarrow 
\ug(\cdot-\tau_0(\lambda)). 
$$

In the case $N=0$ the above equation does not depend on $s$ and we further divide out the $\R$-action on the $s$-variable. For a generic choice of the representation $J$ the space $\cM^N(\og,\ug;H,J,\rho)$ is a smooth manifold of dimension 
$$
|\og|-|\ug|+\dim\cM(N,0)=|\og|-|\ug|+2N-1.
$$

This defines for each $N\ge 0$ a map $\varphi_N:SC_*(H)\to SC_{*+2N-1}(H)$. Note that $\varphi_0$ is the differential in the Floer complex with respect to the almost complex structure $J_{0,0}\in\cJ$. The way in which we have set up the Floer moduli problem, using translation invariant representations which satisfy the Gluing Axiom in Definition~\ref{defi:representation}, implies that for all $k\ge 0$ we have the relation
$$
\sum_{i+j=k}\varphi_i\varphi_j=0.
$$
The collection $\{\varphi_i\}$, $i\ge 0$ defines therefore an $S^1$-structure on $SC_*(H)$, with associated $S^1$-equivariant homology groups denoted $SH_*^{S^1}(H,J,\rho)$. Transversality is ensured for generic data because the almost complex structure is allowed to depend on all the parameters $(\lambda,s,\theta)$. 

In our case we have considered representations with $\mathscr{X}=\Delta$ or $\mathscr{X}=\mathscr{J}$, so that the space of (translation invariant) representations is convex, respectively contractible. As a consequence, these homology groups are independent of the choice of $\rho$ and $J$. As a matter of fact, they coincide with the homology groups defined in~\S\ref{sec:simplifying}. This is seen as follows. Our function $f$ satisfies conditions~(\ref{item:1}--\ref{item:2}) in~\S\ref{sec:simplifying}, and our Hamiltonians $H_{N,0}$ satisfy condition~\eqref{item:3}. Similarly, the representation $J$ satisfies~\eqref{item:4} since it is assumed to be translation invariant. Taking as a $\cF(\R,\Delta)$-valued representation $\rho$ the one in Example~\ref{ex:1}, we see that our maps $\varphi_N$ tautologically coincide with the analogous maps defined in~\S\ref{sec:simplifying}. Thus the corresponding $S^1$-equivariant homology groups $SH_*^{S^1}(H)$ are the same. 

This construction is based on the definition of $H_{N,0}$, and the latter applies universally to all $H\in\cH$. As a consequence, the homology groups $SH_*^{S^1}(H)$ behave in the expected manner with respect to continuation morphisms induced by increasing homotopies of Hamiltonians in $\cH$. The limit over $H\in\cH$ is therefore isomorphic to $SH_*^{S^1}(W)$ as defined in the previous sections. 

A discussion analogous to the one in~\S\ref{sec:simplifying} applies in order to define $SH_*^{S^1,\pm}(W)$. 

\subsubsection{An explicit representation} \label{sec:explicit}

We spell out in this section the resulting construction in the case of Example~\ref{ex:2}. Figure~\ref{fig:param-ham} gives a hopefully useful, though schematic, illustration of our notation.

We fix an increasing cutoff function $\beta:\R\to[0,1]$ such that $\beta(s)=0$ for $s\le 0$ and $\beta(s)=1$ for $s\ge 1$. 

Let $H\in \cH$. Given an integer $N\ge 1$ and parameters $\utau=(\tau_0,\dots,\tau_{N-1})\in (S^1)^N$, $\uL=(L_1,\dots, L_{N-1})\in \left( [0,\infty[ \right)^{N-1}$, we denote $\tau_N:=0\in S^1$ and define
$$
H_{\utau,\uL}:\R\times S^1\times \hat W\to \R
$$ 
by
$$
H_{\utau,\uL}(s,\theta,x) := \sum_{i=0}^N t_i(s)H^{\theta-\tau_i}(x), 
$$
with 
$$
t_i(s):=\beta(s-L_{N-1}-\dots-L_{i+1})-\beta(s-L_{N-1}-\dots-L_i), \qquad i=N-1,N-2,\dots, 1
$$
and 
$$
t_0(s)=\beta(s-L_{N-1}-\dots-L_1), \qquad t_N(s)=1-\beta(s)=1-\sum_{i=0}^{N-1}t_i(s).
$$

We now need to rephrase the definition of a translation invariant $\cF(\R,\mathscr{J})$-valued representation in terms of families of almost complex structures parametrized by $(S^1)^N\times([0,\infty[)^{N-1}$. Restricting slightly the degree of generality that we adopted in Definition~\ref{defi:representation}, it will be enough for our purposes  to consider sequences $J_N=J_N^{\utau,\uL}(s,\theta)$ with $(\utau,\uL)\in(S^1)^N\times([0,\infty[)^{N-1}$ and $N\ge 0$ that satisfy the following inductive property:
\begin{itemize}
\item $J_0$ depends only on $\theta$;
\item $J_1$ satisfies $J_1^{\tau_0}(s,\theta)=J_0^\theta$ for $s\le 0$ and $J_1^{\tau_0}(s,\theta)=J_0^{\theta-\tau_0}$ for $s\ge 1$; 
\item $J_2$ satisfies $J_2^{(\tau_1,\tau_0),L_1}(s,\theta)=J_1^{\tau_1}(s,\theta)$ for $s\le L_1$ and $L_1$ large enough, as well as $J_2^{(\tau_1,\tau_0),L_1}(s,\theta)=J_1^{\tau_0-\tau_1}(s-L_1,\theta-\tau_1)$ for $s\ge 1$ and $L_1$ large enough;
\item $J_N$ satisfies  
$J_N^{(\tau_{N-1},\dots,\tau_0),(L_{N-1},\dots,L_1)}(s,\theta)=J_{N-1}^{(\tau_{N-1},\dots,\tau_1),(L_{N-1},\dots,L_2)}(s,\theta)$ for $s\le L_{N-1}+L_{N-2}+\dots+L_1$ and $L_1$ large enough, as well as  \break
$J_N^{\utau,\uL}(s,\theta)=J_{N-1}^{(\tau_{N-2}-\tau_{N-1},\dots,\tau_0-\tau_{N-1}),(L_{N-2},\dots,L_1)}(s-L_{N-1},\theta-\tau_{N-1})$ for $s\ge 1$ and $L_{N-1}$ large enough.
\end{itemize}
We denote by $\mathscr{R}ep\mathscr{J}$ the space of sequences $\{J_N\}$, $N\ge 0$ as above. 

\begin{figure}
         \begin{center}
\input{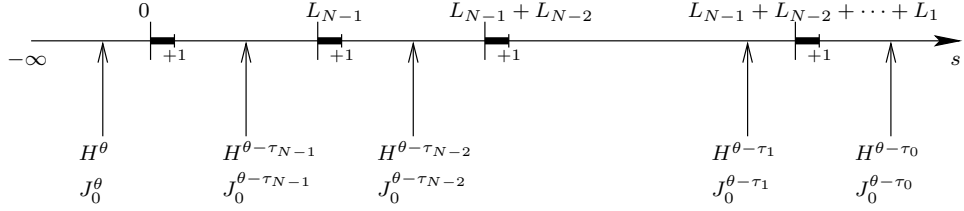}
         \end{center}
\caption{The almost complex structure $J_N$ for large values of $L_1,\dots,L_{N-1}$. The dependence on $s$ takes place only in the thickened regions. Similarly for the Hamiltonian $H_{\utau,\uL}$, with the difference that the picture is valid in the whole range $L_i\ge 1$.  \label{fig:param-ham}}
\end{figure}

Let $H\in\cH$ and $J=\{J_N\}\in\mathscr{R}ep\mathscr{J}$ be a generic choice of almost complex structures as above. Given $\og,\ug\in\cP(H)$,  denote $\cM^N(\og,\ug;H,J)$ the moduli space consisting of tuples $(u,\utau,\uL)$ such that $(\utau,\uL)\in(S^1)^N\times([0,\infty[)^{N-1}$ and $u:\R\times S^1\to\hat W$ solves the equation 
\begin{equation} \label{eq:utauL}
\p_s u + J_N^{\utau,\uL}(s,\theta,u(s,\theta))\left(\p_\theta u - X _{H_{\utau,\uL}}(s,\theta,u(s,\theta))\right) = 0
\end{equation}
subject to the asymptotic conditions 
$$
u(s,\cdot)\stackrel {s\to-\infty}\longrightarrow \og(\cdot), \qquad u(s,\cdot) \stackrel {s\to+\infty}\longrightarrow 
\ug(\cdot-\tau_0). 
$$
If $N=0$ then~\eqref{eq:utauL} is the usual Floer equation and does not depend on $s$, so that we further divide out the $\R$-action. If $N\ge 1$ the equation depends essentially on $s$. For a generic choice of $J$ the space $\cM^N(\og,\ug;H,J)$ is a smooth manifold of dimension 
$|\og|-|\ug|+2N-1$. In case the dimension is zero, one can associate using coherent orientations a sign $\eps(u)$ to each element $u$ of $\cM^N(\og,\ug;H,J)$.

Let $SC_*(H):=\oplus_{\gamma\in\cP(H)} \Z\langle\gamma\rangle$ be the Floer complex of $H$. 
The count of rigid elements in moduli spaces $\cM^N(\og,\ug;H,J)$ defines a map
$$
\varphi_N:SC_*(H)\to SC_{*+2N-1}(H),
$$
\begin{equation} \label{eq:phiNrep}
\varphi_N(\og):=\sum_{\substack{
  \ug\in\cP(H) \\ |\ug|=|\og|+2N-1}} \sum_{u\in\cM^N(\og,\ug;H,J)} \eps(u)\ug.
\end{equation}
Note that $\varphi_0:SC_*(H)\to SC_{*-1}(H)$ is
the differential in the Floer complex with respect to the almost complex structure $J_0$. The arguments in~\S\ref{sec:reformulate} prove the following.

\begin{proposition}
The maps $\{\varphi_i\}$, $i\ge 0$ define an $S^1$-structure, and the corresponding $S^1$-equivariant homology groups are isomorphic to $SH_*^{S^1}(H)$. \hfill{$\square$}
\end{proposition}

Passing to the limit over $H\in\cH$ we recover the groups $SH_*^{S^1}(W)$. The groups $SH_*^{S^1,\pm}(W)$ are defined using special Hamiltonians as in~\S\ref{sec:simplifying}.

\section{Linearized contact homology} \label{sec:CHlin}

We continue in this section with the notation of~\S\ref{sec:S1}. Unless otherwise specified $(W,\omega)$ denotes a Liouville domain with torsion first Chern class. The primitive of $\omega$ is denoted $\lambda$, its restriction to $M=\p W$ is denoted $\alpha$, and the primitive of $\hat\omega$ on the symplectic completion $\hat W$ is denoted $\hat\lambda$.
 
In~\S\ref{sec:CH-Ham} we recast linearized contact homology groups in a Hamiltonian formulation following~\cite{BOcont}. In~\S\ref{sec:S1-CH} we prove Theorem~\ref{thm:isomorphism}, stating that the linearized contact homology groups are functorially isomorphic over $\Q$ with the $S^1$-equivariant symplectic homology groups.
In~\S\ref{sec:CHlin-defi} we define filled and/or non-equivariant linearized contact homology groups. In contrast to~\S\ref{sec:S1-CH} which is algebraic in nature, we need geometric methods in~\S\ref{sec:CHlin-defi} in order to construct suitable chain homotopies.  

We work in this section under ideal transversality assumptions that guarantee well-definedness of linearized contact homology for particular choices of contact form and almost complex structure (see~\S\ref{sec:linconthom}). In particular, we do not appeal to abstract perturbations like the ones used in the polyfold theory of Hofer, Wysocki and Zehnder~\cite{Hofer-Polyfolds-survey,HWZ-polyfolds-I,HWZ-polyfolds-II,HWZ-polyfolds-III}. While our assumptions severely restrict the class of Liouville domains to which the current section applies, the fact that linearized contact homology is in general not a well-defined theory constitutes indeed one of the motivations for our Theorem~\ref{thm:isomorphism}. The latter can be read as an existence result: the positive part of $S^1$-equivariant symplectic homology is well-defined~\cite{BOtransv} and provides therefore an alternative to linearized contact homology within a ``classical'' transversality scheme.

The main result of this section is Proposition~\ref{prop:isomorphism}, which corresponds to Theorem~\ref{thm:isomorphism} in the Introduction. As already noted in the Introduction, its topological counterpart is Lemma~\ref{lem:Leray}, which we restate here for the reader's convenience. 

\noindent {\bf Lemma~\ref{lem:Leray}.}
{\it Let $X$ be a manifold and assume $Stab(x)$ is finite at every point $x\in X$. Then the canonical projection $X_{Borel}\to X/S^1$ induces an isomorphism in homology with $\Q$-coefficients. 
}

\begin{remark} {\it The assumption is equivalent to the fact that the $S^1$-action has no fixed points. The core of the proof is that each $Stab(x)$ is a finite cyclic group $\Z/k\Z$, $k\ge 1$, and the latter have vanishing rational homology in positive degrees. 
}
\end{remark}

\begin{proof}[Proof of Lemma~\ref{lem:Leray}]
It is enough to prove the same statement in cohomology with $\Q$-coefficients. We use the Leray spectral sequence for the map 
$$
pr_1:X_{Borel}\to X/S^1.
$$ 
This spectral sequence converges to $H^*(X_{Borel})$ and its second page is $E_2^{p,q}=H^p(X/S^1;\cH^q)$, with $\cH^q$ the sheaf generated by the presheaf $V\mapsto
 H^q(pr_1^{-1}(V);\Q)$ \cite[Thm.~4.17.1]{Godement}. The stalk $\cH^q_{[x]}$ at $[x]$ is $H^q(pr_1^{-1}([x]);\Q)=H^q(BStab(x);\Q)$. 
 The assumption that all stabilizers are finite implies $H^q(BStab(x);\Q)=0$ for $q>0$, and therefore $\cH^q=0$ for $q>0$. Thus
 $E_2^{p,q}$ vanishes for $q>0$, while for $q=0$ it
 is isomorphic to $H^p(X/S^1;\Q)$. The spectral
 sequence degenerates at $E_2$ for dimensional reasons, and the conclusion follows. 
\end{proof}

The content of Lemma~\ref{lem:Leray} is that, for all purposes cohomological, the map $X_{Borel}\to X/S^1$ that naturally fits into the diagram 

$$
\xymatrix
{ 
BStab(x)\ar[r] & X_{Borel} \ar[d]^{pr_1}  
\\
& [x]\in X/S^1 \qquad}
$$
behaves like a locally trivial fibration in case the action has no fixed points, provided one uses rational coefficients. The statement of Theorem~\ref{thm:isomorphism} is then motivated as follows. Both $SH_*^{+,S^1}(W)$ and $CH_*^{lin}(M)$ are constructed from closed characteristics on $M=\p W$, the first as an $S^1$-equivariant theory, the second as a quotient theory. Since the closed characteristics on $M$ are nonconstant, the action of $S^1$ by reparametrization has no fixed points and, in view of Lemma~\ref{lem:Leray}, the two homology groups should be isomorphic with $\Q$-coefficients. Moreover, the isomorphism should be given by a natural chain map analogous to the first projection $X_{Borel}\to X/S^1$. This is precisely what we will achieve within the next three sections. 

Let us make a final remark regarding the generality of the results that we prove in this section. In order to make our arguments conceptual we chose to work under the restrictive but general transversality assumptions stated below in~\S\ref{sec:linconthom}. However, the isomorphism between positive $S^1$-equivariant symplectic homology and linearized contact homology holds in more general situations, in which the transversality assumptions below are not satisfied but linearized contact homology is defined (i.e. the differential is defined and squares to zero) for specific choices of contact forms and almost complex structures. 
We know of two such instances: (i) the linearized contact complex is lacunary, as in Ustilovsky's work on Brieskorn spheres~\cite{Ustilovsky}; (ii) the linearized contact complex can be identified with a Morse complex, as in Mei-Lin Yau's work on boundaries of subcritical Stein manifolds~\cite{MLYau}. We refer to~\S\ref{sec:applications} for a more detailed discussion of these examples.

\subsection{Hamiltonian formulation of linearized contact homology} \label{sec:CH-Ham}

We summarize in this section some definitions and results from~\cite{BOcont}, and explain in particular how to recast linearized contact homology in Hamiltonian terms.

\subsubsection{Definition of linearized contact homology} \label{sec:linconthom} We follow~\cite[\S3.1]{BOcont}. Let $(W, \omega=d\lambda)$ be a Liouville domain with boundary $(M, \xi=\ker\,\alpha)$, $\alpha:=\lambda|_M$ and with torsion first Chern class. We assume that $\alpha$ is nondegenerate in the sense that all its closed Reeb orbits are transversally nondegenerate. We denote the set of equivalence classes of closed Reeb orbits modulo reparametrization which belong to a given free homotopy class $c$ in $\hat W$ by $\cP^c(\alpha)$. We use the generic notation $\gamma'$ for elements of $\cP^c(\alpha)$ and we denote by $0$ the trivial free homotopy class. 

Let $c$ be a fixed free homotopy class in $\hat W$. To define linearized contact homology in the class $c$ we assume the existence of an almost complex structure $J\in\cJ$ (cf.~\S\ref{sec:S1-Borel}) that satisfies the following two regularity conditions:
\begin{enumerate} \label{pagenumber:transv}
\item[($A$)] $J$ is regular for holomorphic planes in $\hat W$ asymptotic to some $\gamma'\in\cP^0(\alpha)$ and which belong to moduli spaces -- denoted $\cM(\gamma',\emptyset;J)$ -- of virtual dimension $\le 1$;
\item[($B_c$)] let $J_\infty$ be the almost complex structure on the symplectization $(\R_+^*\times M, d(r\alpha))$ which is determined by the restriction of $J$ to the complement of a large enough compact set inside $\hat W$. One requires $J_\infty$ to be regular for punctured holomorphic cylinders which belong to moduli spaces of virtual dimension $\le 2$, asymptotic at $\pm\infty$ to arbitrary elements of $\cP^c(\alpha)$ and asymptotic at the punctures to elements $\gamma'\in\cP^0(\alpha)$ such that there exists a $J$-holomorphic building of type $0|1|k_+$, $k_+\ge 0$ in the sense of~\cite[\S8.1]{BEHWZ} with exactly one positive puncture and asymptote $\gamma'$. 
\end{enumerate}
These conditions are identical to the conditions formulated in~\cite{BOcont-err}, which correct the ones that we had previously formulated in~\cite[Remark~9]{BOcont}. That the latter were insufficient for linearized contact homology to be defined was pointed out to us independently by Janko Latschev and Mohammed Abouzaid. We refer to~\cite{BOcont-err} for a detailed discussion. 

The \emph{linearized contact homology group in the free homotopy class $c$} is a $\Z$-graded $\Q$-vector space denoted $CH_*^{lin;c}(M,\alpha,J)$ and defined as follows (see~\cite[\S3]{BOcont}). Given a trivialization of $T\hat W$ along a simple closed Reeb orbit $\gamma'$, the transverse Conley-Zehnder indices of the iterates of $\gamma'$ either all have the same parity, or the parity differs between the odd and the even iterates~\cite[\S5]{Bourgeois-Mohnke}, \cite[\S4.4]{BOauto}. In the second case, the even iterates are called \emph{bad}. The closed Reeb orbits which are not bad are called \emph{good}. The group $CH_*^{lin;c}(M,\alpha,J)$ is the homology of a complex $(CC_*^{lin;c}(\alpha,J),\p)$ whose underlying graded $\Q$-vector space is generated by good elements of $\cP^c(\alpha)$. The grading is defined by first choosing a reference loop in the free homotopy class $c$ and a trivialization of $T\hat W$ along that loop, and then associating to each element $\gamma'\in\cP^c(\alpha)$ its Conley-Zehnder index $\mu(\gamma')\in\Z$ with respect to the trivialization induced by a homotopy to the reference loop~\cite[p.~626]{BOcont}. 

To define the differential on $CC_*^{lin;c}(\alpha,J)$ we proceed as follows. We fix a reference point on the geometric image of each closed Reeb orbit. Given $\og',\ug'\in\cP^c(\alpha)$ and $\gamma'_1,\dots,\gamma'_k\in\cP^0(\alpha)$, we denote by $\cM(\og',\ug',\gamma'_1,\dots,\gamma'_k;J)$ the moduli space of punctured $J_\infty$-holomorphic cylinders endowed with asymptotic tangent directions at the punctures, with asymptotes $\og',\ug',\gamma'_1,\dots,\gamma'_k$ and converging to the reference points along the corresponding tangent directions. Note that, unless $\og'=\ug'$ and $k=0$, the multiplicative group $\R_+^*$ acts freely on $\cM(\og',\ug',\gamma'_1,\dots,\gamma'_k;J)$ by homotheties in the target. The moduli space $\cM(\og',\ug',\gamma'_1,\dots,\gamma'_k;J)$ has virtual dimension $\mu(\og')-\mu(\ug')-\sum_{i=1}^k\mu(\gamma'_i)-k(n-3)$. Similarly, the virtual dimension of the moduli space $\cM(\gamma',\emptyset;J)$ of $J$-holomorphic planes asymptotic to $\gamma'\in\cP^0(\alpha)$ and converging to the reference point along the asymptotic tangent direction at infinity is $\mu(\gamma')+n-3$. The dimension formulas take a simpler form if we define  $\omu(\gamma'):=\mu(\gamma')+n-3$. 

Denote by $\kappa_{\gamma'}\in\N^*$ the multiplicity of a closed Reeb orbit representing $\gamma'$. Let $n(\og',\ug',\gamma'_1,\dots,\gamma'_k)$ be the signed count of elements in the quotient moduli space 
$\cM(\og',\ug',\gamma'_1,\dots,\gamma'_k;J)/\R_+^*$ for $\mu(\og')-\mu(\ug')-\sum\omu(\gamma'_i)=1$. Define a linear map $e:CC_*^{lin;c}(\alpha,J)\to\Z$ by assigning to each $\gamma'$ the signed count of elements in $\cM(\gamma',\emptyset;J)$ provided $\omu(\gamma')=0$, and zero otherwise. Define the \emph{linearized contact differential}Ê\ $\p:CC_*^{lin;c}(\alpha,J)\to CC_{*-1}^{lin;c}(\alpha,J)$ on the generators by the formula 
$$
\p\og'=\sum_{\scriptsize\begin{array}{c}\ug',\gamma'_1,\dots,\gamma'_k\\ \mu(\og')-\mu(\ug')-\sum\omu(\gamma'_i)=1\end{array}} \frac{n(\og',\ug',\gamma'_1,\dots,\gamma'_k)}{\kappa_{\ug'}\prod_{i=1}^k\kappa_{\gamma'_i}} e(\gamma'_1)\dots e(\gamma'_k)\ug'.
$$ 

Since we use asymptotic markers and work under ideal transversality assumptions, the definition of the differential 
is such that the linear sum that expresses it has integer coefficients. However, it is expected that linearized contact homology is well-defined in general only with $\Q$-coefficients because one needs to use multivalued perturbations. In order to emphasize this expectation we choose to use $\Q$-coefficients from the start. 

We denote $CH_*^{lin;c}(M,\alpha,J):=H_*(CC_*^{lin;c}(\alpha,J),\p)$ and $CH_*^{lin;c;\le a}(M,\alpha,J):=H_*(CC_*^{lin;c;\le a}(\alpha,J),\p)$, where $CC_*^{lin;c;\le a}(\alpha,J)$ is the subcomplex generated by elements of $\cP^c(\alpha)$ with action $\le a$. The polyfold theory of Hofer, Wysocki, Zehnder~\cite{Hofer-Polyfolds-survey,HWZ-polyfolds-I,HWZ-polyfolds-II,HWZ-polyfolds-III} should eventually prove that $CH_*^{lin;c}(M,\alpha,J)$ does not depend neither on the choice of $J$, nor on the choice of $\alpha$, whereas $CH_*^{lin;c;\le a}(M,\alpha,J)$ does not depend on the choice of $J$. To simplify the notation we denote in the sequel these homology groups by $CH_*^{lin;c}(M)$, respectively by $CH_*^{lin;c;\le a}(M,\alpha)$.

\smallskip 

\noindent {\bf Notational convention.}Ê\ \emph{In the sequel we do not specify anymore the free homotopy class $c$ in the notation for the linearized contact homology groups. It is understood that we fix a collection $\cC$ of such classes such that $CH_*^{lin;c}(M,\alpha,J)$ is defined for all $c\in \cC$ and for some common $J\in\cJ$, and we work with the direct sum $\oplus_{c\in\cC} CH_*^{lin;c}(M,\alpha,J)$.}

\subsubsection{Hamiltonian formulation}\label{sec:Hamiltformul}
 Recall the space $\cH$ of Hamiltonians in~\S\ref{sec:S1-Borel}. 
For each $a>0$ not in $\mathrm{Spec}(M,\alpha)$ we define $\cH_0^a\subset\cH$ to consist of Hamiltonians $H\in\cH$ which satisfy the following properties:
\begin{itemize}
\item $H$ is time-independent;
\item $H$ is a negative Morse function on $W$ such that all its $1$-periodic orbits starting in $W$ are constant and have action smaller than the smallest action of a closed Reeb orbit on $M=\p W$; 
\item $H=h(r)$ for $r\ge 1$, with $h:[1,\infty[\to\R$ a smooth function such that $h'>0$, $h''\ge 0$, $h(r)=ar+b$ for $r$ large enough, and $h''(r)>0$ if $h'(r)\in\mathrm{Spec}(M,\alpha)$.  
\end{itemize}
The $1$-periodic orbits of a Hamiltonian $H\in\cH_0^a$ are either nondegenerate critical points of $H$ inside $W$, or transversally nondegenerate $1$-periodic orbits on levels $r=const$ such that $h'(r)\in\mathrm{Spec}(M,\alpha)$. We denote by $\cP(H)$ the set of all $1$-periodic orbits, and we denote by $\cP^+(H)$ the set of all nonconstant $1$-periodic orbits, so that $\cP(H)=\cP^+(H)\cup \mathrm{Crit}(H)$. Since $X_h=-h'(r)R_\alpha$, the elements of $\cP^+(H)$ are in one-to-one bijective correspondence with closed Reeb orbits of period $<a$. Any $\gamma\in\cP^+(H)$ determines by reparametrization a Morse-Bott nondegenerate circle of $1$-periodic orbits, which we denote $S_\gamma$. 

Given $J\in\cJ$, $\og,\ug\in\cP^+(H)$ and $p\in\mathrm{Crit}(H)$ we denote by $\cM(S_\og,S_\ug;H,J)$, respectively by $\cM(S_\og,p;H,J)$ the moduli space of solutions $u:\R\times S^1\to \hat W$ of Floer's equation $\p_s u+J(\p_\theta u-X_H)=0$ such that $\lim_{s\to-\infty}u(s,\cdot)\in S_\og$ and $\lim_{s\to\infty}u(s,\cdot)\in S_\ug$, respectively $\lim_{s\to\infty}u(s,\cdot)=p$. 

Let $c$ be a free homotopy class of loops in $\hat W$ for which the transversality assumptions in~\S\ref{sec:linconthom} hold. Then the proofs of~\cite[Lemma~2 and Proposition~3]{BOcont} adapt to show that for each $a>0$ there exists $H\in\cH_0^a$ such that the almost complex structure $J$ is regular for all solutions of the Floer equation for $(H,J)$ that belong to moduli spaces of virtual dimension $\le 2$ with asymptotes in the class $c$. Moreover, for such an $H$ the function $h$ can be chosen such that $h''\ge 0$ is arbitrarily small. In particular, for any $a<a'$ one finds such Hamiltonians $H^a\in\cH_0^a$ and $H^{a'}\in\cH_0^{a'}$ with $H^a<H^{a'}$ by first picking $H^{a'}$ and then choosing $H^a$.

\noindent {\bf Notational convention.} {\it In the sequel we work in such a fixed free homotopy class $c$ (or collection $\cC$ of free homotopy classes). However, for readability we do not specify this data in the notation.}

To associate an index to the elements of $\cP^+(H)$, we recall our choice of a reference loop in each free homotopy class of loops in $\hat W$ and of trivializations of $T\hat W$ along these reference loops. For the class of contractible loops we choose a constant loop and a constant trivialization. Given $\gamma\in\cP^+(H)$ with underlying closed Reeb orbit $\gamma'$, the Robbin-Salamon index of the linearized Hamiltonian flow along $\gamma$, denoted $\mu_{RS}(\gamma)$, is a half-integer equal to $-\mu(\gamma')-\frac 1 2$~\cite[Lemma~3.4]{BOauto}. We define the grading 
$$
|\gamma|:=\mu(S_\gamma):=\mu(\gamma').
$$ 
The Conley-Zehnder index of the linearized Hamiltonian flow at $p\in\mathrm{Crit}(H)$, denoted $\mu(p)$, is equal to $\ind(p;-H)-n$, where $\ind(p;-H)$ denotes the Morse index of $p$ with respect to $-H$~\cite[Lemma~7.2]{SZ92}. 
The virtual dimensions of the moduli spaces $\cM(S_\og,S_\ug;H,J)$ and $\cM(S_\og,p;H,J)$ are then expressed by the formulas~\cite[Proposition~3.5]{BOauto}
$$
\dim\,\cM(S_\og,S_\ug;H,J)=\mu(S_\og)-\mu(S_\ug), \qquad \dim\,\cM(S_\og,p;H,J)=\mu(S_\og)-\mu(p).
$$

In case $S_\og\neq S_\ug$ the moduli spaces $\cM(S_\og,S_\ug;H,J)$ carry the canonical $S^1$-action $(\tau\cdot u)(s,\cdot):=u(s,\cdot-\tau)$, $\tau\in S^1$. We define the chain complex 
$$
SC_k^{+,inv}(H,J)=\bigoplus_{\scriptsize\begin{array}{c} \gamma\in\cP^+(H) \\ \gamma' \mbox{ good and }Ê\mu(S_\gamma)=k \end{array}}\Z\langle S_\gamma\rangle
$$
with differential $\p:SC_*^{+,inv}(H,J)\to SC_{*-1}^{+,inv}(H,J)$ given by 
\begin{equation} \label{eq:differentialSH+inv}
\p \langle S_\og \rangle:=\sum_{\scriptsize\begin{array}{c} \ug' \mbox{ good} \\ \mu(S_\og)=\mu(S_\ug)+1\end{array}} \kappa_{\og'} \#\cM(S_{\og},S_\ug;H,J)/S^1\langle S_\ug\rangle.
\end{equation}
Here we denoted by $\#\cM(S_\og,S_\ug;H,J)/S^1$ the signed count of elements in the $0$-dimensional quotient of the moduli space by the $S^1$-action, with signs being determined by the coherent orientation scheme in~\cite[\S4.4]{BOauto}. Our definition of $\p$ is such that it coincides with $\bar d^1$ on the first line of the first page of the spectral sequence $(E^r_d,\bar d^r)$ of~\cite[Corollary~1]{BOcont}. In particular, it satisfies $\p^2=0$ and we denote 
$$
SH_*^{+,inv}(H,J):=H_*(SC_*^{+,inv}(H,J),\p).
$$
Given $H^a, K^a \in\cH_0^a$ such that $H^a\ge  K^a$ and $J$ is regular for both $H^a$ and $K^a$ as above, there is a well-defined continuation map $\sigma_{K^aH^a}:SH_*^{+,inv}(K^a,J)\to SH_*^{+,inv}(H^a,J)$ induced by an autonomous increasing homotopy~\cite[Remark~15]{BOauto}, and this continuation map 
is an isomorphism. As a consequence, given $H^a\in\cH_0^a$ and $H^{a'}\in\cH_0^{a'}$ with $a<a'$ and such that $J$ is regular for both of them as above -- so that $H^a<H^{a'}$ holds \emph{a priori} only in a neighborhood of infinity --, we can define a \emph{generalized continuation map} 
$\bar\sigma_{H^aH^{a'}}:SH_*^{+,inv}(H^a,J)\to SH_*^{+,inv}(H^{a'},J)$ by choosing $K^a\in \cH_0^a$ such that $K^a\le H^a$ and $K^a \le H^{a'}$ and setting $\bar\sigma_{H^aH^{a'}}:=\sigma_{K^aH^{a'}}\circ \sigma_{K^aH^a}^{-1}$. We define then 
$$
SH_*^{+,inv}(\alpha,J):=\lim_{\stackrel\longrightarrow a} SH_*^{+,inv}(H^a,J), 
$$
the direct limit as $a\to\infty$ being taken with respect to the generalized continuation maps defined above. The independence of $SH_*^{+,inv}(\alpha,J)$ on the choice of $\alpha,J$ is of course subject to the same caveats as in the case of linearized contact homology and would be a consequence of polyfold theory~\cite{Hofer-Polyfolds-survey,HWZ-polyfolds-I,HWZ-polyfolds-II,HWZ-polyfolds-III}. To simplify the notation we denote in the sequel these homology groups $SH_*^{+,inv}(W)$ and call them \emph{invariant symplectic homology truncated to positive values of the action} or, for short, \emph{positive invariant symplectic homology}. 

Corollary~1 in~\cite{BOcont} ensures that given $H^a\in\cH_0^a$ such that $J$ is regular for $H^a$ as above, we have an isomorphism 
\begin{equation*} 
I_{H^a}:CH_*^{lin;\le a}(\alpha,J)\stackrel \sim \longrightarrow SH_*^{+,inv}(H^a,J).
\end{equation*}

For the reader's convenience, we recall now from~\cite[\S6]{BOcont} the definition of $I_{H^a}$ and the reason why it is an isomorphism. The map $I_{H^a}$ is defined at chain level by a signed count of elements in moduli spaces 
\begin{equation} \label{eq:MI}
\cM_I(\og',S_\ug), \qquad \og'\in\cP(\alpha),\ \ug\in\cP^+(H^a)
\end{equation}
that are defined as follows. We consider an increasing map $\rho:\R\to[0,1]$ such that $\rho(s)=0$ for $s\ll 0$ and $\rho(s)=1$ for $s\gg 0$. We denote $H^{a;\rho}(s,x)=\rho(s) H^a(x)$. We then define $\cM_I(\og',S_\ug)$ to consist of maps $u:\R\times S^1\to \hat W$ solving $\p_s u +J(\p_\theta u-X_{H^{a;\rho}}(u))=0$, such that $u$ is asymptotic at $-\infty$ to $\og'$ and converges to the reference point of $\og'$ along the tangent direction $\R\times \{0\}$, and such that $u(+\infty,\cdot)\in S_\ug$. The virtual dimension of $\cM_I(\og',S_\ug)$ is equal to $\mu(\og')-|\ug|$. A standard argument shows that we thus define a chain map. Note that in~\cite[\S6]{BOcont} the map $I_{H^a}$ is defined after having performed a preliminary neck-stretching along the boundary, so that the definition involves moduli spaces of capped punctured holomorphic curves in the symplectization that are analogous to the moduli spaces $\cM_I(\og',S_\ug)$ above. In the current setup we do not stretch the neck so that the moduli spaces that define $I_{H^a}$ do not involve additional punctures. 

The map $I_{H^a}$ is seen to be an isomorphism at chain level as follows. The elements of $\cP(\alpha)$ with action $\le a$ are naturally in one-to-one bijective correspondence with the elements of the set $\cP^+(H^a)/S^1=\{S_\gamma\, : \, \gamma\in \cP^+(H^a)\}$ (we recall that we work in some collection $\mathcal{C}$ of free homotopy classes which we omit from the notation). These elements provide natural bases for the chain complexes $CC_*^{lin;\le a}(\alpha,J)$ and $SC_*^{+,inv}(H^a,J)$ and one can easily see that $I_{H^a}$ decreases the action. Thus, if we order the generators by the action and in such a way that the above correspondence is respected, the map $I_{H^a}$ is represented at chain level by an upper triangular matrix. Finally, one sees by an explicit argument that the diagonal entries in this matrix are $\pm 1$, so that $I_{H^a}$ is an isomorphism at chain level.

We showed in~\cite[Proposition~7]{BOcont} that the maps $I_{H^a}$ are compatible with the natural continuation morphisms given by enlarging the action window on the contact side, respectively given by increasing homotopies on the Hamiltonian side. This implies in particular that we have an isomorphism 
$$
I:CH_*^{lin}(\alpha,J)\stackrel\sim\longrightarrow SH_*^{+,inv}(\alpha,J).
$$
In a different notation, we have an isomorphism $I:CH_*^{lin}(M)\stackrel\sim\longrightarrow SH_*^{+,inv}(W)$. 
This completes our alternative description of the linearized contact homology groups.

\begin{remark} \label{rmk:Theta} {\it Instead of using the expression~\eqref{eq:differentialSH+inv}, one can define a differential $\p'$ on $SC_*^{+,inv}(H,J)$ by the formula  
$$
\p' \langle S_\og \rangle := \sum_{\scriptsize\begin{array}{c} \ug' \mbox{ good} \\ \mu(\og')=\mu(\ug')+1\end{array}} \kappa_{\ug'} \#\cM(S_{\og},S_\ug;H,J)/S^1\langle S_\ug\rangle.
$$
The definition of $\p'$ is such that it coincides with the expression of $\bar d^1$ on the second line of the first page of the spectral sequence $(E^r_d,\bar d^r)$ of~\cite[Corollary~1]{BOcont}, so that $(\p')^2=0$. The differentials $\p$ and $\p'$ are in general not equal, yet they are conjugated via the automorphism $\Theta$ of $SC_*^{+,inv}(H,J)$ given by $S_\gamma\mapsto \frac 1 {\kappa_\gamma} S_\gamma$, so that $\p'=\Theta^{-1}\circ \p \circ \Theta$.  The automorphism $\Theta$ is only defined over $\Q$ and the resulting homology groups are isomorphic over $\Q$.}
\end{remark}

\subsubsection{Gysin exact sequence for positive invariant symplectic homology} \label{sec:GysinSH+inv}

The long exact sequence of~\cite[Theorem~1]{BOcont} relating $SH_*^+(W)$ and $CH_*^{lin}(M)$ takes the form of an exact triangle (recall that we work with some collection $\mathcal{C}$ of free homotopy classes in $W$ that we omit from the notation)
\begin{equation} \label{eq:lesGysinSH+inv}
\xymatrix
@C=20pt
{
SH_*^+(W) \ar[rr] & & 
SH_*^{+,inv}(W) \ar[dl]^D \\ & SH_{*-2}^{+,inv}(W) \ar[ul]^{[+1]}  
}
\end{equation}

Note that, because of the fact that the conjugating automorphism $\Theta$ of Remark~\ref{rmk:Theta} is only defined over $\Q$, this exact triangle, just like the one in~\cite[Theorem~1]{BOcont}, is also only defined over $\Q$. For the purposes of the present paper the explicit expression of $D$ is not important and the only thing that will matter is the formal structure of the complex $SC_*^+(H,J)$ that computes $SH_*^+(H,J)$ for an autonomous Hamiltonian $H\in\cH_0^a$. More precisely, the latter has the form
\begin{equation} \label{eq:SC*+HJ}
SC_*^+(H,J):=\bigoplus_{\scriptsizeÊS_\gamma\subset \cP^+ (H)} \langle \gamma_m,\gamma_M\rangle, 
\end{equation}
where $\gamma_m,\gamma_M\in S_\gamma$ are two $1$-periodic orbits which we view as the two critical points -- minimum $m$ and maximum $M$ -- of a perfect Morse function $f_\gamma:S_\gamma\to \R$. Their gradings are respectively $\mu(\gamma_M):=\mu(\gamma)$ and $\mu(\gamma_m):=\mu(\gamma)+1$. The complex is filtered by $\mu(\gamma)$ and, with respect to this filtration, the differential $d$ on $SC_*^+(H,J)$ can be written 
$$
d=d^0+d^1+d^2.
$$
The various components $d^i$, $i=0,1,2$ act as follows (in the remainder of this section we write $SC_*^+$ instead of $SC_*^+(H,J)$ for readability): 
\begin{itemize}
\item the component $d^0$ satisfies 
$$
d^0(\gamma_M)=0,\qquad d^0(\gamma_m)=\left\{\begin{array}{rl}Ê0, & \mbox{if } \gamma \mbox{ is good,}Ê\\Ê\pm 2 \gamma_M, & \mbox{if } \gamma \mbox{ is bad}. \end{array}\right. 
$$ 
\item denote $Cm_*:=\bigoplus \langle \gamma_m\rangle$, $CM_*:=\bigoplus\langle \gamma_M\rangle$, and denote by $Cm_*^{bad/good}$, $CM_*^{bad/good}$ the respective subspaces generated by $\gamma_m$, $\gamma_M$ with $\gamma$ bad/good. Then $d^1(CM_*^{bad})=0$, whereas $d^1(Cm_*)\subset Cm_*^{good}$. Moreover, through the obvious identifications $CM_*/CM_*^{bad}\simeq SC_*^{+,inv}$ and $Cm_*^{good}\simeq SC_{*-1}^{+,inv}=:SC_*^{+,inv}[-1]$ we have (see also Remark~\ref{rmk:Theta})
$$
d^1|_{CM_*/CM_*^{bad}}\simeq \p, \qquad d^1|_{Cm_*^{good}}\simeq \p'=\Theta^{-1}\circ\p\circ\Theta.
$$
\item the component $d^2$ vanishes on $Cm_*$ and $d^2(CM_*)\subset Cm_*$.
\end{itemize}

It is proved in~\cite[\S7.1]{BOcont} that the spectral sequence associated to the filtration by $\mu(\gamma)$ on $SC_*^+$ is supported on two lines and degenerates at the third page for dimensional reasons. Modulo taking into account the isomorphism $\Theta$, the resulting exact triangle is precisely~\eqref{eq:lesGysinSH+inv} and the map $D$ is induced by $d^2$ and $d^1$. Moreover, this exact triangle is isomorphic to~\eqref{eq:Gysin-contact}.

\subsection{Isomorphism with $S^1$-equivariant symplectic homology}  
\label{sec:S1-CH}
What we actually prove in this section is that positive $S^1$-equivariant symplectic homology is functorially isomorphic over $\Q$ to positive invariant symplectic homology. Here ``functoriality'' is understood with respect to the Gysin exact triangles~\eqref{eq:Gysintriangle} and~\eqref{eq:lesGysinSH+inv} and with respect to the Viterbo transfer map. 

Let $H\in\cH_0^a$ and $J\in\cJ$ be such that the pair $(H,J)$ is regular as in~\S\ref{sec:CH-Ham} (for some fixed free homotopy class). In particular, the complex $SC_*^+(H,J)$ has the form~\eqref{eq:SC*+HJ}. Recall also the definition of the $S^1$-equivariant complex in~\S\ref{sec:simplifying}
$$SC_*^{+,S^1}(H,J)=\Z[u]\otimes_\Z SC_*^+(H,J),$$ 
whose differential acts by $\p^{S^1}(u^\ell \gamma_p)=\sum_{N=0}^\ell u^{\ell-N}\varphi_N(\gamma_p)$ with $\varphi_N:SC_*^+(H,J)\to SC_{*+2N-1}^+(H,J)$ and $\varphi_0=d$, the differential on $SC_*^+(H,J)$.

We use the language of~\S\ref{sec:explicit} to describe the maps $\varphi_N$, $N\ge 1$. Given a small time-dependent perturbation $K$ of the Hamiltonian $H$ we have defined the Hamiltonian $K_{\utau,\uL}(s,\theta,x):=\sum_{i=0}^Nt_i(s) K^{\theta-\tau_i}(x)$. Let also $J_N^{\utau,\uL}$, $(\utau,\uL)\in(S^1)^N\times ([0,\infty[)^{N-1}$ be a generic family  of $(s,\theta)$-dependent almost complex structures satisfying for large enough values of $s,L_1,\dots,L_{N-1}$ the compatibility conditions listed in~\S\ref{sec:explicit}. The maps $\varphi_N$ are defined by counting with signs the elements of the moduli spaces $\cM^N(\og,\ug;K,J_N)$, $\og,\ug\in\cP(K)$ defined by~\eqref{eq:utauL}. 

If the perturbation $K$ is of Morse-Bott type as in~\cite[\S3]{BOauto}\, these moduli spaces admit the following alternative description similar to that of ``moduli spaces of $s$-dependent Morse-Bott broken trajectories'' in~\cite[p.~624]{BOcont}. For each $S_\gamma\subset \cP^+(H)$ we choose a perfect Morse function $f_\gamma:S_\gamma\to\R$ and denote $\varphi_{f_\gamma}$ the positive gradient flow of $f_\gamma$. (Since we use \emph{positive} gradient flows, minima are understood to have index $1$ and maxima are understood to have index $0$.) Given $\og,\ug\in\cP^+(H)$, $p\in\mathrm{Crit}(f_\og)$, $q\in\mathrm{Crit}(f_\ug)$, $\ell\ge 1$, and $i\in\{1,\dots,\ell\}$ we define 
$$
\cM^N_\ell(p,q;H,\{f_\gamma\},J_N)
$$
to consist of the disjoint union over $\tilde\gamma_1,\dots,\tilde\gamma_{\ell-1}\in\cP^+(H)$ of tuples $(\u,\utau,\uL)$ such that $(\utau,\uL)\in(S^1)^N\times([0,\infty[)^{N-1}$ and $\u$ is an element of the fiber product (with the convention $\tilde\gamma_0:=\og$, $\tilde\gamma_\ell:=\ug$)
\begin{eqnarray*} 
&& 
W^u(p)  
\times_{\oev} 
(\cM(S_{\tgamma_0},S_{\tgamma_1};H,J_0^\theta)\!\times\!\R^+) \\ 
&& {_{\varphi_{f_{\tgamma_1}}\!\circ\uev}}\!\times _{\oev}  
\ldots {_{\varphi_{f_{\tgamma_{i-2}}}\!\circ\uev}}\!\times _{\oev}  
(\cM(S_{\tgamma_{i-2}},S_{\tgamma_{i-1}};H,J_0^\theta)\!\times\!\R^+) \\ 
&&  
{_{\varphi_{f_{\tgamma_{i-1}}}\!\circ\uev}\times_{\oev}}  
(\cM^N(S_{\tgamma_{i-1}},S_{\tgamma_i};H,J_N)\!\times\!\R^+) \\ 
&&{_{\varphi_{f_{\tgamma_i}}\!\circ\uev}\times_{\oev}}  
(\cM(S_{\tgamma_i},S_{\tgamma_{i+1}};H,J_0^{\theta-\tau_0})\!\times\!\R^+) \\ 
&&  
{_{\varphi_{f_{\tgamma_{i+1}}}\!\circ\uev}\times_{\oev}} \ldots\, 
{_{\varphi_{f_{\tgamma_{\ell-1}}}\!\!\circ\uev}}\!\!\times 
_{\oev}   
\cM(S_{\tgamma_{\ell-1}},\!S_{\tgamma_\ell};H,J_0^{\theta-\tau_0})  
{_{\uev}\times} W^s(q). 
\end{eqnarray*} 
Here $\cM^N(S_{\tgamma_{i-1}},S_{\tgamma_i};H,J_N)$ is the moduli space consisting of tuples $(u,\utau,\uL)$ such that $(\utau,\uL)\in(S^1)^N\times([0,\infty[)^{N-1}$ and $u:\R\times S^1\to \hat W$ solves the 
equation 
$\p_s u+J_N^{\utau,\uL}(s,\theta,u(s,\theta))(\p_\theta u-X_H(u(s,\theta))=0$ and is subject to the asymptotic conditions 
$$
\lim_{s\to-\infty}u(s,\cdot)\in S_{\tgamma_{i-1}}, \qquad \lim_{s\to\infty}u(s,\cdot)\in S_{\tgamma_i}.
$$
For a generic choice of the family $J_N$ this moduli space is a smooth manifold of dimension $\dim\,\cM^N(S_{\tgamma_{i-1}},S_{\tgamma_i};H,J_N)=|\tgamma_{i-1}|-|\tgamma_i|+2N-1$. 

One shows as in~\cite[Lemma~3.6]{BOauto} that, for a generic choice of 
the collection of Morse functions $\{f_\gamma\}$, each space  
$\cM^N_\ell(p,q;H,\{f_\gamma\},J)$ is a smooth manifold, and its dimension is 
\begin{eqnarray} 
\lefteqn{\dim \, \cM^N_\ell(p,q;H,\{f_\gamma\},J_N)} \nonumber \\ 
& = & |\og| +\ind(p) - |\ug| - 
\ind(q) +2N-1. \label{eq:dimension-formula}
\end{eqnarray} 
We denote 
$$\cM^N(p,q;H,\{f_\gamma\},J_N)$$
the union over $\ell\ge 1$ of the spaces 
$\cM^N_\ell(p,q;H,\{f_\gamma\},J_N)$. The arguments of~\cite{BOauto} carry over in order to show that the elements of the $0$-dimensional moduli spaces $\cM^N(p,q;H,\{f_\gamma\},J_N)$ are in one-to-one bijective correspondence preserving signs with the elements of the $0$-dimensional moduli spaces $\cM^N(\og_p,\ug_q;K,J_N)$, where $\og_p,\ug_q\in\cP^+(K)$ are the periodic orbits that correspond to the critical points $p\in\mathrm{Crit}(f_\og)$ and $q\in\mathrm{Crit}(f_\ug)$. We denote the sign of an element $(\u,\utau,\uL)\in\cM^N(p,q;H,\{f_\gamma\},J_N)$ by $\eps(\u,\utau,\uL)$. We thus obtain 
$$
\varphi_N(\og_p)=\sum_{\scriptsize \begin{array}{c} \ug_q \\ |\ug_q|=|\og_p|+2N-1\end{array}} \sum_{(\u,\utau,\uL)\in\cM^N(p,q;H,\{f_\gamma\},J_N)} \eps(\u,\utau,\uL)\ug_q.
$$

In the rest of this section we assume that the family $J_N^{\utau,\uL}$ is a small enough perturbation of the regular almost complex structure $J$, and that $J_N^{\utau,\uL}(s,\theta)$ coincides with $J$ for all $\utau,\uL$ and $|s|\ge s_0=s_0(\utau,\uL)$. 

\begin{lemma} \label{lem:phivanish}
 We have $\varphi_N=0$ for $N\ge 2$, whereas $\varphi_1:SC_*^+(H,J)\to SC_{*+1}^+(H,J)$ acts by 
\begin{equation}\label{eq:phi1}
\varphi_1(\gamma_m)=0, \qquad \varphi_1(\gamma_M)=\left\{\begin{array}{ll}Ê\kappa_\gamma \gamma_m, & \mbox{if } \gamma \mbox{ is good}, \\ 0, & \mbox{if } \gamma \mbox{ is bad}.\end{array}\right.
\end{equation}
\end{lemma}

\begin{proof}
The dimension formula~\eqref{eq:dimension-formula} shows that for moduli spaces of dimension $0$ we have 
$$
|\og|-|\ug|=1-2N+\ind(q)-\ind(p).
$$
Since $\ind(p),\ind(q)\in\{0,1\}$ the right hand side of this equality is $\le 2-2N$, and for $N\ge 1$ we obtain 
\begin{equation} \label{eq:leq0}
|\og|-|\ug|\le 0.
\end{equation}

The regularity assumption on the pair $(H,J)$ implies that the moduli space $\cM(S_{\gamma^-},S_{\gamma^+};H,J)$, $\gamma^\pm\in\cP^+(H)$ is non-empty only if $|\gamma^-|-|\gamma^+|\ge 0$, and the equality implies $S_{\gamma^-}=S_{\gamma^+}$. A limit argument shows that the same conclusion holds for any moduli space $\cM(S_{\gamma^-},S_{\gamma^+};H,J')$ obtained from an almost complex structure $J'=J'(s,\theta)$ which is close enough to $J$. In particular the same conclusion holds for the moduli spaces $\cM(S_{\gamma^-},S_{\gamma^+};H,J_N^{\utau,\uL})$, $(\utau,\uL)\in(S^1)^N\times([0,\infty[)^{N-1}$, and we obtain that $\cM^N(p,q;H,\{f_\gamma\},J_N)$ is non-empty only if 
\begin{equation} \label{eq:geq0}
|\og|-|\ug|\ge 0.
\end{equation}

Let us first assume $N\ge 2$. It follows from~\eqref{eq:leq0} and~\eqref{eq:geq0} that the moduli spaces $\cM^N(p,q;H,\{f_\gamma\},J_N)$ are empty, so that $\varphi_N=0$. 

Let us now assume $N=1$. We obtain that $|\og|=|\ug|$ and this implies $S_\og=S_\ug$. Moreover $\ind(q)=1$, $\ind(p)=0$, so that $q$ is a minimum (denoted $m$) and $p$ is a maximum (denoted $M$). Up to reparametrization we can choose a common representative for $S_\og$ and $S_\ug$, denoted $\gamma$. The condition $|\og|=|\ug|$ implies that the Floer trajectories involved in $\cM^1(M,m;H,\{f_\gamma\},J_1)$ are all constant and in particular $\ell=1$, i.e. $\cM^1(M,m;H,\{f_\gamma\},J_1)=\cM^1_1(M,m;H,\{f_\gamma\},J_1)$ and there is a unique such Floer trajectory involved in our configuration. Since we are counting only rigid configurations in $\cM^1_1(M,m;H,\{f_\gamma\},J_1)$, this is the constant trajectory at $\gamma_M$. The elements of $\cM^1_1(M,m;H,\{f_\gamma\},J_1)$ are therefore pairs $(u\equiv \gamma_M,\tau_0)$ with $\tau_0$ such that $\gamma_m(\cdot-\tau_0)=\gamma_M$, i.e. $\tau_0\cdot\gamma_m=\gamma_M$. We see that there are precisely $\kappa_\gamma$ possible choices for $\tau_0\in S^1$. If $\gamma$ is good all the configurations count with the same sign equal to $1$ since the $S^1$-action preserves the geometric orientation and, for $p\in \mathrm{Crit}(f_\gamma)$, the coherent orientations for $W^u(p)$ and $W^s(p)$ coincide with the geometric ones~\cite[Remark~4.33]{BOauto}. 
Equivalently,  the space of Cauchy-Riemann operators on a punctured Riemann surface with one puncture asymptotic to some element of $S_\gamma$ is equipped with a natural $S^1$-action rotating that asymptote.  In the case of the complex plane, we adopted the convention that this $S^1$-action  preserves the coherent orientations. Therefore, this $S^1$-action must preserve the coherent orientations for any punctured Riemann surface, and in particular for the cylinder. Hence, any element of $\cM^1_1(M,m;H,\{f_\gamma\},J_1)$ has the same coherent sign as a constant Floer
trajectory, i.e. $+1$.
Thus $\varphi_1(\gamma_M)=\kappa_\gamma \gamma_m$ if $\gamma$ is good. If $\gamma$ is bad then $\kappa_\gamma$ is even and two successive choices of $\tau_0$ on $S^1$ count with opposite signs, so that $\varphi_1(\gamma_M)=0$. 

The Lemma is proved.
\end{proof}

\begin{remark} {\it Recall from~\S\ref{sec:algebraicS1} that, for a general $S^1$-complex, the square of the BV-operator $\Delta:=\varphi_1$ is homotopic to zero. The content of Lemma~\ref{lem:phivanish} above is that $\Delta^2=0$ \emph{at chain level} provided we use Morse-Bott autonomous data and transversality holds. Thus $(SC_*^+(H,J),\p,\Delta)$ is a mixed complex in the sense of Remark~\ref{rmk:mixed-complex}. This is an instance of the general principle that choosing the auxiliary data in a way that is compatible with the geometry results in simpler algebraic structures. We refer to~\cite[Remark~5.7]{BOGysin} for a description of $\Delta$ within a different framework.
}
\end{remark}

We now define moduli spaces relating the positive $S^1$-equivariant symplectic complex as described in \S\ref{sec:explicit} to the positive invariant symplectic complex 
from \S\ref{sec:Hamiltformul}.

Let $\beta : \R \to [0,1]$ be a smooth, increasing function such that $\beta(s) = 0$ if $s \le -1$
and $\beta(s) = 1$ for $s \ge 0$. Let $r : [0,1] \times \cJ \to \cJ$ be a smooth map such that
$r(0,\cdot)$ is the identity on $\cJ$ and $r(1,J') = J$ for all $J' \in \cJ$.
For $(\utau, \uL) \in (S^1)^N \times ([0,\infty[)^N$, we define 
$H^{\utau,\uL}_\beta : \R \times S^1 \times \widehat{W} \to \R$ by
$H^{\utau,\uL}_\beta(s,\theta,x) = (1-\beta(s-L_1 - \cdots - L_N))K_{(\utau,\uL)}(s,\theta,x) + 
\beta(s-L_1 - \cdots - L_N) H(x)$ 
and $J^{\utau,\uL}_\beta : \R \times S^1 \times \widehat{W} \to \cJ$ by 
$J^{\utau,\uL}_\beta(s,\theta,x) = r(\beta(s-L_1 - \cdots - L_N), J^{\utau,\uL}_N(s,\theta,x))$.

For simplicity we assume in the sequel that $K$ is a small time-dependent perturbation of $H$ of Morse-Bott type, so that elements in $\cP^+(K)$ are of the form $\gamma_p$, with $\gamma\in\cP^+(H)$ and $p\in\mathrm{Crit}(f_\gamma)$. Given $\og, \ug \in \cP(H)$ and $p \in \textrm{Crit}(f_\og)$ we define $\cM^N(\og_p, \ug; H_\beta, J_\beta)$ to consist of tuples $(u, \utau, \uL)$ such that
$(\utau, \uL) \in (S^1)^N \times ([0,\infty[)^N$ and $u : \R \times S^1 \to \widehat{W}$ solves
the equation 
$$
\p_s u + J^{\utau,\uL}_\beta(s, \theta, u(s,\theta)) (\p_\theta u - X_{H^{\utau,\uL}_\beta}(s, \theta,u(s,\theta))) = 0
$$ 
and is subject to the asymptotic conditions 
$$
\lim_{s \to -\infty} u(s,\cdot) = \og_p, \qquad \lim_{s \to\infty} u(s,\cdot) \in S_\ug.
$$
For a generic choice of the family $J_N$ this moduli space is a smooth manifold of dimension
$\dim \cM^N(\og_p, \ug; H_\beta, J_\beta) = |\og| + \ind(p) - |\ug| + 2N$.

Let us define the map
$$
\Pi:SC_*^{+,S^1}(H,J)\to SC_*^{+,inv}(H,J)
$$
by
$$
\Pi(u^N \og_p) = \sum_{\scriptsize \begin{array}{c}Ê\ug \mbox{ good} \\ |\ug| = |\og| + \ind(p) + 2N \end{array}} 
\sum_{(u,\utau,\uL) \in \cM^N(\og_p,\ug;H_\beta,J_\beta)}  \epsilon(u,\utau,\uL) \ S_\ug .
$$

\begin{remark} {\it Referring to the finite-dimensional example mentioned at the beginning of~\S\ref{sec:CHlin}, the map $\Pi$ corresponds to the first projection $X_{Borel}\to X/S^1$.}
\end{remark} 

As was the case with the maps $\varphi_N$, the map $\Pi$ admits an alternative description in terms of Morse-Bott moduli spaces as follows. We define a moduli space $\cM^N_\ell(p,S_\ug; H, \{ f_\gamma \}, J_\beta)$ to consist of the disjoint union over 
$\tilde\gamma_1,\dots,\tilde\gamma_{\ell-1}\in\cP^+(H)$ of tuples $(\u,\utau,\uL)$ such that 
$(\utau,\uL)\in(S^1)^N\times([0,\infty[)^N$ and $\u$ is an element of the fiber product 
(with the convention $\tilde\gamma_0:=\og$, $\tilde\gamma_\ell:=\ug$)
\begin{eqnarray*} 
&& 
W^u(p)  
\times_{\oev} 
(\cM(S_{\tgamma_0},S_{\tgamma_1};H,J_0^\theta)\!\times\!\R^+) \\ 
&& {_{\varphi_{f_{\tgamma_1}}\!\circ\uev}}\!\times _{\oev}  
\ldots {_{\varphi_{f_{\tgamma_{\ell-2}}}\!\circ\uev}}\!\times _{\oev}  
(\cM(S_{\tgamma_{\ell-2}},S_{\tgamma_{\ell-1}};H,J_0^\theta)\!\times\!\R^+) \\ 
&&  
{_{\varphi_{f_{\tgamma_{\ell-1}}}\!\circ\uev}\times_{\oev}}  
\cM^N(S_{\tgamma_{\ell-1}},S_{\tgamma_\ell};H,J_\beta)
\end{eqnarray*} 
Here $\cM^N(S_{\tgamma_{\ell-1}},S_{\tgamma_\ell};H,J_\beta)$ is the moduli space consisting of tuples $(u,\utau,\uL)$ such that $(\utau,\uL)\in(S^1)^N\times([0,\infty[)^N$ and $u:\R\times S^1\to \widehat W$ solves the 
equation $\p_s u+J_\beta^{\utau,\uL}(s,\theta,u(s,\theta))(\p_\theta u-X_H(u(s,\theta))=0$ and is subject to the asymptotic conditions 
$$
\lim_{s\to-\infty}u(s,\cdot)\in S_{\tgamma_{\ell-1}}, \qquad \lim_{s\to\infty}u(s,\cdot)\in S_{\tgamma_\ell}.
$$
For a generic choice of the family $J_N$ this moduli space is a smooth manifold of dimension 
$\dim\,\cM^N(S_{\tgamma_{\ell-1}},S_{\tgamma_\ell};H,J_\beta)=|\tgamma_{\ell-1}|-|\tgamma_\ell|+2N$. 

One shows as in~\cite[Lemma~3.6]{BOauto} that, for a generic choice of 
the collection of Morse functions $\{f_\gamma\}$, each space  
$\cM^N_\ell(p,S_\ug;H,\{f_\gamma\},J_\beta)$ is a smooth manifold, and its dimension is 
\begin{equation} \label{eq:dimension-formulaPi}
\dim \, \cM^N_\ell(p,S_\ug;H,\{f_\gamma\},J_\beta) =  |\og| +\ind(p) - |\ug| +2N. 
\end{equation} 
We denote 
$$\cM^N(p,S_\ug;H,\{f_\gamma\},J_\beta)$$
the union over $\ell\ge 1$ of the spaces 
$\cM^N_\ell(p,S_\ug;H,\{f_\gamma\},J_\beta)$. The arguments of~\cite{BOauto} carry over in order to show that the elements of the $0$-dimensional moduli spaces $\cM^N(p,S_\ug;H,\{f_\gamma\},J_\beta)$ are in one-to-one bijective correspondence preserving signs with the elements of the $0$-dimensional moduli spaces $\cM^N(\og_p,\ug;H_\beta,J_\beta)$, where $\og_p \in\cP^+(K)$ is the periodic orbit that corresponds to the critical point $p\in\mathrm{Crit}(f_\og)$. We thus obtain 
$$
\Pi(u^N \og_p)=\sum_{\scriptsize \begin{array}{c}Ê\ug \mbox{ good}\\ |\ug| = |\og| + \ind(p) + 2N \end{array}} 
\sum_{(\u,\utau,\uL) \in \cM^N(p,S_\ug;H,\{f_\gamma\},J_\beta)}  \epsilon(\u,\utau,\uL) \ S_\ug .
$$

\begin{lemma} \label{lem:Pi}
The map $\Pi$ is explicitly given by
\begin{equation}\label{eq:Pi}
\Pi(u^N\otimes \gamma_m)=0,\qquad \Pi(u^N\otimes \gamma_M)=\left\{\begin{array}{ll}ÊS_\gamma, & \mbox{if }ÊN=0 \mbox{ and } \gamma \mbox{ is good,} \\ 0, & \mbox{else.}  \end{array}\right.
\end{equation}
\end{lemma}

\begin{proof}

The dimension formula~\eqref{eq:dimension-formulaPi} shows that for moduli spaces of dimension $0$ we have 
$$
|\og|-|\ug|=-2N-\ind(p) \le 0.
$$
We infer that this moduli space is empty unless $N = 0$ and $\ind(p)=0$. In particular, we have
$\Pi(u^N \otimes \gamma_M) = 0$ if $N > 0$ and $\Pi(u^N \otimes \gamma_m) = 0$.
 
Assume now that $N = 0$ and $\ind(p)=0$.  Then 
$\cM^0(\og_M, S_\ug; H_\beta, J_\beta)$ coincides with the moduli space 
$\cM(\og_M,S_\ug;H,J):=W^u(M)\times_{\oev}\cM(S_\og, S_\ug; H, J)$. Our transversality assumption on the pair $(H,J)$ implies that $\cM(S_\og,S_\ug;H,J)$ is nonempty only if $|\og|-|\ug|\ge 0$, with equality only if $\og=\ug$. Since in our case we have $|\og|-|\ug|\le 0$, we infer that indeed $\og=\ug=:\gamma$. But $\cM(S_\gamma,S_\gamma;H,J)$ is naturally identified with $S_\gamma$, so that $\cM(\gamma_M,S_\gamma;H,J)$ consists of a single element. 

That the sign of this element is $+1$ can be seen as follows. The coherent orientation for $W^u(M)$ coincides with its geometric orientation~\cite[Remark~4.33]{BOauto} and is therefore equal to $+1$, while the coherent orientation of $\cM(S_\gamma,S_\gamma;H,J)$ coincides with the geometric orientation of $S_\gamma$. Thus the fiber product is oriented with sign $+1$. Alternatively, the coherent orientation of the determinant line over the space of Cauchy-Riemann operators defined on a plane and asymptotic to (the linearization of) some element of $S_\gamma$ induces the coherent orientation by restricting to the space of Cauchy-Riemann operators asymptotic to (the linearization of) $\gamma_M$. This implies again that the fiber product is oriented with sign $+1$.

Since we defined $\Pi$ by a sum over good orbits we obtain that $\Pi(1\otimes \gamma_M)=0$ if $\gamma$ is bad, whereas $\Pi(1\otimes \gamma_M)=S_\gamma$ if $\gamma$ is good by the argument above.
\end{proof}

The next Proposition is the main result of this section. The description of $S^1$-equivariant homology as the homology of an $S^1$-complex and the explicit description of the $S^1$-equivariant differential provided by Lemma~\ref{lem:phivanish} give us the possibility to present a purely algebraic proof. 

\begin{proposition} \label{prop:isomorphism}
$\Pi$ is a chain map, it induces an isomorphism in homology over $\Q$ and it realizes an isomorphism of exact triangles between~\eqref{eq:Gysintriangle} for $SH_*^{+,S^1}(H,J)$ and~\eqref{eq:lesGysinSH+inv} for $SH_*^{+,inv}(H,J)$.
\end{proposition}

\begin{proof}

(I) We first prove that $\Pi$ is a chain map, i.e. $\Pi\p^{S^1}=\p\Pi$. This identity needs to be tested on four types of generators: (i)~$u^j\otimes\gamma_m$, (ii)~$u^j\otimes\gamma_M$ with $\gamma$ bad, (iii)~$u^j\otimes\gamma_M$ with $\gamma$ good and $j\ge 1$, and (iv)~$1\otimes \gamma_M$ with $\gamma$ good.
The first three cases follow immediately from the fact that $\p^{S^1}=d+u^{-1}\Delta$, from the definitions of $\Pi$ and $\Delta$, and from the properties of $d^0$ and $d^1$ in the decomposition $d=d^0+d^1+d^2$. In the fourth case we have from~\eqref{eq:differentialSH+inv} that 
$$
\p\Pi(1\otimes \gamma_M)=\p\langle S_\gamma\rangle=\sum_{\scriptsize \ug \mbox{ good}} \kappa_\gamma\#\cM(S_\gamma,S_\ug)/S^1\langle S_\ug \rangle.
$$
On the other hand, using that $d^0\gamma_M=0$, $d^2(CM_*)\subset Cm_*$, and $\Pi|_{\Z[u]\otimes CM_*^{bad}}=0$, we have
\begin{eqnarray*}
\Pi\p^{S^1}(1\otimes\gamma_M) & = & \Pi(1\otimes d^1\gamma_M) \\
& = & \Pi\left(1\otimes \sum_{\scriptsize \ug} \#\cM(\gamma_M,\ug_M)\ug_M\right) \\
& =& \sum_{\scriptsize \ug \mbox{ good}} \#\cM(\gamma_M,\ug_M) \langle S_\ug \rangle\\
& =& \sum_{\scriptsize \ug \mbox{ good}} \kappa_\gamma\#\cM(S_\gamma,S_\ug)/S^1\langle S_\ug \rangle.
\end{eqnarray*}
Here the coefficient $\#\cM(\gamma_M,\ug_M)$ describes the number of elements in some moduli space which gives the coefficient in front of $\ug_M$. For a general description of this moduli space the reader can refer to~\cite[(12) and~(44)]{BOcont}, but in our case this specializes to $\cM(\gamma_M,\ug_M)=W^u(M)\times_{\oev} \cM(S_\gamma,S_\ug)\, _{\uev}\times W^s(\underline{M})$. The second equality above is then just the definition of $d^1$, the third equality is the definition of $\Pi$, and the fourth equality expresses the fact that, modulo the obvious identification of $CM_*/CM_*^{bad}$ with $SC_*^{+,inv}$, the differential induced by $d^1$ coincides with the differential $\p$. (This is implicit in the last paragraph of~\S\ref{sec:GysinSH+inv}.)

(II) We now prove that $\Pi$ induces an isomorphism in homology over $\Q$. This makes use of the fact that the complex $SC_*^{+,S^1}$ is filtered by $\mu(\gamma)$, which follows from the explicit description of the BV-operator and the fact that $\p^{S^1}=d+u^{-1}\Delta$. The complex $SC_*^{+,inv}$ carries a filtration by $\mu(\gamma)$ (\emph{la filtration b\^ete}), the $E^1$-page is canonically identified with $SC_*^{+,inv}$ and the spectral sequence degenerates at the $E^2$-page. The map $\Pi$ preserves the filtrations and induces a map of spectral sequences. We shall prove that $\Pi$ induces an isomorphism at the level of the $E^1$-page. 
\emph{Note that this is the Morse theoretical analogue of the proof of Lemma~\ref{lem:Leray}.}

As a vector space, the $E^0$-page of the spectral sequence for $SC_*^{+,S^1}$ coincides with $SC_*^{+,S^1}$. The differential $\bar d^0$ is induced by $d^0$ and acts by  
$$
\bar d^0(u^j\otimes \gamma_m)=\left\{\begin{array}{ll}Ê\pm2 u^j\otimes \gamma_M, & \gamma \mbox{ bad,}Ê\\Ê0, & \gammaÊ\mbox{ good,}\end{array}\right. \qquad 
\bar d^0(u^j\otimes \gamma_M)=\left\{\begin{array}{ll}Ê\kappa_\gamma u^{j-1}\gamma_m,& j\ge 1,\\
0,& j=0.\end{array}\right.
$$

In the terminology of~\S\ref{sec:algebraicS1}, given $\kappa\in\Z$ we denote by $C^{\kappa}$ the $S^1$-complex $(\Lambda(a),\varphi_\cdot)$, where $a$ has degree $1$, $\Lambda(a)$ denotes the free graded unital algebra generated by $a$, $\varphi_j=0$ for $j\neq 1$ and $\varphi_1=:\Delta$ with $\Delta1=\kappa a$, $\Delta a=0$. This is a model for the circle $S^1=\R/\Z$ endowed with the $S^1$-action $\tau\cdot\theta:=\theta-\kappa\tau$, or equivalently for $S_\gamma$ with $\gamma$ a good orbit of multiplicity $\kappa$. The $S^1$-equivariant homology with $\Q$-coefficients for this $S^1$-complex is the homology of the complex $\tilde C^\kappa:=(\Q[u]\otimes_\Z\Lambda(a),\tilde\p^\kappa=u^{-1}\Delta)$, which has rank $1$ in all non-negative degrees and whose differential can be explicitly written as follows
$$
\xymatrix{\Q & \ar[l]_0 \Q & \ar[l]_{\times \kappa} \Q & \ar[l]_0 \Q & \ar[l]_{\times \kappa}Ê\Q & \ar[l]_0 \dots}
$$
This complex computes -- as it should -- the homology of the classifying space $BStab(x)$, with $Stab(x)$ being the stabilizer of the geometric action at an arbitrary point $x$. For $\kappa\neq 0$ the stabilizer is $\Z/\kappa\Z$ and the homology is supported in degree $0$, where it has rank $1$.

Still in the terminology of~\S\ref{sec:algebraicS1}, let us denote $C^{bad}$ the $S^1$-complex $(\Lambda(a),\varphi_\cdot)$, with $\varphi_j=0$ for $j\ge 1$ and $\varphi_0 =:\p$ with $\p a=\pm 2$. This is a model for $S_\gamma$ with $\gamma$ a bad orbit of arbitrary (even) multiplicity. The $S^1$-equivariant homology with $\Q$-coefficients for this $S^1$-complex is the homology of the complex $\tilde C^{bad}:=(\Q[u]\otimes_\Z\Lambda(a),\tilde\p^{bad}=\p)$ and it vanishes since $H(\Lambda(a),\p)=0$. Note that $H(\Lambda(a),\p)$ is the homology of $S^1$ with twisted coefficients. 

If we identify $\gamma_m$ with $a\otimes\gamma_M$, the $E^0$-page of the spectral sequence for $SC_*^{+,S^1}$ is given as a complex by 
$$
E^0\simeq \bigoplus_{\scriptsize \begin{array}{c} S_\gamma \\ \gamma \mbox{ good}\end{array}} \tilde C^{\kappa_\gamma} \otimes \langle \gamma_M \rangle \quad \oplus \quad
\bigoplus_{\scriptsize \begin{array}{c} S_\gamma \\ \gamma \mbox{ bad}\end{array}} \tilde C^{bad}\otimes \langle \gamma_M\rangle.
$$

It follows from the above discussion on $C^\kappa$ and $C^{bad}$ that upon taking homology the $E^1$-page of the spectral sequence is given by 
$$
E^1=\bigoplus_{\scriptsize \begin{array}{c} \gamma \mbox{ good}\end{array}} \langle \gamma_M\rangle
$$
and in particular the map $\Pi$ induces an isomorphism of vector spaces $E^1\stackrel\sim\longrightarrow SC_*^{+,inv}$. Although this is enough to conclude that $\Pi$ itself is a quasi-isomorphism, let us note that the  isomorphism $E^1\stackrel\sim\longrightarrow SC_*^{+,inv}$ is actually an isomorphism of complexes. 
Indeed, the differential $\bar d^1$ induced by $d^1$ acts as already seen by 
$$
\bar d^1\langle\gamma_M\rangle=\sum_{\scriptsize \ug \mbox{ good}} \#\cM(\gamma_M,\ug_M)\langle\ug_M\rangle=\sum_{\scriptsize \ug \mbox{ good}} \kappa_\gamma\#\cM(S_\gamma,S_\ug)/S^1\langle \ug_M\rangle,
$$
and the latter expression coincides with that of the differential $\p$ on $SC_*^{+,inv}$ modulo the identification $\gamma_M\mapsto S_\gamma$.

(III) We need to prove the commutativity or anti-commutativity of the three squares in the following diagram.
\begin{equation} \label{eq:diagIII}
\xymatrix{
SH_*^+ \ar@{=}[d] \ar[r]^-I & 
SH_*^{+,S^1} \ar[d]^{\Pi_*} \ar[r]^-S & 
SH_{*-2}^{+,S^1} \ar[d]^{\Pi_*} \ar[r]^-B & 
SH_{*-1}^+ \ar@{=}[d] \\
SH_*^+ \ar[r]Ê& 
SH_*^{+,inv}Ê\ar[r]^-DÊ& 
SH_{*-2}^{+,inv}Ê\ar[r] & 
SH_{*-1}^+ 
}
\end{equation}

\noindent {\bf Remark.}Ê\ {\it We will actually prove that the first and the third squares are commutative, whereas the middle square is anti-commutative. This seems to be related to the anti-commutativity of the last square in Lemma~\ref{lem:grid}, although we do not have a conceptual explanation.}

We define the bottom sequence in~\eqref{eq:diagIII} via the following identification 
$$
\xymatrix{
SH_*^+ \ar[r]Ê\ar@{=}[dd] & 
SH_*^{+,inv}Ê\ar[r]^-DÊ\ar[dd]^{F_*} & 
SH_{*-2}^{+,inv}Ê\ar[r] \ar[d]^{ \Theta^{-1}} & 
SH_{*-1}^+  \ar@{=}[dd] \\
& & \left(SH_{*-2}^{+,inv}\right)' \ar[d]^{F'_*} & \\
SH_*^+ \ar[r]^-{p_*} &
E^2_{*,0} \ar[r]^{\bar d^2} & 
E^2_{*-2,1} \ar[r]^-{i_*} &
SH_{*-1}^+ 
}
$$
The map $F_*:SH_*^{+,inv}\to E^2_{*,0}$ is induced by $F:SC_*^{+,inv}\to E^1_{*,0}=CM_*/CM_*^{bad}$, $S_\gamma\mapsto \gamma_M \mbox{ mod } CM_*^{bad}$, whereas the map $ F'_*:(SH_{*-2}^{+,inv})'=:H_{*-2}(SC_*^{+,inv},\p')\to E^2_{*-2,1}$ is induced by $F':SC_{*-2}^{+,inv}\to E^1_{*-2,1}=Cm_{*-1}^{good}$, $S_\gamma\mapsto \gamma_m$. 

We now prove commutativity of the three squares in the diagram obtained by composing the above two diagrams.

(i) Commutativity of the first square 
$$
\xymatrix{
SH_*^+ \ar[r]^-I \ar@{=}[d] & SH_*^{+,S^1} \ar[d] \\
SH_*^+ \ar[r]^-{p_*} & E^2_{*,0}
}
$$
We have a commutative diagram at chain level
$$
\xymatrix{
SC_*^+ \ar[r]^-I \ar@{=}[d] & SC_*^{+,S^1} \ar[d] \\
SC_*^+ \ar[r]^-{p_*} & E^1_{*,0}
}
$$
Recall that the map $I$ is induced by inclusion (Proposition~\ref{prop:Gysin+spectral}). Thus, an element $\gamma_m$ in the top copy of $SC_*^+$ is sent through both compositions to zero, just like an element $\gamma_M$ with $\gamma$ bad. On the other hand, an element $\gamma_M$ with $\gamma$ good is sent through both compositions to $(\gamma_M \mbox{ mod } CM_*^{bad})\in E^1_{*,0}$.

(ii) Commutativity of the third square
$$
\xymatrix{
SH_{*-2}^{+,S^1} \ar[r]^-B \ar[d] & SH_{*-1}^+ \ar@{=}[d] \\
E^2_{*-2,1} \ar[r]^-{i_*} & SH_{*-1}^+
}
$$
We again prove the commutativity of the diagram at chain level
$$
\xymatrix{
SC_{*-2}^{+,S^1} \ar[r]^-B \ar[d] & SC_{*-1}^+ \ar@{=}[d] \\
E^1_{*-2,1} \ar[r]^-{i_*} & SC_{*-1}^+
}
$$
Recall that $B$ is induced by $\Delta$ (Proposition~\ref{prop:Gysin+spectral}) and acts therefore by 
$$
u^j\otimes\gamma_m\mapsto 0,\qquad u^j\otimes \gamma_M\mapsto \left\{\begin{array}{ll} \kappa_\gamma\gamma_m, & \mbox{if } \gamma \mbox{ is good and } j=0,\\ 0, & \mbox{otherwise.} \end{array}\right.
$$
Thus an element $u^j\otimes \gamma_m$ in the top left corner of the diagram is sent to $0$ via both compositions, and similarly for an element $u^j\otimes \gamma_M$ such that $\gamma$ is bad or $j\ge 1$. On the other hand, an element $1\otimes \gamma_M$ with $\gamma$ good is sent via both the top composition and the bottom composition to $\kappa_\gamma\gamma_m$.

(iii) We prove anti-commutativity of the middle square 
$$
\xymatrix{
SH_k^{+,S^1}Ê\ar[r]^-S \ar[d]^-{F_*\Pi_*} & SH_{k-2}^{+,S^1} \ar[d]^-{F'_*\Theta^{-1}\Pi_*} \\
E^2_{k,0} \ar[r]^{\bar d^2}Ê& E^2_{k-2,1}
}
$$
In this case it only makes sense to check commutativity of the diagram in homology since $\bar d^2$ is a differential on the $E^2$-page of a spectral sequence and is a priori \emph{not} induced by a chain map defined on the $E^1$-page. 

Recall that $S$ is ``multiplication by $u^{-1}$'' (Proposition~\ref{prop:Gysin+spectral}). Let 
$$
c_k:=\sum_{j\ge 0} u^j\otimes \big(\sum_{\scriptsize \begin{array}{cc} |\gamma'|+2j=k \end{array}} \alpha_{\gamma'}^j \gamma'_M + \sum_{\scriptsize \begin{array}{cc} |\gamma''|+2j+1=k \end{array}} \beta_{\gamma''}^j \gamma''_m\big)
$$
be a cycle in $SC_k^{+,S^1}$, with $\alpha_{\gamma'}^j,\beta_{\gamma''}^j\in\Q$. The condition $\p^{S^1}c_k=0$ is equivalent to the fact that, for any $j\ge 0$, we have
\begin{eqnarray}\label{eq:pS1ck}
\sum_{\scriptsize |\gamma'|=k-2j} \alpha_{\gamma'}^j (d^1\gamma'_M+d^2\gamma'_M) + \sum_{\scriptsize |\gamma'|=k-2j-2} \alpha_{\gamma'}^{j+1} \Delta\gamma'_M \\
+ \sum_{|\gamma''|=k-2j-1} \beta_{\gamma''}^j (d^0\gamma''_m+d^1\gamma''_m) \ = \ 0. \nonumber
\end{eqnarray}

For $j=0$ we obtain the following two relations by considering the components of~\eqref{eq:pS1ck} on $CM_*$, respectively on $Cm_*$: 
\begin{equation} \label{eq:pS1ckCM*}
\sum_{\scriptsize \begin{array}{cc} |\gamma'|=k \end{array}} \alpha_{\gamma'}^0 d^1\gamma'_M
+ \sum_{|\gamma''|=k-1} \beta_{\gamma''}^0 d^0\gamma''_m = 0,
\end{equation}
\begin{equation} \label{eq:pS1ckCm*}
\sum_{\scriptsize \begin{array}{cc} |\gamma'|=k \end{array}} \alpha_{\gamma'}^0 d^2\gamma'_M 
+ \sum_{\scriptsize |\gamma'|=k-2} \alpha_{\gamma'}^{1} \Delta\gamma'_M
+ \sum_{|\gamma''|=k-1} \beta_{\gamma''}^0 d^1\gamma''_m = 0.
\end{equation} 

Equation~\eqref{eq:pS1ckCM*} expresses the fact that $F\Pi(c_k)$ is a cycle in $E^1_{k,0}=CM_k/CM_k^{bad}$. Indeed, viewing $F\Pi(c_k)$ as an element of $CM_*$ we compute 
\begin{eqnarray*}
d^1(F\Pi(c_k)) & = & \sum_{\scriptsize \begin{array}{cc} \gamma' \mbox{ good}Ê\\Ê|\gamma'|=k \end{array}} \alpha_{\gamma'}^0 d^1\gamma'_M \\
& = & \sum_{\scriptsize \begin{array}{cc} \gamma'Ê\\Ê|\gamma'|=k \end{array}} \alpha_{\gamma'}^0 d^1\gamma'_M\\
& = & - \sum_{|\gamma''|=k-1} \beta_{\gamma''}^0 d^0\gamma''_m \\
& = & - d^0\big(\sum_{\scriptsize \begin{array}{cc} \gamma'' \mbox{ bad}Ê\\Ê|\gamma''|=k-1 \end{array}} \beta_{\gamma''}^0 \gamma''_m\big), 
\end{eqnarray*}
and this last expression belongs to $CM_*^{bad}$. Thus $\bar d^1(F\Pi(c_k))=0$.

Now, by the definition of the induced map $\bar d^2$ in a spectral sequence (see~\cite[Example~7.1]{LS}) we have 
$$
\bar d^2F_*\Pi_*[c_k]=[d^2F\Pi(c_k)+d^1\beta]
$$
where $d^1(F\Pi(c_k))+d^0\beta=0$. The previous computation shows that we can take 
$$
\beta:=\sum_{\scriptsize \begin{array}{cc} \gamma'' \mbox{ bad}Ê\\Ê|\gamma''|=k-1 \end{array}} \beta_{\gamma''}^0 \gamma''_m
$$ 
and we obtain that the bottom composition in our diagram is 
\begin{eqnarray*}
\bar d^2F_*\Pi_*[c_k] & = & \big[d^2\sum_{\scriptsize \begin{array}{cc} \gamma'Ê\mbox{ good}Ê\\Ê|\gamma'|=k \end{array}} \alpha_{\gamma'}^0 \gamma'_M + d^1\sum_{\scriptsize \begin{array}{cc} \gamma'' \mbox{ bad}Ê\\Ê|\gamma''|=k-1 \end{array}} \beta_{\gamma''}^0 \gamma''_m\big] \\
& = & \big[d^2\sum_{\scriptsize \begin{array}{cc} \gamma'Ê\\Ê|\gamma'|=k \end{array}} \alpha_{\gamma'}^0 \gamma'_M + d^1\sum_{\scriptsize \begin{array}{cc} \gamma'' \mbox{ bad}Ê\\Ê|\gamma''|=k-1 \end{array}} \beta_{\gamma''}^0 \gamma''_m\big]
\end{eqnarray*}
The second equality is just the consequence of the fact that we are considering $F\Pi(c_k)$ as a class in $CM_*/CM_*^{bad}$.

On the other hand, the top composition is given at chain level by $F'\Theta^{-1}\Pi S:SC_k^{+,S^1}\to E^1_{k-2,1}=Cm_{k-1}^{good}$ and is directly computed to be  
\begin{eqnarray*}
F'\Theta^{-1}\Pi S(c_k) & = & F'\Theta^{-1}(\sum_{\scriptsize \begin{array}{cc} \gamma' \mbox{ good}Ê\\Ê|\gamma'|=k-2 \end{array}} \alpha_{\gamma'}^1 S_{\gamma'}) \\
& = & F'(\sum_{\scriptsize \begin{array}{cc} \gamma' \mbox{ good}Ê\\Ê|\gamma'|=k-2 \end{array}} \alpha_{\gamma'}^1 \kappa_{\gamma'} S_{\gamma'})\\
& = & \sum_{\scriptsize \begin{array}{cc} \gamma' \mbox{ good}Ê\\Ê|\gamma'|=k-2 \end{array}} \alpha_{\gamma'}^1 \kappa_{\gamma'} \gamma'_m \\
& = & \sum_{\scriptsize \begin{array}{cc} \gamma' \mbox{ good}Ê\\Ê|\gamma'|=k-2 \end{array}} \alpha_{\gamma'}^1\Delta(\gamma'_M) \\
& =  & \sum_{\scriptsize \begin{array}{cc} \gamma' Ê\\Ê|\gamma'|=k-2 \end{array}} \alpha_{\gamma'}^1\Delta(\gamma'_M).  
\end{eqnarray*} 

We obtain 
\begin{eqnarray*}
\lefteqn{\bar d^2F_*\Pi_*[c_k] + F'_*\Theta^{-1}\Pi_* S[c_k]} \\
& = & \big[d^2\sum_{\scriptsize \begin{array}{cc} \gamma'Ê\\Ê|\gamma'|=k \end{array}} \alpha_{\gamma'}^0 \gamma'_M 
+ d^1\sum_{\scriptsize \begin{array}{cc} \gamma'' \mbox{ bad}Ê\\Ê|\gamma''|=k-1 \end{array}} \beta_{\gamma''}^0 \gamma''_m
+\sum_{\scriptsize \begin{array}{cc} \gamma' Ê\\Ê|\gamma'|=k-2 \end{array}} \alpha_{\gamma'}^1\Delta(\gamma'_M)\big] \\
& = & \big[-d^1\sum_{\scriptsize \begin{array}{cc} \gamma'' \mbox{ good}Ê\\Ê|\gamma''|=k-1 \end{array}} \beta_{\gamma''}^0 \gamma''_m\big]\\
& = & 0 \in E^2_{k-2,1}.
\end{eqnarray*}
For the second equality we use~\eqref{eq:pS1ckCm*} and for the third equality we use the definition of $E^2_{k-2,1}=H_{k-2}(E^1_{*,1}=Cm^{good}[-1]_*,\bar d^1=d^1)$, where $Cm^{good}[-1]_*:=Cm_{*-1}^{good}$. This proves anti-commutativity of the middle diagram and finishes the proof of the Proposition. 
\end{proof}

It follows from the definition that the map $\Pi$ commutes with the continuation maps induced by increasing homotopies. The next statement is a straightforward consequence of this fact and of Proposition~\ref{prop:isomorphism}.

\begin{proposition} There is an isomorphism $\Pi:SH_*^{+,S^1}(W)\stackrel \sim\longrightarrow SH_*^{+,inv}(W)$ in homology over $\Q$ which moreover realizes an isomorphism of exact triangles between~\eqref{eq:Gysintriangle} for $SH_*^{+,S^1}(W)$ and~\eqref{eq:lesGysinSH+inv} for $SH_*^{+,inv}(W)$.
\hfill{$\square$}
\end{proposition}

\subsection{Filled versions of linearized contact homology}Ê\label{sec:CHlin-defi}

The isomorphism in the previous section serves as a motivation for defining filled versions of linearized contact homology, respectively of positive invariant symplectic homology. These involve closed orbits in the trivial free homotopy class and are ultimately isomorphic using $\Q$-coefficients to the $S^1$-equivariant symplectic homology groups. However, we provide in this section separate definitions in the spirit of SFT in order to emphasize the parallel between linearized contact homology and symplectic homology. 
As in the previous sections, we give the definitions under ideal transversality assumptions and prove the isomorphism with the corresponding versions of symplectic homology. 
This is motivated by the fact that $S^1$-equivariant symplectic homology is defined for all Liouville domains. 

In this section we work in the trivial free homotopy class of loops in $\hat W$. 

\subsubsection{Non-equivariant filled linearized contact homology} \label{sec:filled-nonequivariant}

As a warm-up, we start with a non-equivariant version of filled linearized contact homology. In order to define it we need to impose as transversality assumptions the condition $(B_0)$ in~\S\ref{sec:linconthom} and the following generalization of condition $(A)$: 
\begin{enumerate} \label{pagenumber:transv-bar}
\item[($\overline A$)] $J$ is regular for holomorphic planes in $\hat W$ asymptotic to some $\gamma'\in\cP^0(\alpha)$ and which belong to moduli spaces $\cM(\gamma',\emptyset;J)$ of virtual dimension $\le 2n-2$.
\end{enumerate}

For each periodic Reeb orbit $\gamma'\in\cP^0(\alpha)$, fix two generic points $\gamma'_M$, $\gamma'_m$ on its image. We view these as the minimum, respectively maximum of a perfect Morse function $f_{\gamma'}$ defined on the image of $\gamma'$. We also choose a Morse function $f:W\to \R$ whose gradient points outwards along the boundary and for which the boundary is a regular level. We extend $f$ to a Morse function defined on $\hat W$ and having no critical points on $\hat W\setminus W$. We denote a critical point of $f$ by $q$, and a critical point of $f_{\gamma'}$ by $\gamma'_p$. We also choose a generic Riemannian metric $g$ on $W$. Given this data, we define 
$$
\hat \cM(\gamma'_p,q) 
$$
to consist of pairs $(u,v)$ such that 
\begin{itemize}
\item $u:\C\to \hat W$ is a $J$-holomorphic plane asymptotic at $+\infty$ to $\gamma'$, 
\item $v:[0,\infty[\to \hat W$ is a positive gradient trajectory of $f$, i.e. $\dot v=\nabla f(v)$ asymptotic at $+\infty$ to $q$,
\item $u(0)=v(0)$,
\item $u(+\infty,0)\in W^u(\gamma'_p;\nabla f_{\gamma'})$.
\end{itemize}
We define the associated moduli space
$$
\cM(\gamma'_p,q)
$$
to be the quotient of $\hat\cM(\gamma'_p,q)$ under the $\R^+$-action given by $r\cdot (u,v)\mapsto (u(r\cdot),v)$.
The virtual dimension of $\cM(\gamma'_p,q)$ is 
\begin{eqnarray*}
\dim\,\cM(\gamma'_p,q) & = & (CZ(\gamma') + n+1) -(1 - \ind(\gamma'_p; f_{\gamma'})) - 2n + \ind(q; f) - 1 \\
& =& |\gamma'_p| +n-1 -2n +\ind(q; f).
\end{eqnarray*}
We use here the notation $|\gamma'_p|=CZ(\gamma')+\ind(\gamma'_p;f_{\gamma'})$. 

Let $BC_*(\alpha,J)=\bigoplus_{\gamma'\in\cP(\alpha)}\langle \gamma'_m,\gamma'_M\rangle$. This complex is endowed with a differential as in~\cite[\S3.2]{BOcont} and is graded by $|\gamma'_p|$. Denote $MC^*(f,g)$ the cohomological Morse complex for $f$ and $MC_{*}(-f,g)$ the homological Morse complex for $-f$. We define a linear map 
$$
\Gamma:BC_*(\alpha,J)\to MC^{-*+n+1}(f,g)=MC_{*+n-1}(-f,g)
$$
by letting it act on the generators via 
$$
\Gamma(\gamma'_p):=\sum_{q\,:\, \dim\,\cM(\gamma'_p,q)=0}\#\cM(\gamma'_p,q)q.
$$

The next Lemma is proved by examining the boundary of $1$-dimensional moduli spaces $\cM(\gamma'_p,q)$.
\begin{lemma}
$\Gamma$ is a chain map.  \hfill{$\square$}
\end{lemma} 

\begin{remark}Ê\emph{In case the almost complex structure $J$ is independent on time over the whole of $\hat W$, we have $\Gamma(\gamma'_m)=0$ for all $\gamma'\in\cP^0(\alpha)$ since moduli spaces $\cM(\gamma'_m,q)$ can never be rigid due to the $S^1$-invariance of the Cauchy-Riemann equation for holomorphic planes.}
\end{remark}

\begin{remark}[Mapping cone] \label{rmk:cone} \emph{We now fix conventions for the cone of a chain map. We work in a homological setting and our conventions are \emph{dual} to those of~\cite[Chapter~III]{Gelfand-Manin}. Given a homological chain complex $(A_\cdot,\p_A)$ and $k\in \Z$ we define $A[k]_\cdot:=(A_{\cdot+k},(-1)^k\p_A)$. Let $f:(A_\cdot,\p_A)\to (B_\cdot,\p_B)$ be a degree $0$ chain map, so that $f\p_A-\p_Bf=0$.  We define its cone to be the complex $C(f)_\cdot:= B[1]_\cdot \oplus A_\cdot =  B_{\cdot+1} \oplus A_\cdot$ with differential written in matrix form as 
$$
\left(\begin{array}{cc}  \p_{B[1]} & f \\Ê0 & \p_A \end{array}\right)=\left(\begin{array}{cc} -\p_B & f \\Ê0 & \p_A \end{array}\right).
$$
The short exact sequence of complexes 
$$
0\to B[1]_\cdot\to C(f)_\cdot\to A_\cdot\to 0
$$
gives rise to an exact triangle in homology for which the connecting homomorphism $H_\cdot(A)\to H_{\cdot-1}(B[1])$ is induced by $f$. In particular $C(f)$ is acyclic if and only if $f$ is a quasi-isomorphism. }
\end{remark}

We define the \emph{non-equivariant filled contact complex} as 
$$
\overline{BC}_*(\alpha,J,f,g):=C(\Gamma)_*
$$ 
and call the corresponding homology groups \emph{filled non-equivariant linearized contact homology groups} 
$$
\overline{NCH}_*^{lin}(\alpha,J,f,g)=H_*(\overline{BC}_*(\alpha,J,f,g)).
$$
We denote $NCH_*^{lin}(\alpha,J)=H_*(BC_*(\alpha,J))$ and call them \emph{non-equivariant linearized contact homology groups}. These groups depend \emph{a priori} on the choice of $\alpha,J,f,g$ and the proof of invariance with respect to these choices should again be a consequence of polyfold theory~\cite{Hofer-Polyfolds-survey,HWZ-polyfolds-I,HWZ-polyfolds-II,HWZ-polyfolds-III}. For readability we denote them by $\overline{NCH}_*^{lin}(M)$, respectively $NCH_*^{lin}(M)$. The exact triangle of the cone construction is the tautological triangle~\eqref{eq:taut-triangle-ncontact}
\begin{equation} \label{eq:taut-triangle-ncontact-2}
\xymatrix
@C=20pt
{
H_{*+n}(W,M)\ar[rr] & & 
\overline{NCH}_*^{lin}(M)  \ar[dl] \\ & NCH_*^{lin}(M) \ar[ul]^{[-1]}_-\Gamma  
}
\end{equation}

The next result is a direct consequence of the arguments developed in~\cite{BOcont}. 

\begin{proposition}
The tautological triangle of the cone construction~\eqref{eq:taut-triangle-ncontact-2} is isomorphic to the tautological exact triangle for symplectic homology~\eqref{eq:tautologicaltriangle}.
\end{proposition}

\begin{proof}[Sketch of proof.] The construction in~\cite{BOcont} provides an isomorphism 
$$I:BC_*(\alpha,J)\stackrel\sim \to SC_*^+(H,J)$$ 
using (a perturbation of) a Hamiltonian $H$ which depends only on the vertical coordinate $r$, which is superlinear on the symplectization $\hat W\setminus W$, and which is equal to a $C^2$-small Morse function $f$ on $W$. When written in matrix form with respect to the decomposition $SC_*^+(H,J)\oplus MC_{*+n}(-f,g)$ the differential on $SC_*(H,J)$ is lower triangular and $SC_*(H,J)$ is canonically identified with the cone $C(\tilde \Gamma)$ of a map $\tilde \Gamma:SC_*^+(H,J)\to MC_{*+n-1}(-f,g)$. Moreover, the tautological exact triangle~\eqref{eq:tautologicaltriangle} is isomorphic to the exact triangle of the cone. One then shows using standard arguments in Floer theory that the diagram 
$$
\xymatrix
{
BC_*(\alpha,J) \ar[d]_I \ar[r]^-\Gamma & MC_{*+n-1}(-f,g) \ar@{=}[d] \\
SC_*^+(H,J) \ar[r]_-{\tilde \Gamma} & MC_{*+n-1}(-f,g)
}
$$
is commutative up to homotopy, so that the cone exact triangles~\eqref{eq:taut-triangle-ncontact-2} and~\eqref{eq:tautologicaltriangle} are isomorphic.
\end{proof}

\subsubsection{Filled linearized contact homology} \label{sec:filledCHlin}

The content of this section is the following. We first define \emph{filled invariant symplectic homology groups} and \emph{filled linearized contact homology groups} by means of a cone construction, which directly implies the tautological triangle~\eqref{eq:taut-triangle-contact} and the Gysin triangle~\eqref{eq:Gysin-triangle-ncontact-filled} in the Introduction. We then prove that the morphisms involved in the cone constructions fit into diagrams that commute up to homotopy and which involve $S^1$-equivariant symplectic homology groups. This implies Theorem~\ref{thm:isomorphism} for the filled versions. 

We assume as in~\S\ref{sec:filled-nonequivariant} the existence of an almost complex structure $J$ on $\hat W$ which satisfies the transversality assumptions $(\overline A)$ and $(B_0)$. We fix a generic $C^2$-small Morse function $f$ on $W$ which admits $M$ as a regular level along which its gradient points outside $W$. Finally, we choose a generic Riemannian metric $g$ on $W$ induced by a compatible and time-independent almost complex structure. 

\noindent {\bf Notational convention.} \emph{For readability we suppress $J$ and $g$ from the notation of the various chain complexes and homology groups that appear in this section.}

We denote $(MC_*(-f),\p_{\mathrm{Morse}})$ the homological Morse complex for $-f$. This means that the generators are critical points $q$ of $-f$ graded by $|q|:=\ind(q;-f)$ and the differential $\p_{\mathrm{Morse}}$ is defined by counting rigid negative gradient lines of $-f$, i.e. solutions $v:\R\to \hat W$ of the equation $\dot v=\nabla f (v)$. We consider on $(MC_*(-f),\p_{\mathrm{Morse}})$ the trivial $S^1$-structure with $\varphi_0=\p_{\mathrm{Morse}}$ and $\varphi_j=0$, $j\ge 1$, and denote $(MC_*^{S^1}(-f),\p_{\mathrm{Morse}}^{S^1}):=(\Z[u]\otimes_\Z MC_*(-f),\mathrm{Id}\otimes \p_{\mathrm{Morse}})$ the corresponding chain complex for $S^1$-equivariant homology.

Let $H^a\in\cH_0^a$ be a Hamiltonian as in~\S\ref{sec:Hamiltformul} and assume that $H^a|_W=f$. The subcomplex $SC_*^{-,S^1}(H^a)$ is then identified with $(MC_{*+n}^{S^1}(-f),\p_{\mathrm{Morse}}^{S^1})$ and $SC_*^{S^1}(H^a)$ is isomorphic to the cone of a map 
$$
\tPhi_a:SC_*^{+,S^1}(H^a)\to MC_{*+n-1}^{S^1}(-f), \qquad \tPhi_a(u^k\otimes x)=\sum_{j=0}^k u^{k-j}\otimes \Phi_{j;a}(x).
$$

Recall from~\S\ref{sec:linconthom} that, for each $a>0$ which is not the period of a closed characteristic on $M$, we have an isomorphism $I_{H^a}:CH_*^{lin;\le a}(\alpha)\to SH_*^{+,inv}(H^a)$ induced by a chain level isomorphism $I_{H^a}:BC_*^{\le a}(\alpha)\to SC_*^+(H^a)$. We will construct in this section maps 
$$
\Psi_a:SC_*^{+,inv}(H^a)\to MC_{*+n-1}^{S^1}(-f), \qquad \Psi'_a:CC_*^{lin;\le a}(\alpha)\to MC_{*+n-1}^{S^1}(-f)
$$
such that the diagrams
\begin{equation} \label{diag:tPhi-Psi}
\xymatrix{
SC_*^{+,S^1}(H^a) \ar[r]^\Pi \ar[d]_{\tPhi_a} & SC_*^{+,inv}(H^a) \ar[d]^{\Psi_a} \\
MC_{*+n-1}^{S^1}(-f) \ar@{=}[r]_{\mathrm{Id}} & MC_{*+n-1}^{S^1}(-f)
}
\end{equation}
and
\begin{equation} \label{diag:Psi-Psi-prime}
\xymatrix{
CC_*^{lin;\le a}(\alpha) \ar[r]^{I_{H^a}} \ar[d]_{\Psi'_a} & SC_*^{+,inv}(H^a)Ê\ar[d]^{\Psi_a} \\
MC_{*+n-1}^{S^1}(-f) \ar@{=}[r]_{\mathrm{Id}} & MC_{*+n-1}^{S^1}(-f)
}
\end{equation}
commute up to homotopy, i.e. there exist $\sH_a:SC_*^{+,S^1}(H^a)\to MC_{*+n}^{S^1}(-f)$ and $\sH_a':CC_*^{lin;\le a}(\alpha)\to MC_{*+n}^{S^1}(-f)$ such that 
\begin{equation} \label{eq:tPhi-Psi}
\tPhi_a-\Psi_a\Pi=\p^{S^1}_{Morse}\sH_a-\sH_a\p^{S^1}, 
\end{equation}
\begin{equation} \label{eq:Psi-Psi-prime}
\Psi'_a-\Psi_a I_{H^a} = \p^{S^1}_{Morse}\sH_a'-\sH_a'\p.
\end{equation}
The homotopies $\sH_a$ and $\sH_a'$ determine obvious chain homotopies 
$$
C(\Psi'_a)\stackrel \sim\longrightarrow C(\Psi_a)\stackrel\sim\longleftarrow C(\tPhi_a).
$$ 
We then define \emph{filled invariant symplectic homology} $SH_*^{inv}(W)$ as a direct limit 
$$
SH_*^{inv}(W):=\lim_{a\to\infty} H_*(C(\Psi_a))
$$
with respect to generalized continuation maps as described in~\S\ref{sec:Hamiltformul}. 
We also define \emph{filled linearized contact homology} as a direct limit
$$
\overline{CH}_*^{lin}(M):=\lim_{a\to\infty} H_*(C(\Psi'_a))
$$
with respect to the canonical continuation maps given by enlarging the action window.  
Functoriality of the cone construction combined with arguments analogous to those of~\cite[Proposition~7]{BOcont} show that these groups satisfy the properties stated in the Introduction and in Theorem~\ref{thm:isomorphism}. As emphasized at the beginning of this section, these groups depend \emph{a priori} on the choices of $\alpha,J,f,g,H$; however, they are \emph{a posteriori} independent of these in view of the isomorphism with $S^1$-equivariant symplectic homology.

We now proceed with the construction of the maps $\tPhi_a$, $\Psi_a$, $\Psi'_a$, $\cH$, $\cH'$ and prove the identities~(\ref{eq:tPhi-Psi}--\ref{eq:Psi-Psi-prime}). \emph{For readability we write $H$ instead of $H^a$.}

\noindent \ref{sec:filledCHlin}(i) \emph{We first describe the components $\Phi_{j;a}:SC_*^{+,S^1}(H)\to MC_{*+n-1+2j}(-f)$, $j\ge 0$ of the map $\tPhi_a$.}
 
\noindent  We need to unfold the definitions in~\S\ref{sec:simplifying}. Recall that $SC_*^{+,S^1}(H)=\Z[u]\otimes_\Z SC_*^+(H)$. Since our Hamiltonian $H\in\cH_0^a$ is by assumption time-independent, the chain complex $SC_*^+(H)$ has to be understood in a Morse-Bott sense. 

Let us consider as a warm-up the idealized case where $H$ is a small Morse-Bott time-dependent perturbation of some $H'\in\cH_0^a$. The non-constant $1$-periodic orbits of $H$ are then of the form $\gamma_p$ with $\gamma\in\cP^+(H')$ and $p\in\mathrm{Crit}(f_\gamma)$, where $f_\gamma:S_\gamma\to\R$ is a collection of perfect Morse functions defined on the geometric images $S_\gamma$ of elements $\gamma\in \cP^+(H')$. We also fix a perfect Morse function $\varphi:\C P^\infty\to \R$ as in~\S\ref{sec:simplifying} with critical points $[Z_j]$, $j\ge 0$ of index $2j$ and denote $\tvarphi$ its lift to $S^\infty=\lim_j S^{2j+1}$. We denote $Z_j$ a distinguished point on the critical $S^1$-orbit of index $2j$ of $\tvarphi$.

Given $\og_p\in\cP(H)$ and $q\in\mathrm{Crit}(f)$ we define 
$$
\hat \cM^j(\og_p,q)
$$
to consist of pairs $(u,z)$ such that 
\begin{itemize} 
\item $u:\R\times S^1\to \widehat W$, $z:\R\to S^{2j+1}$  solve the equations 
$$\p_su+J^\theta_{z(s)}(u)(\p_\theta u-X^\theta_{H_{z(s)}}(u))=0,\qquad \dot z-\nabla \tvarphi(z)=0,$$
\item $(u,z)$ are subject to the asymptotic conditions 
$$(u(-\infty,\cdot),z(-\infty))\in S^1\cdot(\og_p,Z_j),\qquad (u(+\infty,\cdot),z(+\infty))\in S^1\cdot(q,Z_0).$$
\end{itemize} 
Here $H^\theta_z$ is an $S^1$-invariant extension of the Hamiltonian $H$ as in~\S\ref{sec:simplifying} and $J^\theta_z$ is a regular $S^1$-invariant almost complex structure. The group $\R\times S^1$ acts freely on $\hat\cM^j(\og_p,q)$ by shifts in the $s$-variable and by simultaneous rotations on $(u,z)$. We denote the quotient by 
$$
\cM^j(\og_p,q).
$$ 
In case $\dim\,\cM^j(\og_p,q)=0$ each element of the moduli space $\cM^j(\og_p,q)$ inherits a sign from a choice of coherent orientations and we denote $\#\cM^j(\og_p,q)$ the algebraic count of its elements. We then have by definition
$$
\Phi_{j;a}(\og_p):=\sum_{q\, : \, \dim\, \cM^j(\og_p,q)=0} \#\cM^j(\og_p,q)q.
$$

Let us now turn to the Morse-Bott description of the maps $\Phi_{j;a}$ for $H\in\cH_0^a$. Recall that, on the set $\hat W\setminus W=\{r\ge 1\}$, the Hamiltonian $H$ is equal to a convex function of $r$ which is linear for $r\gg 1$ and has small enough second derivative. Together with the transversality assumptions $(\overline A)$ and $(B_0)$, this guarantees that our time-independent almost complex structure $J$ is regular for the relevant moduli spaces of Floer trajectories for $H$ (see~\cite[\S5.1]{BOcont}). We further assume that $H\equiv 0$ on $W$ and we treat the Morse function $f$ as a perturbation of this degenerate situation. We also choose a generic perfect Morse function $f_\gamma:S_\gamma\to\R$ for every circle of nonconstant orbits $S_\gamma\subset \cP(H)$. Given $\og\in\cP(H)$, $p\in\mathrm{Crit}(f_\og)$, $q\in\mathrm{Crit}(f)$, we define 
$$
\hat\cN^j(\og_p,q)
$$
to consist of tuples $(u,z,v)$ such that 
\begin{itemize}
\item $u:\R\times S^1\to \hat W$, $z:\R\to S^{2j+1}$ solve the equations 
$$\p_su+J^\theta_{z(s)}(u)(\p_\theta u-X^\theta_{H_{z(s)}}(u))=0,\qquad \dot z -\nabla\tvarphi(z)=0$$ 
and are subject to the asymptotic conditions 
$$(u(-\infty,\cdot),z(-\infty))\in S^1\cdot(W^u(p),Z_j),$$ 
$$(u(+\infty,\cdot),z(+\infty))\in S^1\cdot (pt,Z_0)$$ 
for some $pt\in W$,
\item $v:[0,\infty[\to W$ solves the equation 
$$\dot v=\nabla f(v)$$ 
and is subject to the asymptotic condition 
$$v(+\infty)=q,$$
\item $u(+\infty,\cdot)\equiv pt=v(0)\in W$.
\end{itemize}
Here $(H^\theta_z,J^\theta_z)$ is a regular $S^1$-invariant pair which extends $(H,J)$. The group $\R\times S^1$ acts freely on $\hat\cN^j(\og_p,q)$ by shifts in the $s$-variable and by rotations on $(u,z)$. We denote the quotient by  
$$
\cN^j(\og_p,q).
$$
In the case of dimension zero one associates from coherent orientations a sign to each element of $\cN^j(\og_p,q)$. The space $\cN^j(\og_p,q)$ is then in one-to-one bijective correspondence preserving signs with $\cM^j(\og_p,q)$ for a Morse-Bott perturbation of $H$, see~\cite{BOauto} or~\cite{Bourgeois-thesis-Fields}. (Strictly speaking the space $\hat \cN^j(\og_p,q)$ should be defined to contain configurations of Morse-Bott broken trajectories with an arbitrary number of sublevels as in~\cite[pp.~90-91]{BOauto}, but in the case of dimension zero the only configurations that actually appear are the ones that we took into account in our definition.) In particular we have 
$$
\Phi_{j;a}(\og_p)=\sum_{q\, : \, \dim\, \cN^j(\og_p,q)=0}Ê\#\cN^j(\og_p,q)q.
$$

If we further assume that $J^\theta_z$ is a small perturbation of our regular $J$ and $H^\theta_z$ is a small perturbation of our time-independent Hamiltonian $H$, we actually have 
$$
\Phi_{j;a}=0, \quad j\ge 1,
$$
so that the map $\tPhi_a:SC_*^{+,S^1}(H)\to MC_{*+n-1}^{S^1}(-f)$ is given by 
$$
\tPhi_a(u^k\otimes x)=u^k\otimes \Phi_{0;a}(x).
$$
Indeed, the dimension formula for $\cN^j(\og_p,q)$ is 
\begin{eqnarray*}
\dim\, \cN^j(\og_p,q) & = & |\og_p|+n-|q|+2j+1 - 2 \\
& = & |\og| +n- |q| + \ind(p;f_\og) +2j-1\\ 
& \ge & 1 + 0 + 2j-1 \\
& = & 2j.
\end{eqnarray*}
Our transversality assumptions guarantee that, in case $\cN^j(\og_p,q)$ is non-empty, we have $|\og|+n-|q|\ge 1$ because pairs $(u,v)$ that are involved in $\cN^j(\og_p,q)$ are $C^\infty$-close to pairs $(u',v')$ subject to the same asymptotic conditions and solving the same equations with $(H^\theta_z,J^\theta_z)$ replaced by $(H,J)$. Thus the dimension of $\cN^j(\og_p,q)$ is $>0$ whenever $j\ge 1$ and the space is non-empty.

In the case $j=0$ the dimension is zero only if $|\og|+n-|q|=1$ and $\ind(p;f_\og)=0$. Thus $p=M$ is a maximum, the component $z$ is constant, and $u$ viewed as a map $u:\C\to\hat W$ is a perturbed pseudo-holomorphic plane that is rigidified by requiring $u(0)$ to belong to $W^s(q;\nabla f)$. We can in particular give the following alternative description of $\Phi_{0;a}$. Using the same notation as above for $\og_p$ and $q$ we define 
$$
\hat \cN_{\Phi_0}(\og_p,q) 
$$
to consist of pairs $(u,v)$ such that  
\begin{itemize}
\item $u:\R\times S^1\to \hat W$ solves $$\p_su+J(u)(\p_\theta u-X_H(u))=0$$ and $u(-\infty,\cdot)\in W^u(p)$,
\item $v:[0,\infty[\to W$ solves $$\dot v=\nabla f(v)$$ and $v(+\infty)=q$,
\item $u(+\infty,\cdot)\equiv v(0)\in W$. 
\end{itemize}
The additive group $\R$ acts freely on $\hat \cN_{\Phi_0}(\og_p,q)$ by shifts in the $s$-variable and we denote the quotient by 
$$
\cN_{\Phi_0}(\og_p,q).
$$
Then $\cN_{\Phi_0}(\og_p,q)$ is in one-to-one bijective correspondence preserving signs with $\cN^0(\og_p,q)$ and we obtain 
$$
\Phi_{0;a}(\og_p)=\sum_{q\,:\, \dim\, \cN_{\Phi_0}(\og_p,q)=0}\#\cN_{\Phi_0}(\og_p,q)q.
$$

\noindent \ref{sec:filledCHlin}(ii) \emph{We construct the map $\Psi_a:SC_*^{+,inv}(H)\to MC_{*+n-1}^{S^1}(-f)$.} 

\noindent We assume henceforth the Hamiltonian $H\in\cH_0^a$ to be as in the definition of $\hat \cN^j(\og_p,q)$, i.e. with small enough slope on $\{r\ge 1\}$ and equal to $0$ on $W$.  
Given $q\in\mathrm{Crit}(f)$ and $\gamma\in\cP(H)$ such that the underlying Reeb orbit $\gamma'$ is good, we define 
$$
\hat\cM_\Psi(S_\gamma,q)
$$
to consist of triples $(u,\theta_0,v)$ such that 
\begin{itemize}
\item $u:\R\times S^1\to\hat W$ solves $$\p_su+J(u)(\p_\theta u-X_H(u))=0$$ and $u(-\infty,\cdot)\in S_\gamma$,
\item $\theta_0\in S^1$ is such that $u(-\infty,\cdot)=\gamma(\cdot-\theta_0)$,
\item $v:[0,\infty[\to W$ solves $$\dot v=\nabla f(v)$$ and $v(+\infty)=q$,
\item $u(+\infty,\cdot)\equiv v(0)\in W$.
\end{itemize}
The group $\R\times S^1$ acts freely on $\hat\cM_\Psi(S_\gamma,q)$ by shifts in the $s$-variable on $u$ and by simultaneous rotation on $(u,\theta_0)$. We denote the quotient by
$$
\cM_\Psi(S_\gamma,q).
$$
The virtual dimension of $\cM_\Psi(S_\gamma,q)$ is 
\begin{eqnarray*}
\dim\,\cM_\Psi(S_\gamma,q) & = & |\gamma|+n+1 -\mathrm{codim}\, W^s(q;\nabla f) - 2 \\
& = & |\gamma|-|q| + n-1.
\end{eqnarray*}
In the case of dimension zero each element of $\cM_\Psi(S_\gamma,q)$ inherits a sign from coherent orientations and we denote $\#\cM_\Psi(S_\gamma,q)$ the algebraic count of its elements. We then define 
$$
\Psi_a:SC_*^{+,inv}(H)\to MC_{*+n-1}(-f)\hookrightarrow MC_{*+n-1}^{S^1}(-f)
$$
by
$$
\Psi_a(S_\gamma):=\sum_{q\, : \, \dim\, \cM_\Psi(S_\gamma,q)=0} \#\cM_\Psi(S_\gamma,q)q.
$$

\noindent \ref{sec:filledCHlin}(iii) \emph{We construct the homotopy} $\sH_a:SC_*^{+,S^1}(H)\to MC_{*+n}^{S^1}(-f)$.

\noindent As in~\ref{sec:filledCHlin}(ii) we assume $H$ to have small enough second derivative for $r\ge 1$ and to be equal to $0$ on $W$. 
We first need to set up some notation. As in~\S\ref{sec:S1-CH} we consider a smooth increasing function $\beta:\R\to [0,1]$ such that $\beta(s)=0$ for $s\le -1$ and $\beta(s)=1$ for $s\ge 0$. We also consider a smooth retract $r:[0,1]\times \cJ\to \cJ$ such that $r(0,\cdot)=\mathrm{Id}_\cJ$ and $r(1,\cdot)$ is the constant map equal to $J$. Given an $S^1$-invariant family $J^\theta_z$ we define $J^\theta_{z,\beta}:=r(\beta(s),J^\theta_z)$. Note that this is an $S^1$-invariant family for every $s\in\R$. Similarly, given an $S^1$-invariant extension $H^\theta_z$ of $H$ we denote $H^\theta_{z,\beta}:=(1-\beta(s))H^\theta_{z,\beta} + \beta(s)H$. 

For every $\gamma\in\cP(H)$, $p\in\mathrm{Crit}(f_\gamma)$ and $j\ge 0$ we fix a generic point $P_j(\gamma_p)\in S^1\cdot(\gamma_p,Z_j)$. This determines uniquely an element $\tau_j(\gamma_p)\in S^1$ by the relation $P_j(\gamma_p)=\tau_j(\gamma_p)\cdot(\gamma_p,Z_j)$. We denote $P_j(W^u(\gamma_p)):=\tau_j(\gamma_p)\cdot (W^u(\gamma_p)\times \{Z_j\})$.

Given $j\ge 0$, $\og\in\cP(H)$, $p\in\mathrm{Crit}(f_\og)$, $q\in\mathrm{Crit}(f)$, we define 
$$
\cM^j_\sH(\og_p,q)
$$
to consist of triples $(u,z,v)$ such that:
\begin{itemize}
\item $u:\R\times S^1\to \hat W$ solves $$\p_s u +J^\theta_{z,\beta}(u)(\p_\theta u - X^\theta_{H_{z,\beta}}(u))=0$$ and $u(+\infty,\cdot)$ is a constant orbit of $H$ in $W$,
\item $z:]-\infty,0]\to S^{2j+1}$ solves $$\dot z=\nabla \tvarphi(z),$$
\item $(u(-\infty,\cdot),z(-\infty))\in P_j(W^u(\og_p))$,
\item $v:[0,\infty[\to W$ solves $$\dot v=\nabla f(v)$$ and $v(+\infty)=q$,
\item $u(+\infty,\cdot)\equiv v(0)\in W$.
\end{itemize}
The virtual dimension of this moduli space is 
\begin{eqnarray*}
\dim\, \cM^j_\sH(\og_p,q) & = & |\og_p|+n-|q|+2j+1 -1 \\
& = & |\og|+n-|q|+\ind(p;f_\og) + 2j. 
\end{eqnarray*}
As usual, in case $\cM^j_\sH(\og_p,q)$ is $0$-dimensional its elements inherit a sign from coherent orientations and we denote $\#\cM^j_\sH(\og_p,q)$ their algebraic count. The map $\sH_a$ is then defined by 
$$
\sH_a(u^k\otimes \og_p):=\sum_{j=0}^k u^{k-j}\otimes \sum_{q\,:\, \dim\,\cM^j_\sH(\og_p,q)=0} \#\cM^j_\sH(\og_p,q)q.
$$
As above, under our standing transversality assumptions and for $(H^\theta_z,J^\theta_z)$ a small perturbation of $(H,J)$ we infer that $|\og|+n-|q|\ge 0$ in case $\cM^j_\sH(\og_p,q)$ is nonempty. Thus $\cM^j_\sH(\og_p,q)$ can be $0$-dimensional only if $p=M$, $j=0$ and $|\og|+n-|q|= 0$. Thus 
$$
\sH_a(u^k\otimes \og_p)=u^k\otimes \sum_{q\,:\, \dim\,\cM^0_\sH(\og_p,q)=0} \#\cM^0_\sH(\og_p,q)q.
$$

\noindent \ref{sec:filledCHlin}(iv) \emph{We prove equation~\eqref{eq:tPhi-Psi}.}

\noindent We examine the boundary of the $1$-dimensional moduli spaces $\cM^j_\sH(\og_p,q)$. There can be four types of breaking for $1$-parameter families of configurations $(u,z,v)\in\cM^j_\sH(\og_p,q)$: 

\noindent 1) breaking at $-\infty$ in the domain of $(u,z)$ on an orbit $S^1\cdot(\gamma_r,Z_i)$, with $\gamma\in\cP^+(H)$ and $r\in\mathrm{Crit}(f_\gamma)$. This gives rise to a boundary term equal to $\sH_a\circ\p^{S^1}$;

\noindent 2) breaking at $-\infty$ in the domain of $(u,z)$ on an orbit $S^1\cdot(q',Z_i)$, with $q'\in\mathrm{Crit}(f)$. We claim that this gives rise to a boundary term equal to $\tPhi_a$. Clearly, this type of breaking gives rise to a boundary term equal to a composition $\chi\circ\tPhi_a$, where 
$$\chi:MC_*^{S^1}(-f)\to MC_*^{S^1}(-f)$$ 
is described by a count of rigid elements in moduli spaces $\cM^j_\chi(q',q)$ defined as follows. Given $q',q\in\mathrm{Crit}(f)$ the moduli space $\cM^j_\chi(q',q)$ consists of pairs $(v,z)$ such that 
\begin{itemize}
\item $v:\R\to W$ solves $\dot v =\nabla f(v)$,
\item $z:]-\infty,0]\to S^{2j+1}$ solves $\dot z=\nabla \tvarphi(z)$, 
\item $(v(-\infty),z(-\infty))=P_j(q')$,
\item $v(+\infty)=q$.
\end{itemize}
Here $P_j(q')$ is a fixed generic point on the $S^1$-orbit $S^1\cdot(q',Z_j)=\{q'\}\times S^1\cdot Z_j$. The virtual dimension of this moduli space is 
$$\dim\,\cM^j_\chi(p,q)=|q'|-|q|+2j+1-1=|q'|-|q|+2j.$$ 
By transversality for the gradient flow of $f$ we infer that $|q'|-|q|\ge 0$ if the moduli space $\cM^j_\chi(q',q)$ is nonempty. Thus in order for $\cM^j_\chi(q',q)$ to be non-empty and $0$-dimensional we must necessarily have $j=0$, $q'=q$ and the pair $(v,z)$ must be constant equal to $P_0(q)$. As a consequence $\chi=\mathrm{Id}$ and $\chi\circ\tPhi_a=\tPhi_a$.

\noindent 3) breaking at $+\infty$ in the domain of $u$ on an orbit $\gamma\in\cP^+(H)$. We claim that this gives rise to a boundary term equal to $\Psi_a\circ \Pi$. Indeed, this gives rise to a boundary term equal to a composition $\Psi_a\circ \tPi$, where $\tPi:SC_*^{+,S^1}\to SC_*^{+,inv}$ is defined by a count of rigid elements in moduli spaces $\cM^j_\tPi(\og_p,S_\gamma)$ defined as follows. Given $\gamma\in\cP^+(H)$ the moduli space $\cM^j_\tPi(\og_p,S_\gamma)$ consists of pairs $(u,z)$ such that: 
\begin{itemize}
\item $u:\R\times S^1\to \hat W$ solves $$\p_s u +J^\theta_{z,\beta}(u)(\p_\theta u - X^\theta_{H_{z,\beta}}(u))=0$$ and $u(+\infty,\cdot)\in S_\gamma$,
\item $z:]-\infty,0]\to S^{2j+1}$ solves $$\dot z=\nabla \tvarphi(z),$$
\item $(u(-\infty,\cdot),z(-\infty))\in P_j(W^u(p))$.
\end{itemize}
The map $\tPi$ is then defined by 
$$
\tPi(u^k\otimes \og_p):=\sum_{j=0}^k u^{k-j}\otimes \sum_{\scriptsize\begin{array}{c}Ê\gamma \mbox{ good} \\ \dim\, \cM^j_\tPi(\og_p,S_\gamma)=0\end{array}}Ê\# \cM^j_\tPi(\og_p,S_\gamma) \langle S_\gamma \rangle.
$$
The virtual dimension of $\cM^j_\tPi(\og_p,S_\gamma)$ is 
$$\dim\,\cM^j_\tPi(\og_p,S_\gamma)=|\og_p|-|S_\gamma|+2j+1 -1=|\og|-|\gamma|+\ind(p;f_\og)+2j.$$ 
Under our standing transversality assumptions and if $(H^\theta_z,J^\theta_z)$ is a small enough perturbation of $(H,J)$, we have $|\og|-|\gamma|\ge 0$ if $\cM^j_\tPi(\og_p,S_\gamma)$ is non-empty. Thus in order for $\cM^j_\tPi(\og_p,S_\gamma)$ to be non-empty and $0$-dimensional we must necessarily have $j=0$, $\og=\gamma$, $p=M$, and the maps $u$ and $z$ must be constant with $(u,z)\equiv P_0(\gamma_M)$. Thus $\tPi=\Pi$. 

\noindent 4) breaking at $+\infty$ in the domain of $v$ on a critical point $q'\in\mathrm{Crit}(f)$. This gives rise to a boundary term equal to $\p_{Morse}\circ\sH_a=\p^{S^1}_{Morse}\circ\sH_a$.

We obtain 
$$
\Psi_a\Pi+\p^{S^1}_{Morse}\sH_a -\tPhi_a-\sH_a\p^{S^1}=0,
$$
which is equivalent to~\eqref{eq:tPhi-Psi}. \hfill{$\square$}

\noindent \ref{sec:filledCHlin}(v) \emph{We construct the maps $\Psi'_a:CC_*^{lin;\le a}(\alpha)\to MC_{*+n-1}^{S^1}(-f)$ and $\cH'_a:CC_*^{lin;\le a}(\alpha)\to MC_{*+n}^{S^1}(-f)$, and we prove~\eqref{eq:Psi-Psi-prime}.} 

\noindent Recall the notations in the definition of the complex $CC_*^{lin;\le a}(\alpha)$ and the fact that we fix a reference point on the geometric image of each closed Reeb orbit $\gamma'\in\cP(\alpha)$. Given $\gamma'\in\cP(\alpha)$ we denote by $\cM_0(\gamma',\emptyset)$ the moduli space of $J$-holomorphic planes in $\hat W$ with a marked point at $0\in\C$ and with asymptotic tangent direction at $\infty$, which are asymptotic to $\gamma'$ and converge to the reference point along the tangent direction. The moduli space $\cM_0(\gamma',\emptyset)$ has virtual dimension $\mu(\gamma')+n-1$ and carries an evaluation map $\ev_0:\cM_0(\gamma',\emptyset)\to \hat W$ at the marked point. Given $q\in\mathrm{Crit}(f)$ we denote 
$$
\cM(\gamma',q):=\cM_0(\gamma',\emptyset)\ _{\ev_0}\!\!\times W^s(q;\nabla f).
$$
The virtual dimension of $\cM(\gamma',q)$ is $\mu(\gamma')-|q|+n-1$. In the case of dimension $0$ the elements of $\cM(\gamma',q)$ inherit a sign from coherent orientations and we denote by $\#\cM(\gamma',q)$ their algebraic count. We then define 
$$
\Psi'_a:CC_*^{lin;\le a}(\alpha)\to MC_{*+n-1}(-f)\hookrightarrow MC_{*+n-1}^{S^1}(-f)
$$
by 
$$
\Psi'_a(\gamma'):=\sum_{q\, : \, \dim\, \cM(\gamma',q)=0} \#\cM(\gamma',q)q.
$$
That $\Psi'_a$ is a chain map follows by inspecting the boundary of the $1$-dimensional moduli spaces $\cM(\gamma',q)$. That the diagram~\eqref{diag:Psi-Psi-prime} is commutative up to homotopy follows from the arguments of~\cite[\S6]{BOcont} and the usual ``homotopy of homotopies'' argument in Floer theory. 

Let us make explicit this argument for completeness. Recall from~\S\ref{sec:Hamiltformul} that the map $I_H$ is defined by a count of elements in moduli spaces $\cM_I(\og',S_\ug)$, $\og'\in\cP(\alpha)$, $\ug\in\cP^+(H)$ defined in~\eqref{eq:MI}. 

Given $q\in\mathrm{Crit}(f)$ we define moduli spaces $\cM(\og',q;H^\rho)$ to consist of pairs $(u,v)$ such that:
\begin{itemize}
\item $u:\R\times S^1\to \hat W$ solves 
$$
\p_su+J(\p_\theta u - X_{H^\rho}(u))=0,
$$ 
it is asymptotic at $-\infty$ to $\og'$ and converges to the reference point of $\og'$ along the tangent direction $\R\times \{0\}$, and it is asymptotic at $+\infty$ to a constant orbit of $H$ in $W$,
\item $v:[0,\infty[\to W$ solves 
$$\dot v=\nabla f(v)$$ 
and $v(+\infty)=q$,
\item $u(+\infty,\cdot)\equiv v(0)\in W$.
\end{itemize}
The expected dimension of $\cM(\og',q;H^\rho)$ is $\mu(\og')-|q|+n-1$. The count of elements in this moduli space describes the composition $\Psi_a\circ I_H$.

Let now $H^\rho_\lambda$, $\lambda\in[0,1]$ be an increasing homotopy from $0$ to $H^\rho$. 
The moduli space 
$$
\cM^{[0,1]}_{\sH'}(\og',q):=\bigsqcup_{\lambda\in[0,1]}\{\lambda\}\times \cM(\og',q;H^\rho_\lambda)
$$
has expected dimension $\mu(\og')-|q|+n$. We define  
$$
\sH_a':CC_*^{lin;\le a}(\alpha)\to MC_{*+n}(-f)\hookrightarrow MC_{*+n}^{S^1}(-f)
$$
by 
$$
\sH_a'(\og'):=\sum_{q\,:\, \dim\, \cM^{[0,1]}_{\sH'}(\og',q)=0} \#\cM^{[0,1]}_{\sH'}(\og',q)q.
$$

The relation~\eqref{eq:Psi-Psi-prime} is established by examining the boundary of $1$-dimensional moduli spaces $\cM^{[0,1]}_{\sH'}(\og',q)$. At $\lambda=1$ we obtain $\Psi_a\circ I_H$, at $\lambda=0$ we obtain $\Psi'_a$, and inner boundary components correspond to $\p_{Morse}\circ \sH_a'- \sH_a'\circ \p$. Thus 
$$
\Psi'_a-\Psi_a \circ I_H = \p_{Morse}\circ \sH_a'- \sH_a'\circ \p.
$$

\vspace{-9mm}

\hfill{$\square$}

\smallskip

This completes the proof of the identities~(\ref{eq:tPhi-Psi}--\ref{eq:Psi-Psi-prime}), and also the proof of Theorem~\ref{thm:isomorphism} for the case of filled linearized contact homology groups.

\section{Applications, computations}Ê\label{sec:applications}

\subsection{Applications}

\subsubsection{Vanishing theorem} 

The following result was already stated by Seidel in~\cite{Seidel07}. 

\begin{theorem}[Seidel~\cite{Seidel07}] \label{thm:vanishing}
Given the symplectic manifold $(W,\omega)$, we have that $SH_*(W,\omega)=0$ if and only if $SH_*^{S^1}(W,\omega)=0$.
\end{theorem}

\begin{proof}
Assuming $SH_*(W)=0$ we obtain that $SH_*^{S^1}(W)=0$ from the spectral sequence of Theorem~\ref{thm:Gysin-spec}(ii). Conversely, if $SH_*^{S^1}(W)=0$ we obtain that $SH_*(W)=0$ from the Gysin triangle in Theorem~\ref{thm:Gysin-spec}(i). 
\end{proof}

This theorem can be used to clarify the relationship between two concepts introduced by Viterbo in~\cite{Viterbo99}. Let us say that a convex symplectic manifold $W^{2n}$ satisfies the \emph{Strong Algebraic Weinstein Conjecture (SAWC)} if the map $H_{2n}(W,\partial W)\to SH_n(W)$ vanishes. Letting $\mu_{2n}\in H_{2n}(W,\partial W)$ be the fundamental class and $u_k$ be a generator of $H_{2k}(BS^1)$, $k\ge 0$, we say that it satisfies the \emph{Strong Equivariant Algebraic Weinstein Conjecture (EWC)} if, for all $k\ge 0$, the element $u_k\otimes \mu_{2n}$ lies in the kernel of the map $H_{2n+2k}^{S^1}(W,\partial W)\to SH_{n+2k}^{S^1}(W)$. We proved the following result in~\cite{BOGysin}, and we give here an alternative proof. 

\begin{proposition}[{\cite[Corollary~1.4]{BOGysin}}]
SAWC $\Longrightarrow$ EWC.
\end{proposition}

\begin{proof} Symplectic homology has the structure of a ring with unit~\cite{Seidel07,McLean,Ritter}, the unit being the image of $1\in H^0(W)$ under the map $H^0(W)\simeq H_{2n}(W,\partial W)\to SH_n(W)$. Thus $SAWC$ is equivalent to the vanishing of the unit of $SH_*(W)$, hence to the vanishing of $SH_*(W)$ and, in view of Theorem~\ref{thm:vanishing}, to the vanishing of $SH_*^{S^1}(W)$. This clearly implies $EWC$.
\end{proof}

\subsubsection{A substitute for cylindrical contact homology} \label{sec:rigorousCHcyl}

The current definition of cylindrical/linearized contact homology groups as in~\cite{EGH,BOcont,BOcont-err} suffers from the fact that transversality cannot in general be achieved by a generic choice of time-independent cylindrical almost complex structure. In view of Theorem~\ref{thm:isomorphism} and of the fact that transversality can be achieved in an $S^1$-equivariant setting (as in~\cite{BOtransv} or as in~\S\ref{sec:S1-families}, depending on the flavor of $S^1$-equivariant homology group one wishes to consider), we propose the following substitute for cylindrical/linearized contact homology groups.

(1) {\it Fillable contact manifolds}. Let $(M^{2n-1},\xi)$ be a contact manifold which admits a symplectically atoroidal symplectic filling $W$. Assume further that $c_1(W)=0$ and that $\xi$ admits a nondegenerate contact form such that all closed Reeb orbits $\gamma'$ which are contractible in $W$ have positive SFT grading 
\begin{equation} \label{eq:mubar>0} 
\bar\mu(\gamma'):=\mu(\gamma')+n-3>0,
\end{equation}
where $\mu(\gamma')$ denotes the Conley-Zehnder index of the linearized Reeb flow on $\xi$. Our first observation is that the positive $S^1$-equivariant symplectic homology group $SH_*^{+,S^1}(W)$ is a contact invariant of $(M,\xi)$. This amounts to proving that $SH_*^{+,S^1}(W)$ does not depend on the choice of filling inducing the contact form $\alpha$, and also that $SH_*^{+,S^1}(W)$ does not depend on the choice of contact form $\alpha$ satisfying~\eqref{eq:mubar>0}. The independence with respect to the choice of filling follows from the stretch-of-the-neck argument in~\cite[\S5.2]{BOcont} and~\cite[Theorem~1.14]{Cieliebak-Frauenfelder-Oancea}. The independence with respect to the choice of contact form follows from the existence of Viterbo transfer morphisms 
$$
SH_*^{+,S^1}(W)\to SH_*^{+,S^1}(W')
$$
associated to codimension $0$ embeddings $W'\hookrightarrow W$ such that $\overline{W\setminus W'}$ is exact~\cite{Viterbo99,Gutt-thesis,Cieliebak-Oancea}.

If one removes the assumption on the grading of contractible closed Reeb orbits, the result is an invariant of $\xi$ which depends on $W$. The direct sum over all possible such fillings $W$ is then an invariant of $\xi$.

(2) {\it Contact manifolds.} Let $(M^{2n-1},\xi)$ be a contact manifold that satisfies one of the following two conditions: (i) either $c_1(\xi)=0$ and $\xi$ admits a nondegenerate contact form $\alpha$ for which all contractible closed Reeb orbits $\gamma'$ have SFT grading 
$$
\bar\mu(\gamma'):=\mu(\gamma')+n-3>1,
$$ 
(ii) or $\xi$ admits a hypertight contact form $\alpha$, i.e. a contact form with no contractible periodic Reeb orbits. Without loss of generality this contact form can then be assumed to be nondegenerate. 

One can define an invariant of $(M,\xi)$, denoted 
$$
SH_*^{+,S^1}(M,\xi)
$$ 
as follows. We consider the symplectization $(]0,\infty[\times M,d(r\alpha))$ and for each $a>0$ which is not the period of a closed Reeb orbit we consider a Hamiltonian $H^a:\,]0,\infty[\times M\to \R$ of the form $H^a(r,x)=h^a(r)$, with $h^a:\, ]0,\infty[\to \R$ a convex increasing function such that $\frac d {dr}{h^a}(r)\to 0$, $r\to 0$, $\frac d {dr}{h^a}(r)=a$ for $r\gg 0$ and $\frac {d^2} {dr^2}{h^a}(r)>0$ whenever $\frac d {dr}{h^a}(r)<a$. We also consider a cylindrical almost complex structure $J$ and a generic family of perturbations of it, still denoted by $J$, as in~\S\ref{sec:reformulate}. 

We claim that the Floer complex of $H^a$ is an $S^1$-complex as in~\S\ref{sec:reformulate}. (Here, of course, $H^a$ has to be slightly perturbed to a time-dependent Hamiltonian in order to achieve nondegeneracy of its $1$-periodic orbits.) Moreover, for $a<a'$ and $H^a<H^{a'}$, we have continuation maps $SH_*^{+,S^1}(H^a,J)\to SH_*^{+,S^1}(H^{a'},J)$ and we define $SH_*^{+,S^1}(\alpha,J)$ as the direct limit 
$$
SH_*^{+,S^1}(\alpha,J):=\lim_{a\to\infty} SH_*^{+,S^1}(H^a,J). 
$$
The group $SH_*^{+,S^1}(\alpha,J)$ is independent of $J$ and we denote it by $SH_*^{+,S^1}(\alpha)$. It is also independent of $\alpha$, and if we wish to emphasize this we shall denote it by $SH_*^{+,S^1}(M,\xi)$. This last independence statement is proved along the same lines as in the case of filled contact manifolds using transfer morphisms, so that we focus on the proof of the claim. 

To prove the claim, let us note that the relation $\p^{S^1}\circ \p^{S^1}=0$ for the $S^1$-equivariant differential $\p^{S^1}$ defined on $SC_*^{+,S^1}(H^a,J)$ can only be violated by the appearance of $J$-holomorphic buildings with more than one level (in the sense of~\cite{BEHWZ}) on the boundary of $1$-dimensional moduli spaces $\cM^N(\og,\ug;H^a,J,\rho)$ of Floer trajectories in~\S\ref{sec:reformulate}. (The situation is similar to that of cylindrical contact homology.) Under assumption~(ii) such $J$-holomorphic buildings are ruled out because the asymptotes that are common to two consecutive levels must be contractible. Under assumption~(i) such $J$-holomorphic buildings are 
ruled out using transversality for their top component as follows (see also~\cite[Theorem~1.14]{Cieliebak-Frauenfelder-Oancea}). The top component of such a $J$-holomorphic building is a punctured sphere which is asymptotic at two of its punctures to $\og,\ug$, where it solves a \emph{perturbed} Cauchy-Riemann equation. This ensures that transversality holds for the corresponding moduli spaces. On the other hand, if there would be some other (negative) puncture in the neighborhood of which the curve is holomorphic and asymptotic to some closed Reeb orbit, the index formula would show that the virtual dimension of the corresponding moduli space would be $<0$. By transversality, the moduli space has to be empty and this rules out $J$-holomorphic buildings with more than one level. The same argument applies to show that continuation morphisms are well-defined. 

Let $J$ be a cylindrical almost complex structure on the symplectization $(\R_+^*\times M, d(r\alpha))$. The following conditions play for cylindrical contact homology the role of conditions $(A)$ and $(B_c)$ in~\S\ref{sec:linconthom}. We refer to~\cite[\S1.9.2]{EGH} for the original definition of cylindrical contact homology. 

\begin{enumerate} \label{pagenumber:transv-cyl}
\item[($A^{cyl}$)] all contractible closed Reeb orbits of $\alpha$ have SFT grading 
$$\bar\mu(\gamma'):=\mu(\gamma')+n-3>1.$$
\item[($B_c^{cyl}$)]  $J$ is regular for punctured holomorphic cylinders which belong to moduli spaces of virtual dimension $\le 2$, asymptotic at $\pm\infty$ to arbitrary elements of $\cP^c(\alpha)$ and asymptotic at the punctures to elements $\gamma'\in\cP^0(\alpha)$ such that there exists a $J$-holomorphic building in the sense of~\cite[\S7.2]{BEHWZ} with exactly one positive puncture and asymptote $\gamma'$. 
\end{enumerate}

Note that condition $(A^{cyl})$ refers to $\alpha$, whereas $(B_c^{cyl})$ refers to both $\alpha$ and $J$. 

The following is a consequence of Theorem~\ref{thm:isomorphism}. 

\begin{theorem} \label{thm:contact-invariant}
Let $(M^{2n-1},\xi)$ be a contact manifold with $c_1(\xi)=0$ and admitting a nondegenerate contact form $\alpha$ and a cylindrical almost complex structure $J$ on the symplectization $\R_+^*\times M$ which satisfy conditions $(A^{cyl})$ and $(B_c^{cyl})$ for some free homotopy class $c$ of loops in $M$. 

Then cylindrical contact homology $CH_*^{cyl;c}(\alpha,J)$ is defined and we have an isomorphism with $\Q$-coefficients
$$
SH_*^{c,S^1}(\alpha,J)\simeq CH_*^{cyl;c}(\alpha,J),
$$
where $SH_*^{c,S^1}(\alpha,J)$ denotes the summand of $SH_*^{+,S^1}(\alpha,J)$ corresponding to the free homotopy class $c$.
\hfill{$\square$}
\end{theorem}

One consequence of Theorem~\ref{thm:contact-invariant} is the following general principle.  

\emph{All applications of cylindrical or linearized contact homology to contact manifolds $(M,\xi)$, with $c_1(\xi)=0$ and admitting a nondegenerate contact form whose contractible closed Reeb orbits have grading $>1$ can be reproduced -- and often can be generalized -- using $S^1$-equivariant symplectic homology.}

There is a caveat to this principle: $S^1$-equivariant symplectic homology may be more difficult to compute than cylindrical contact homology. Indeed, some of the most spectacular applications of cylindrical contact homology, like Ustilovsky's discovery of exotic contact spheres $S^{4m+1}$~\cite{Ustilovsky}, rely heavily on the fact that for some special choice of contact form all closed Reeb orbits have even degree. This implies that the differential is zero and, assuming invariance under deformations, reduces the computation of cylindrical contact homology to an index computation. Within the context of $S^1$-equivariant symplectic homology invariance comes for free, but the computation is not trivial anymore since, after considering a time-dependent perturbation of a Hamiltonian $H^a$ as above, one also creates orbits that have odd degrees. In order to recover, say, Ustilovsky's result, one either has to check that there exists an almost complex structure which satisfies condition $(B_0^{cyl})$ so that Theorem~\ref{thm:contact-invariant} applies, or one has to compute $SH_*^{S^1}$ by some different method. Gutt found such an alternative method in his thesis~\cite{Gutt-thesis}, which involves a local version of $S^1$-equivariant symplectic homology and a local-to-global spectral sequence.  

\subsubsection{Subcritical surgery exact triangle for positive $S^1$-equivariant symplectic homology} \label{sec:subcrit_surgery}

We state here a result from~\cite{Cieliebak-Oancea} (see also~\cite{Oancea-hab}). Consider the boundary $(M^{2n-1},\xi)$ of a Liouville domain $(W^{2n},\omega)$ and let $S^{k-1}\subset M^{2n-1}$ be an isotropic sphere with trivialized symplectic normal bundle. We denote by $(M',\xi')$ the result of Weinstein surgery~\cite{Weinstein} along $S$, so that $(M',\xi')$ is the boundary of a Liouville domain $(W',\omega')$ obtained from $(W,\omega)$ by attaching a symplectic handle of index $k$. 
 
The most interesting -- and difficult -- situation is that when $k=n$ \emph{(critical surgery)}. This was understood by Bourgeois-Ekholm-Eliashberg~\cite{Bourgeois-Ekholm-Eliashberg-1}, who proved under ideal transversality assumptions that the cobordism map 
$CH_*^{lin}(M')\to CH_*^{lin}(M)$ fits into an exact triangle in which the third term is a group of Legendrian contact homology flavor~\cite[\S2.8]{EGH} \cite{Ekholm-Etnyre-Sullivan} built from cyclic words in Reeb chords with endpoints on $S$. 

The case $k<n$ \emph{(subcritical surgery)} for $SH_*^{+,S^1}$ is explained in the next Theorem. 

\begin{theorem}[Subcritical surgery exact sequence~\cite{Cieliebak-Oancea}] \label{thm:surgery}
Let $(M,\xi)$ be the boundary of a Liouville
domain $W$ of dimension $2n$. Let $(M',\xi')$ be
obtained from $M$ by surgery along an isotropic sphere $S^{k-1}\subset M$
\emph{with $k<n$} and denote $W'$ the resulting Liouville filling. There is an exact triangle
\begin{equation}
\xymatrix
@C=20pt
{
SH_*^{+,S^1}(S^{2(n-k)-1}) \ar[rr] & & 
SH_*^{+,S^1}(W')  \ar[dl]^\Phi \\ & SH_*^{+,S^1}(W) \ar[ul]^{[-1]}  
}
\end{equation}
The map $\Phi$ is the canonical transfer map, and the sphere
$S^{2(n-k)-1}$ is endowed with the standard contact structure. 
\hfill{$\square$}
\end{theorem}

The two ingredients in the proof of Theorem~\ref{thm:surgery} are the functoriality of the symplectic homology exact triangle of a pair of Liouville domains~\cite{Cieliebak-Oancea}, and the invariance of symplectic homology under subcritical handle attachment~\cite{Ci02}. 
Note that $SH_*^{S^1}(B^{2(n-k)})=0$, hence $SH_*^{+,S^1}(S^{2(n-k)-1})\simeq H_{*+n-k-1}^{S^1}(B^{2(n-k)},\p B^{2(n-k)})$ and therefore
$$
SH_*^{+,S^1}(S^{2(n-k)-1})\simeq \left\{\begin{array}{rl}Ê
\mathbb{Z}, & *=n-k+1+2m,\ m\ge 0,\\
0, & \mbox{otherwise.}
\end{array}\right. 
$$

The case $k=1$, which is relevant for contact connected sums, was
proved for linearized contact homology under ideal transversality assumptions using a different method by Bourgeois and van Koert~\cite{BvK}, without the
fillability assumption on $(M,\xi)$. We also refer to Espina's thesis~\cite{Espina} for a careful treatment of subcritical surgery and grading issues for linearized contact homology, and observe that her formula~\cite[Corollary~6.3.3]{Espina} for the behavior of the mean Euler characteristic under subcritical surgery can be deduced from Theorem~\ref{thm:surgery} in conjunction with the isomorphism Theorem~\ref{thm:isomorphism}.

\subsection{Computations}

\subsubsection{Subcritical Stein manifolds.}

Mei-Lin Yau~\cite{MLYau} computed cylindrical contact homology in the trivial homotopy class for subcritical Stein fillable contact manifolds with vanishing first Chern class. 

On the other hand, we can compute $SH_*^{+,S^1}(W)$ from the vanishing of $SH_*^{S^1}(W)$, where $W$ is a subcritical Stein domain. That the non-equivariant version $SH_*(W)$ vanishes was proved by Cieliebak in~\cite{Ci02}, and his proof carries over verbatim to the $S^1$-equivariant setup (equivalently, one can use Theorem~\ref{thm:vanishing} above). 
Using the tautological exact triangle for $S^1$-equivariant symplectic homology~\eqref{eq:tautologicaltriangle}, the outcome is
$$
SH_*^{+,S^1}(W)\simeq H_{*+n-1}^{S^1}(W,\partial W) = \bigoplus_{k\ge 0} H_{*+n-1-2k}(W,\partial W).
$$
Up to a different convention for the grading, the original paper~\cite{MLYau} identified cylindrical contact homology of $\p W$ with $\bigoplus_{k\ge 0} H_{n+2k+1-*}(W)$, which is the same as the above by Poincar\'e duality with coefficients in a field. 

Note that assumption~$(B_0^{cyl})$ in~\S\ref{sec:rigorousCHcyl} is not satisfied in the situation considered by Mei-Lin Yau and her computation assumes the invariance of cylindrical contact homology. The fact that we recover the same result using $S^1$-equivariant symplectic homology is an illustration of the fact that the isomorphism between the two theories holds well-beyond the assumptions $(A^{cyl})$ and $(B_c^{cyl})$ in Theorem~\ref{thm:contact-invariant}.

\subsubsection{Cotangent bundles}

Cieliebak and Latschev~\cite[Theorem~B]{Cieliebak-Latschev} computed under ideal transversality assumptions the linearized contact homology of the unit cotangent bundle of a closed oriented spin manifold $L$ of dimension at least $4$ and identified it with $H_*(\cL L/S^1,L)$. Here $\cL L$ is the free loop space of $L$. Strictly speaking the computation is valid for free homotopy classes $c$ that satisfy assumptions~$(A)$ and $(B_c)$ and takes the form 
\begin{eqnarray*}
CH_*^{lin;c}(ST^*L) & \simeq & H_*(\cL^c L/S^1),\quad c\neq 0, \\
CH_*^{lin;0}(ST^*L) & \simeq & H_*(\cL^0 L/S^1,L),
\end{eqnarray*}
where $\cL^c L$ denotes the component of $\cL L$ corresponding to $c$. 

On the other hand, for a closed spin oriented manifold $L$ we have
$$
SH_*^{+,S^1}(ST^*L)\simeq H_*^{S^1}(\cL L,L). 
$$
This is a variation on a theorem of Viterbo stating that $SH_*^{S^1}(T^*L)\simeq H_*^{S^1}(\cL L)$, see~\cite{Viterbo-cotangent}. The spin condition has proven to be necessary by Kragh, Abouzaid and Seidel, see~\cite{AS-corrigendum} and the references therein. Whereas there are several alternative approaches to the non-equivariant version of this isomorphism~\cite{Viterbo-cotangent,AS,AS-corrigendum,SW,Abouzaid-cotangent}, the details for the $S^1$-equivariant version are not written anywhere yet. We expect that both the methods of Abbondandolo and Schwarz~\cite{AS,AS-corrigendum} and of Abouzaid~\cite{Abouzaid-cotangent} for computing $SH_*(T^*L)$ should adapt to the $S^1$-equivariant setting that we consider.

The connection to the computation by Cieliebak and Latschev is provided via, on the one hand the isomorphism of Theorem~\ref{thm:isomorphism}, and on the other hand the isomorphism $H_*^{S^1}(\cL L,L)\simeq H_*(\cL L/S^1, L)$ which holds with rational coefficients (cf. Lemma~\ref{lem:Leray}).
This example supports the point of view that we tried to emphasize throughout the paper and according to which contact homology is a quotient theory for an $S^1$-action with finite isotropy, hence isomorphic over $\Q$ to the $S^1$-equivariant theory.


\bibliographystyle{abbrv}
\bibliography{../LS}

\enddocument